\documentclass[a4paper,10pt]{article}

\usepackage{exscale,anysize,url}
\marginsize{2cm}{2cm}{2cm}{2cm}
\usepackage[centertags]{amsmath}
\usepackage{color,amscd}
\usepackage{amssymb}
\usepackage{amsthm}
\usepackage{dsfont,verbatim}
\usepackage[all]{xy}
\usepackage{tikz-cd}
\usepackage{hyperref}

\newtheorem{definition}{Definition}[section]
\newtheorem{theorem}[definition]{Theorem}
\newtheorem*{theorem*}{Theorem}
\newtheorem{proposition}[definition]{Proposition}
\newtheorem{lemma}[definition]{Lemma}
\newtheorem{corollary}[definition]{Corollary}
\newtheorem{conjecture}[definition]{Conjecture}

\newtheorem{deflemma}[definition]{Definition-Lemma}

\newtheorem{remark}[definition]{Remark}

\newcommand{\nd}{\noindent}

\newcommand{\dC}{{\mathds C}}
\newcommand{\dQ}{{\mathds Q}}
\newcommand{\dN}{{\mathds N}}

\newcommand{\dZ}{{\mathds Z}}
\newcommand{\dP}{{\mathds P}}

\newcommand{\dW}{{\mathds W}}

\newcommand{\bD}{{\mathbb D}}
\newcommand{\dD}{{\mathds D}}
\newcommand{\astdD}{{^*}\!\dD}

\newcommand{\cA}{\mathcal{A}}
\newcommand{\cB}{\mathcal{B}}
\newcommand{\cC}{\mathcal{C}}
\newcommand{\cD}{\mathcal{D}}
\newcommand{\cE}{\mathcal{E}}
\newcommand{\cF}{\mathcal{F}}
\newcommand{\cG}{\mathcal{G}}
\newcommand{\cH}{\mathcal{H}}
\newcommand{\cI}{\mathcal{I}}
\newcommand{\cJ}{\mathcal{J}}

\newcommand{\cL}{\mathcal{L}}
\newcommand{\cM}{\mathcal{M}}
\newcommand{\cN}{\mathcal{N}}
\newcommand{\cO}{\mathcal{O}}

\newcommand{\cQ}{\mathcal{Q}}
\newcommand{\cR}{\mathcal{R}}

\newcommand{\cV}{\mathcal{V}}

\newcommand{\fm}{\mathfrak{m}}

\newcommand{\D}{\displaystyle}

\DeclareMathOperator{\Hom}{\!\cH\!\textit{om}}
\DeclareMathOperator{\Ext}{\!\cE\!\textit{xt}}

\DeclareMathOperator{\RastHom}{\mathbf{R}\!\astHom}

\DeclareMathOperator{\astotimes}{{^*}\!\!\otimes}

\DeclareMathOperator{\astHom}{{^*}\!\!\Hom}

\DeclareMathOperator{\Sp}{\textup{Sp}}

\DeclareMathOperator{\im}{\textup{im}}

\DeclareMathOperator{\Gr}{\textup{Gr}}
\DeclareMathOperator{\gr}{\textup{Gr}}

\newcommand{\Supp}{\textup{Supp}}

\newcommand{\ind}{\operatorname{ind}}
\newcommand{\can}{\operatorname{can}}
\newcommand{\ord}{\operatorname{ord}}

\newcommand{\MHM}{\operatorname{MHM}}

\newcommand{\dd}{\partial}
\newcommand{\ra}{\rightarrow}

\long\def\inhibe#1\endinhibe{\relax}

\newcommand{\upartial}{\underline{\partial}}

\DeclareMathOperator{\SP}{\rm Sp}

\DeclareMathOperator{\U}{\mathcal{U}}
\DeclareMathOperator{\Sym}{\textup{Sym}}

\newcommand{\Lotimes}{\stackrel{\mathbf{L}}{\otimes}}
\newcommand{\Lastotimes}{\stackrel{\mathbf{L}}{\astotimes}}

\newcommand{\wt}{\widetilde}

\newcommand{\Psibar}{\overline{\Psi}}
\newcommand{\Phibar}{\overline{\Phi}}

\newcommand{\Der}{\cD\!\textit{er}}

\DeclareMathOperator{\Homgr}{\cH\!\textit{omgr}}

\newcommand{\wcD}{\widetilde{\cD}}
\newcommand{\wcE}{\widetilde{\cE}}
\newcommand{\wcO}{\widetilde{\cO}}
\newcommand{\wcV}{\widetilde{\cV}}
\newcommand{\wcL}{\widetilde{\cL}}
\newcommand{\wcM}{\widetilde{\cM}}
\newcommand{\wcN}{\widetilde{\cN}}
\newcommand{\wcB}{\widetilde{\cB}}
\newcommand{\wcQ}{\widetilde{\cQ}}
\newcommand{\womega}{\widetilde{\omega}}

\begin{document}
\title{Hodge ideals of free divisors}
\author{Alberto Casta\~no Dom\'{\i}nguez, Luis Narv\'{a}ez Macarro and Christian Sevenheck}

\maketitle

\begin{abstract}

We consider the Hodge filtration on the sheaf of meromorphic functions along  free divisors for which the logarithmic comparison theorem holds.
We describe the Hodge filtration steps as submodules of the order filtration on a cyclic presentation in terms of a special factor of the Bernstein-Sato polynomial of the divisor
and we conjecture a bound for the generating level of the Hodge filtration. Finally, we develop an algorithm to compute Hodge ideals of such divisors and we apply it to some examples.
\end{abstract}

\footnote{\noindent 2010 \emph{Mathematics Subject Classification.}
32C38, 14F10,  32S35, 32S40\\
The first two authors are partially supported by the ERDF, PID2020-114613GB-I00, P20\_01056 and US-1262169. The first author is also partially supported by VI PPIT US-2018-II.5. The third author is partially supported by the DFG project Se 1114/3-2. \\
The third author would like to thank the Institute of Mathematics of the University of Seville for its hospitality during his visit in 2019.}

\section{Introduction}
\label{sec:Introduction}

For a reduced divisor $D$ in a complex manifold $X$ of dimension $n$,
we consider the sheaf $\cO_X(*D)$ of meromorphic functions along $D$. It is well-known that this is a coherent and holonomic left $\cD_X$-module, which underlies a mixed Hodge module on $X$ (see \cite{Saito1, SaitoMHM} and more specifically \cite{Saitobfunc}). The latter can be constructed in a functorial way as $j_* \dQ_U^H[n]$, where $U:=X\backslash D$, $j:U\hookrightarrow X$ is the canonical embedding, and $\dQ_U^H[n]$ denotes the constant pure Hodge module on $U$.  The aim of this paper is to describe the  Hodge filtration on the mixed Hodge module $j_* \dQ_U^H[n]$ for certain divisors and to compute it explicitly in some examples. This question has some history, but has recently been reconsidered in a series of articles (see \cite{PopaMustata, MP-QDiv-I, MP-Inv, MP-QDiv-II}) by Musta\c{t}\u{a} and Popa (in the algebraic setting though),
from a birational point of view. The authors of these papers introduce the so-called \emph{Hodge ideals:} these are coherent sheaves of ideals $\cI_k(D)\subset \cO_X$ measuring the difference between the Hodge filtration $F^H_\bullet \cO(*D)$ and the pole order filtration $P_\bullet \cO_X(*D)$. The latter consists of locally free $\cO_X$-modules of rank one given by $P_k \cO_X(*D):=\cO_X((k+1)D)$. Indeed, by a classical result of Saito (\cite[Proposition 0.9]{Saitobfunc}) we always have $F^H_k \cO_X(*D) \subset P_k\cO_X(*D)$, with equality if and only if $D$ is smooth (the ``only if'' part of the latter statement is also due to Musta\c{t}\u{a} and Popa, see \cite[Theorem A]{PopaMustata}) and then one puts
$$
\cI_k(D)\cdot P_k(D) := F^H_k \cO_X(*D).
$$
It is known (see \cite[Proposition 10.1]{PopaMustata} as well as
\cite[Theorem 0.4]{SaitoOnTheHodgeFiltration}) that the zeroth Hodge ideal $\cI_0(D)$ coincides with the multiplier ideal $\cJ((1-\varepsilon)D)$. Notice that the latter, although originally defined either analytically or via birational methods (see, e.g., \cite[Section 9]{Lazarsfeld2} and the references given therein) was already known to have a description via $\cD$-modules, see \cite[Theorem 0.1]{SaitoBudur}. For $X$ projective, there is the celebrated Nadel vanishing theorem for multiplier ideals, and one looks for similar statements for the higher Hodge ideals (see \cite[\S\S~G and H]{PopaMustata} and also \cite{Dutta}); these have applications e.g. if $X=\dP^n$ or if $X$ is an abelian variety.

The known results on Hodge ideals are mostly either global in nature, or concern the case of isolated singularities, see e.g. \cite{SaitoJungKimYoon, Zhang}. In this paper, we are interested in a specific class of divisors with highly non-isolated singular loci (namely, we will have $\textup{codim}_D(D_{\textup{sing}})=1$, in particular, these divisors are not normal). These are the so-called \emph{free divisors}, introduced and first studied by K.~Saito almost 40 years ago (see \cite{KS1}). They often appear as discriminants in a generalized sense, e.g., discriminants of singularities of maps (e.g. isolated hypersurface or complete intersection singularities, see, e.g. \cite{Looijenga},
or of reduced space curves \cite{DucoSpaceCurve}) or discriminants in quiver representation spaces (see, e.g., \cite{GMNS}). Another important class of free divisors are free hyperplane arrangements (see, e.g., \cite[Chapter 8]{DimcaHypArrangementsBook}); we will perform some computations for Hodge ideals for low-dimensional examples of free arrangements in section \ref{sec:Examples} below.

Although our methods are adapted to the special situation of free divisors, we believe that they have the potential to give insights into the structure of the Hodge filtration steps on $\cO_X(*D)$ in  more general situations. As an example, we consider cases of divisors that are not free, but close to it at the very end of the paper. More generally, it is clear that the Hodge filtration on $\cO_X(*D)$ can always be accessed using the approach of our paper, that is, by some non-trivial use of the $V$-filtration along this divisor (see below), however, the concrete calculation of this $V$-filtration might sometimes be difficult.
A related, though different approach, also using the $V$-filtration is found in \cite{MP-Inv, MP-QDiv-II}.
In concrete classes of examples, such as hyperplane arrangements, we hope that our results, and more generally the study of Hodge ideals will be useful for classical questions, like Terao's conjecture, or the homological characterization of freeness (\cite{UliEtAlLocCohom}).

Let us briefly recall the definition of free divisors. For a complex manifold $X$, we denote by $\Omega^p_X$ resp. by $\Theta_X$
the locally free $\cO_X$-module of holomorphic $p$-forms resp. the locally free $\cO_X$-module of tangent vector fields or of $\dC$-derivations of $\cO_X$. If $D\subset X$ is a reduced divisor,
then we write $\Omega^1_X(\log\, D)$ resp. $\Theta_X(-\log\,D)$ for the sheaf of logarithmic one-forms resp. of logarithmic vector fields on $X$, that is:
$$
\Omega^1_X(\log\, D):=\left\{\alpha\in\Omega^1_X(D)\,|\,d\alpha\in\Omega^2_X(D)\right\},
\quad\quad\quad\quad
\Theta_X(-\log\,D):=\left\{\theta\in\Theta_X\,|\,\theta(\cI(D))\subset\cI(D)\right\},
$$
where $\cI(D)\subset\cO_X$ is the ideal sheaf of $D$. These are coherent and reflexive $\cO_X$-modules. The examples of divisors that we are studying in this paper are given by the following condition.
\begin{definition}[see \cite{KS1}]
A divisor $D\subset X$ is called free if the sheaf $\Omega^1_X(\log\,D)$ (or, equivalently, the sheaf $\Theta_X(-\log\,D)$) is a locally free $\cO_X$-module.
\end{definition}

If $D$ is free, then we have the equality
$$
\bigwedge^p\Omega^1_X(\log\,D)\cong \Omega^p_X(\log\,D) := \left\{\alpha\in \Omega^p_X(D)\,|\,d\alpha\in \Omega^{p+1}_X(D)\right\}.
$$
Hence the terms of the so-called \emph{logarithmic de Rham complex} of $(X,D)$, i.e., the complex  $(\Omega^\bullet_X(\log\,D),d)$, are locally free $\cO_X$-modules. Notice that the most basic example of a free divisor is a divisor with simple normal crossings, in which case the logarithmic de Rham complex is a well studied object. It is particularly useful for the construction (due to Deligne, see, e.g., \cite[II.4]{PeSt}) of a mixed Hodge structure on the cohomology of the complement $U=X\backslash D$, in case it is quasi-projective. If $D$ has simple normal crossings, then it is classical that the logarithmic de Rham complex computes the cohomology of $U$, in other words, we have a quasi-isomorphism $Rj_* \dC_U \cong (\Omega^\bullet_X(\log\,D),d)$. This is not always true for any free divisor, but if it is, we say that the \emph{logarithmic comparison theorem holds} for $D$ (see \cite{CaNaMo_LCT}). This is in particular the case under a condition called \emph{strongly Koszul} (see \cite[Corollary 4.5]{nar_symmetry_BS}), which we recall in the next section (see Definition \ref{defi:SK} below). Strongly Koszul free divisors are the objects of study of this paper. Many interesting free divisors, such as free hyperplane arrangements, or more generally locally quasi-homogeneous free divisors satisfy the strong Koszul hypothesis.
A nice feature of divisors in this class is that we have a natural isomorphism $\cD_X \otimes_{\cV_X^D} \cO_X(D) \cong \cO_X(*D)$, where $\cV_X^D$ is the sheaf of logarithmic differential operators with respect to $D$, from which we obtain an explicit representation of $\cO_X(*D)\cong \cD_X/\cI$, where $\cI$ is a left ideal in $\cD_X$. As a consequence, we can consider another filtration (besides the Hodge filtration) $F_\bullet^{\ord}\cO_X(*D)$, called the order filtration, which is simply given by
the image of $F_\bullet \cD_X$ (the filtration by order of differential operators) under the isomorphism $\cD_X \otimes_{\cV_X^D} \cO_X(D) \stackrel{\cong}{\rightarrow} \cO_X(*D)$ (or, equivalently, is induced from $F_\bullet \cD_X$ when writing $\cO_X(*D)\cong \cD_X/\cI$ for the ideal $\cI$ mentioned above).

The main tool to describe the Hodge filtration $F^H_\bullet \cO_X(*D)$ is  to look at the graph embedding $i_h:X\hookrightarrow \dC_t\times X$, where $h$ is a local defining equation of $D$, and to consider $i_{h,+} \cO_X(*D)$. It also underlies a mixed Hodge module (on $\dC_t\times X$), and it is well-known that the Hodge filtration on $\cO_X(*D)$ can be deduced from the one on $i_{h,+} \cO_X(*D)$ (up to a shift by one), and vice versa. More precisely,
recall that
$i_{h,+}\cO_X(*D) \cong \cO_X(*D)[\partial_t]$ (we recall the $\cD_{\dC_t\times X}$-module structure on
$\cO_X(*D)[\partial_t]$ in Formula \eqref{eq:LeftActionGraph} on page \pageref{page:RecallStructure} below), and that under this isomorphism, we have
\begin{equation}\label{eq:IntroHodgeNhMh}
F_\bullet^H \cO_X(*D) = \left(F_{\bullet+1}^H i_{h,+} \cO_X(*D)\right) \cap (\cO_X(*D)\otimes 1).
\end{equation}
Hence we are reduced to determine $F^H_\bullet i_{h,+} \cO_X(*D)$. In order to do so, we  use a key property of mixed Hodge modules, which is known as \emph{strict specializability}. It can be rephrased as a formula (see \cite[Proposition 4.2]{Saitobfunc}) describing the Hodge filtration on a module which is the extension of its restriction outside a smooth divisor (which is the hyperplane $\{t=0\}\subset \dC_t\times X$ in our case). In order to use it, we have to
compute some steps of the canonical $V$-filtration along $\{t=0\}$ of the module $i_{h,+}\cO_X(*D)$, denoted by
$V^\bullet_{\can} i_{h,+}\cO_X(*D)$. Here we rely crucially on a previous result of the second named author
(\cite[Theorem 4.1]{nar_symmetry_BS}), namely, that the roots of the Bernstein-Sato polynomial
$b_h(s)$ of $h$ are contained in $(-2,0)$ and that they are symmetric around $-1$. As we will see below, the set of roots of $b_h(s)$ bigger than $-1$ plays a particular role, and we put
\begin{equation}\label{eq:RootsMinusOneIntro}
B'_h:=\left\{\alpha_i\in \dQ \cap (0,1)\,|\, b_h(\alpha_i-1)=0\right\}.
\end{equation}
for later reference.
For an element $\alpha_i\in B_h'$, we write $l_i$ for the multiplicity of the root $\alpha_i-1$ in $b_h$. The polynomial $\prod_{\alpha_i\in B_h'}(s-\alpha_i+1)^{l_i}\in\dC[s]$ is the special factor of the Bernstein-Sato polynomial $b_h(s)$
mentioned in the abstract.

We can consider another $V$-filtration on the module $i_{h,+}\cO_X(*D)$, induced from $V^\bullet \cD_{\dC_t\times X}$ for a cyclic presentation $\cD_{\dC_t\times X}/\cI'$ of $i_{h,+}\cO_X(*D)$ obtained from the cyclic presentation $\cD_X/\cI$ of $\cO_X(*D)$ mentioned above. We will denote it by $V^\bullet_{\ind} i_{h,+}\cO_X(*D)$. Knowing that $b_h(s)$ is closely related to the $b$-function of $V^\bullet_{\ind} i_{h,+}\cO_X(*D)$ (see Lemma \ref{lem:RootsElement1GraphEmbed} below), we can describe $V^k_{\can} i_{h,+}\cO_X(*D)$ at least for integer values of $k$ as a submodule of $V^k_{\ind} i_{h,+}\cO_X(*D)$.

Our main result can then be summarized as follows:
\begin{theorem*}(see Proposition \ref{prop:PreciseDescrCanVFilt}  and Theorem \ref{thm:HodgeOnMeromorphic} below)
Let $D\subset X$ be a strongly Koszul free divisor, and suppose that it is globally given by a reduced equation $h\in \cO_X$. Then:
\begin{enumerate}
    \item We have the following inclusions of coherent $\cO_X$-modules:
    $$
    F^H_\bullet \cO_X(*D) \subset F^{\ord}_\bullet \cO_X(*D) \subset P_\bullet \cO_X(*D).
    $$
    \item
    The zeroth step of the canonical $V$-filtration on $i_{h,+}\cO_X(*D)$ can be described as follows:
    $$
    V^0_{\can} i_{h,+} \cO_X(*D) \cong
    \D V_{\ind}^1 i_{h,+} \cO_X(*D) + \prod_{\alpha_i\in B'_h} (\partial_t t +\alpha_i)^{l_i} V_{\ind}^0 i_{h,+} \cO_X(*D).
    $$
    \item
    For all $k\in \dZ$, we have the following recursive formula for the Hodge filtration on $\cO_X(*D)$ (recall the shift convention between $F_\bullet^H\cO_X(*D)$ and $F_\bullet^H i_{h,+}\cO_X(*D)$):
    \begin{equation}\label{eq:MainFormulaIntro}
    F_k^H \cO_X(*D) \cong
    \left[
    \partial_t F^H_k i_{h,+}\cO_X(*D)+V^0_{\can} i_{h,+}\cO_X(*D)
    \right]
    \cap
    \left( F_k^{\ord}\cO_X(*D)\otimes 1\right).
    \end{equation}
\end{enumerate}
\end{theorem*}
Notice that part 2 actually holds under weaker assumptions on $D$, see Proposition \ref{prop:PreciseDescrCanVFilt} below for more details.

Combining these three results and taking into account equation \eqref{eq:IntroHodgeNhMh} allow us to determine the Hodge ideals of $D$.
Theorem \ref{theo:reduction_to_D_X[s]} below (and specifically Formula
\eqref{eq:recur-FHNh}) gives a concrete and explicit way to calculate the Hodge filtration steps on $i_{h,+}\cO_X(*D)$ resp. on $\cO_X(*D)$ and the Hodge ideals of $D$. In the second part of section \ref{sec:Examples}, we  apply our method to certain interesting examples.

In applications, it is sometimes useful to study the behaviour of a Hodge module under the duality functor. It is defined for objects of the category $\MHM$, but it restricts to the usual holonomic dual of the underlying $\cD$-module.
We give (see Theorem \ref{theo:DualFiltration}) some statements estimating the Hodge filtration on the dual Hodge module with underlying $\cD_X$-module $\dD\cO_X(*D)$, following standard convention, this $\cD_X$-module is denoted by $\cO_X(!D)$. Interestingly, our results on the dual Hodge filtration also involve the cardinality of the set $B'_h$, i.e. the number of roots (counted with multiplicity) of $b_h$ lying strictly in the interval $(-1,0)$. Finally, we state a conjecture (see Conjecture \ref{con:GenLevel}) estimating the \emph{generating level} of the Hodge filtration $F^H_\bullet \cO_X(*D)$, namely, we expect that it is always generated at level $|B'_h|$. Computation of examples in section \ref{sec:Examples} supports this conjecture.

\medskip
The paper is organized as follows: In section \ref{sec:FFilt}, we first introduce rigorously the class of divisors we are interested in, and define the order filtration on $\cO_X(*D)$ globally. We prove an important fact for the filtered module $(\cO_X(*D),F_\bullet^{\ord})$, which is known as the Cohen-Macaulay property (see Proposition \ref{prop:StrictnessDualOrder}). It is known from Saito's theory that the same property holds for $(\cO_X(*D),F_\bullet^H)$, i.e., the filtered module underlying the mixed Hodge module $j_* \dQ_U^H[n]$. We are using the Cohen-Macaulay property of both filtrations to give the estimation of the Hodge filtration on the dual module of $\cO_X(*D)$ (see Theorem \ref{theo:DualFiltration}) referred to above.

In section \ref{sec:VFilt}, we recall some general facts on $V$-filtrations and then describe (see Proposition \ref{prop:PreciseDescrCanVFilt}) the canonical $V$-filtration on the graph embedding module of $\cO_X(*D)$. In the subsequent section \ref{sec:HodgeFiltration}, we use these results to prove formula \eqref{eq:MainFormulaIntro} (see Theorem \ref{thm:HodgeOnMeromorphic}). Finally, in section \ref{sec:Examples}, we develop techniques for the computation of Hodge ideals, and perform them for some significant examples.

\textbf{Acknowledgements:} We would like to thank the anonymous referee for valuable suggestions to improve the readability, as well as for pointing out
a mistake in the proof of Corollary \ref{cor:Involutive_V-Filt} in a former version of this paper.

\section{Filtration by order on $\cO_X(*D)$}
\label{sec:FFilt}

The purpose of this section is to introduce the class of divisors we are going to study in this paper, these are free divisors
satisfying the strong Koszul hypothesis. In that case, we can define a particular good filtration $F_\bullet^{\ord}$ on the sheaf $\cO_X(*D)$. We call it \emph{order filtration}, because if locally we choose a reduced equation $h$ for $D$, then the strong Koszul hypothesis implies that there is a canonical representation of $\cO_X(*D)$ as a cyclic left $\cD_X$-module (generated by $h
^{-1}$) $\cD_X/\cI(h)$, and then our order filtration is the filtration on $\cD_X/\cI(h)$ induced by the filtration on $\cD_X$ by the order of differential operators. Nevertheless, as we will see below,
$F_\bullet^{\ord} \cO_X(*D)$ is globally well defined.
As mentioned in the introduction, one of our main results
is that for each $k\in \dZ$ the Hodge filtration $F_k^H \cO_X(*D)$ is
a coherent $\cO_X$-submodule of $F_k^{\ord}\cO_X(*D)$ as well as a precise description of the inclusion
$F_k^H \cO_X(*D) \subset F_k^{\ord}\cO_X(*D)$ (see Theorem \ref{thm:HodgeOnMeromorphic} below). We will also show in this section that the order filtration $F^{\ord}_\bullet \cO_X(*D)$ shares a key feature with $F_\bullet^H \cO_X(*D)$, known as the
\emph{Cohen-Macaulay property}. This is used in section \ref{sec:HodgeFiltration} (see the proof of Theorem \ref{theo:DualFiltration}) to give some statement about the dual Hodge filtration.

For the remainder of this paper (except in subsection \ref{subsec:Whitney}, where we explicitly relax these assumptions), we will assume $X$ to be an $n$-dimensional complex manifold, and $D\subset X$ a free divisor. We will be specifically working with those free divisors
satisfying an additional hypothesis called \emph{strongly Koszul},
that we define now.

Denote as before by $\Theta_X$ the locally free $\cO_X$-module of rank $n$ of vector fields and by $\cD_X$ the sheaf of linear differential operators with holomorphic coefficients, endowed with the filtration $F_\bullet \cD_X$ by the order of differential operators. For each divisor $D\subset X$ we write $\cO_X(*D)$ for the sheaf of meromorphic functions with poles along $D$ and $\Omega_X^\bullet(*D)$ the meromorphic de Rham complex. If $D$ is a free divisor, we denote by
$\Theta_X(-\log D) \subset \Theta_X$ the locally free $\cO_X$-module of rank $n$ of logarithmic vector fields and by $\Omega_X^\bullet(\log\,D)$   the logarithmic de Rham complex. Moreover, we let
$$
\cV_X^D:=
\left\{
P\in \cD_X\,|\, P(\cI(D)^k)\subset \cI(D)^k\ \forall k\in\dZ
\right\} \subset \cD_X
$$
be the sheaf of rings of logarithmic differential operators with respect to $D$. It is filtered by the order of differential operators as well, namely $F_k \cV_X^D := \cV_X^D \cap F_k \cD_X$ for all $k\in \dZ$, and there is a canonical isomorphism of graded $\cO_X$-algebras (see \cite[Remark 2.1.5]{calde_ens})
$$
\Sym^\bullet_{\cO_X}\Theta_X(-\log D) \xrightarrow{\sim} \gr_\bullet^F \cV_X^D.
$$
The sheaf $\cO_X(D)\subset \cO_X(*D)$ of meromorphic functions with poles of order $\leq 1$ is a left $\cV_X^D$-module by the definition of $\cV_X^D$ and we have a canonical $\cD_X$-linear map (\cite[\S 4]{calde_nar_fou})
\begin{equation} \label{eq:LCT-map}
    \cD_X \Lotimes_{\cV_X^D} \cO_X(D) \longrightarrow \cO_X(*D).
\end{equation}
We quote the following result (see  \cite[Corollaire 4.2]{calde_nar_fou}):
\begin{theorem}\label{theo:LCT}
Let $D\subset X$ be a free divisor.
The following properties are equivalent:
\begin{enumerate}
    \item[(i)] The logarithmic comparison theorem (LCT) holds for $D$ (i.e. the inclusion $\Omega_X^\bullet(\log\,D) \hookrightarrow \Omega_X^\bullet(*D)$ is a quasi-isomorphism of complexes of sheaves of $\dC$-vector spaces).
    \item[(ii)] The map (\ref{eq:LCT-map}) is a quasi-isomorphism of complexes of left $\cD_X$-modules.
\end{enumerate}
\end{theorem}
Let us notice that property (ii) in the above theorem means that the complex $\cD_X \Lotimes_{\cV_X^D} \cO_X(D)$ is concentrated in degree $0$ and the canonical $\cD_X$-linear map
$$
\alpha:\cD_X \otimes_{\cV_X^D} \cO_X(D) \longrightarrow \cO_X(*D)
$$
is an isomorphism of left $\cD_X$-modules.

\begin{definition} \label{def:F-ord}
Assume that $D\subset X$ is free and that the logarithmic comparison theorem holds for $D$. We define the {\em order filtration} $F_\bullet^{\ord} \cO_X(*D)$ by
$$
F_\bullet^{\ord} \cO_X(*D) := \alpha\left( \left( F_\bullet\cD_X\right) \otimes_{\cV_X^D} \cO_X(D)\right) \subset \cO_X(*D),
$$
which by definition is a filtration by $\cO_X$-coherent submodules.
\end{definition}

Let us give a more explicit local description. Let $p\in D$, and take a basis $\delta_1,\dots,\delta_n$  of $\Theta_X(-\log D)_p$ as well as a local equation $h\in \cO_{X,p}$ of $D$ around $p$, such that $\delta_i(h)=\alpha_i h$. Assume that the LCT holds for $D$, then we have the following explicit local presentation
\begin{equation}\label{eq:CyclicPresent-1}
\cD_{X,p}/\cD_{X,p}(\delta_1+\alpha_1,\dots,\delta_n+\alpha_n) \cong \cO_{X,p}(*D),
\end{equation}
where the class $[1]$ is send to $h^{-1}$ (see the explanation after Corollaire 4.2 in \cite{calde_nar_fou}).
Under this isomorphism, the filtration $F_\bullet^{\ord} \cO_X(*D)_p$ is the induced filtration on $ \cD_{X,p}/\cD_{X,p}(\delta_1+\alpha_1,\dots,\delta_n+\alpha_n)$ by $F_\bullet \cD_{X,p}$, which explains its name.
\medskip

We now introduce the strong Koszul hypothesis. It implies that the LCT holds for $D\subset X$, but it is stronger. Its additional assumptions will be needed in the next section.
\begin{definition} \label{defi:SK}
Let $D\subset X$ be a free divisor and let $p\in D$. Let $h\in \cO_{X,p}$ be a local reduced equation of $(D,p)$ and
let $\delta_1,\ldots, \delta_n$ be any $\cO_{X,p}$-basis of $\Theta_{X,p}(-\log\,D)$ with $ \delta_i(h) = \alpha_i h$.
\begin{enumerate}
\item
We say that $D$ is \emph{Koszul} at $p\in D$ if the sequence
$$ \sigma(\delta_1),\ldots, \sigma(\delta_n)
$$
is regular in $\Gr^F_\bullet \cD_{X,p}$. (Here and at later occasions, for an element $P$ in a filtered ring, we denote by
$\sigma(P)$ its symbol, i.e. its class in the associated graded ring.)
\item
We say that $D$ is \emph{strongly Koszul} at $p\in D$ if
the sequence
$$ h, \sigma(\delta_1)-\alpha_1 s,\dots, \sigma(\delta_n)-\alpha_n s
$$
is regular in $\Gr^T_\bullet \cD_{X,p}[s]$ (where $T_\bullet \cD_{X,p}[s]$ is the filtration on $\cD_{X,p}[s]$ for which vector fields on $X$ as well as the variable $s$ have order $1$, we will call it the total order filtration on $\cD_{X,p}[s]$).
\end{enumerate}
We say that $D$ is Koszul resp. strongly Koszul if it is so at any $p\in D$. Since the strong Koszul assumption will be our main hypothesis later, we will sometimes abbreviate it by saying that $D$ is an SK-free divisor.
\end{definition}

Let us notice that, in the above definition, the ordering of the sequences is not relevant because their elements are homogeneous. Let us also notice that if $D$ has an Euler local equation $h\in \cO_{X,p}$, i.e. $h$ belongs to its gradient ideal (with respect to some local coordinates), then there is a basis $\delta_1,\dots,\delta_{n-1},\chi$ of $\Theta_X(-\log D)_p$ with $\delta_i(h)=0$ and $\chi(h)=h$. In such a case, $D$ is strongly Koszul at $p$ if and only if $ h, \sigma(\delta_1),\dots, \sigma(\delta_{n-1})$ is a regular sequence in $\Gr
^F_\bullet\cD_{X,p}$ (compare with Definition 7.1 and Proposition 7.2 of \cite{granger_schulze_rims_2010} in the case of linear free divisors). For $D$ to be strongly Koszul is equivalent to be of ``linear Jacobian type'', i.e. the ideal $(h,h'_{x_1},\dots,h'_{x_1})$ is of linear type, and any strongly Koszul free divisor is Koszul (see Propositions (1.11) and (1.14) of \cite{nar_symmetry_BS}).

Any plane curve is a Koszul free divisor.
Any locally quasi-homogeneous free divisor is strongly Koszul \cite[Theorem 5.6]{calde_nar_fou}, and so are
free hyperplane arrangements or discriminants of stable maps in Mather’s ``nice dimensions''. In particular, any normal crossing divisor is strongly Koszul free.

Let us recall the following properties of SK-free divisors.

\begin{proposition}\label{prop:SummarySKLCT}
If $D$ is a strongly Koszul free divisor, then we have:
\begin{enumerate}
  \item $D$ is locally strongly Euler homogeneous, that is, for each $p\in D$ there is a vector field $\chi\in \Theta_{X,p}$ which vanishes at $p$
  and such that $\chi(h)=h$ for some reduced local equation $h$ of $D$ at $p$.
  \item The LCT holds for $D$, in particular, we have the local cyclic presentation of $\cO_X(*D)$ from formula \eqref{eq:CyclicPresent-1} above. More specifically,
  if for each $p\in D$, we take $h$ and $\chi$ as in point 1 and if we let $\delta_1,\dots,\delta_{n-1}$ to be a basis of germs at $p$ of vector fields vanishing on $h$, in such a way that $ \delta_1,\dots,\delta_{n-1},\chi$ is a basis of $\Theta_X(-log\,D)_p$, then there is
  an isomorphism of left $\cD_{X,p}$-modules
  \begin{equation}\label{eq:CyclicPresent-2}
    \frac{\cD_{X,p}}{\cD_{X,p}\left(\delta_1,\ldots,\delta_{n-1},\chi+1\right)} \cong   \cO_{X,p}(*D),
  \end{equation}
  sending the class of $1\in \cD_{X,p}$ to $h^{-1}\in\cO_{X,p}(*D)$.
  \item Let $p\in D$ be a point, $h\in \cO_{X,p}$ be a reduced local equation of $D$ at $p$ and $b_h(s)$ the Bernstein-Sato polynomial of $h$ (i.e., the monic generator of the ideal
  of polynomials $b(s)\in \dC[s]$ satisfying $P(s) h^{s+1} = b(s) h^s$ for some $P(s)\in \cD_{X,p}[s]$). Then we have $b_h(-s-2) = \pm b_h(s)$. In particular, if $\{\alpha_1,\ldots,\alpha_k\} \subset \dQ$ is the set of roots of $b_h(s)$,
  with $\alpha_1\leq\ldots\leq \alpha_k$, we have $\alpha_i\in (-2,0)$ and
  $\alpha_i+\alpha_{k+1-i}=-2$ for all $i$.
\end{enumerate}
\end{proposition}
\begin{proof}
Part 1 is a consequence of Propositions (1.9) and (1.11) of \cite{nar_symmetry_BS}. Part 2 is a consequence of Corollary (4.5) of \cite{nar_symmetry_BS}. Part 3 is a consequence of Corollary (4.2) of \cite{nar_symmetry_BS}.
\end{proof}

\begin{remark} Definition \ref{defi:SK} is purely algebraic and makes sense in rings other than analytic local rings $\cO_{X,p}$, for instance in polynomial rings $R:=\dC[x_1,\dots,x_n] $, algebraic local rings of polynomial rings at maximal ideals or formal power series rings.
Under this scope, if $h\in \dC[x_1,\dots,x_n] $ is a non-constant reduced polynomial, $D= \cV(h) $ is the affine algebraic  hypersurface with equation $h=0$ and $D^{an}\subset \dC^n$ is the corresponding analytic hypersurface, we know that $D$ is a (algebraic) free divisor if and only if $D^{an}$ is a (analytic) free divisor, and if so, the following properties are equivalent:
\begin{enumerate}
    \item[(a)] The (affine algebraic) divisor $D$ is strongly Koszul, i.e. for some (and hence for any) basis $\delta_1,\dots,\delta_n$ of $\Der(-\log D) = \{\delta\in \Der_\dC(R)\ |\ \delta(h) \in (h)\}$, the sequence
    $$ h, \sigma(\delta_1)-a_1 s,\ldots, \sigma(\delta_n)-a_n s,\quad \text{with}\quad \delta_i(h) = a_i h,\quad a_i\in R,
$$
is regular in $\Gr^T_\bullet \dW[s]$, where $\dW=R\langle \partial_1,\dots,\partial_n\rangle$ is the Weyl algebra and $\Gr^T_\bullet\dW[s]= R[\sigma(\partial_1),\dots, \sigma(\partial_n),s]$ is the graded ring of $\dW[s]$ with respect to the total order filtration.
\item[(b)] For each maximal ideal $\fm\subset R$ containing $h$, $D$ is strongly Koszul at $\fm$, i.e. for some (and hence for any) basis $\delta_1,\dots,\delta_n$ of $\Der(-\log D)_\fm = \Der(-\log D_\fm) = \{\delta\in \Der_\dC(R_\fm)\ |\ \delta(h) \in (h)\}$, the sequence
    $$ h, \sigma(\delta_1)-a_1 s,\ldots, \sigma(\delta_n)-a_n s,\quad \text{with}\quad \delta_i(h) = a_i h, \quad a_i\in R_\fm,
$$
is regular in $\Gr^T_\bullet \dW_\fm[s]$, where $\dW_\fm=R_\fm\langle \partial_1,\dots,\partial_n\rangle$, and $\Gr^T_\bullet\dW_\fm[s]= R_\fm[\sigma(\partial_1),\dots, \sigma(\partial_n),s]$ is the graded ring of $\dW_\fm[s]$ with respect to the total order filtration.
\item[(c)] The (analytic) divisor $D^{an}$  is strongly Koszul (in the sense of Definition \ref{defi:SK} above).
\end{enumerate}
\end{remark}

The next step is to discuss some deeper properties of the
order filtration $F^{\ord}_\bullet \cO_X(*D)$. The main result is Proposition \ref{prop:StrictnessDualOrder} below, which is concerned with the dual
\emph{filtered} module $\dD(\cO_X(*D), F^{\ord}_\bullet)$. In order to do this, we need to recall a few facts about the duality theory of filtered modules. The original reference is \cite[Section 2.4]{Saito1}, but for what follows below, \cite[\S~29]{Sch14} provides enough background information.

For any complex manifold $X$, we consider the (sheaf of) Rees ring(s) $\cR_F \cD_X$ of the filtered ring $(\cD_X,F_\bullet)$, that is $\cR_F \cD_X:=\oplus_{k\in \dZ}F_k\cD_X\cdot z^k \subset \cD_X[z]$. In local coordinates $x_1,\ldots,x_n$ on $U\subset X$,
we have
\begin{equation}\label{eq:DefRU}
R_U=
\Gamma(U,\cR_F \cD_X)=\cO_U[z]\langle z\partial_{x_1},\ldots, z\partial_{x_n} \rangle \subset
\Gamma(U,\cD_X[z]) = \cO_U[z]\langle \partial_{x_1},\ldots, \partial_{x_n} \rangle.
\end{equation}
The ring $\cR_F \cD_X$ will be denoted by $\wcD_X$. For what follows, we will need an interpretation of this sheaf in terms of
Lie algebroids, more precisely, we will use the fact that $\wcD_X$ is
the enveloping algebra of the Lie algebroid
$
z\Theta_X[z] \subset \Der_{\dC[z]}(\cO_X[z])
$
of rank $n=\dim X$
over the $\dC[z]$-algebra $\cO_X[z]$.

Lie algebroids are the sheaf version of Lie-Rinehart algebras (see \cite{Rinehart-1963}). They were originally studied in the setting of Differential Geometry (the book \cite{Mack-2005} is a complete reference here), and for instance in \cite{Chemla-1999} the complex analytic case is considered.

For the ease of the reader, let us recall the notions of Lie-Rinehart algebras and Lie algebroids. Let us take a commutative base ring $k$ and a commutative $k$-algebra $A$ (resp. a sheaf of commutative $k$-algebras $\cA$ over a topological space $X$). A {\em Lie-Rinehart algebra} over $(k,A)$ is a (left) $A$-module $L$ which is also a $k$-Lie algebra, endowed with an {\em anchor map} $\varrho: L \to \Der_k(A)$: a (left) $A$-linear morphism of $k$-Lie algebras such that
$$
[\lambda,a\lambda'] = a [\lambda,\lambda'] + \varrho(\lambda)(a) \lambda'
$$
for all $\lambda,\lambda'\in L$ and $a\in A$. Respectively, a {\em Lie algebroid} over $(k,\cA)$ is a (left) $\cA$-module $\cL$ which is also a sheaf of $k$-Lie algebras, endowed with an {\em anchor map} $\varrho: \cL \to \Der_k(\cA)$: a (left) $\cA$-linear morphism of sheaves of $k$-Lie algebras such that
$$
[\lambda,a\lambda'] = a [\lambda,\lambda'] + \varrho(\lambda)(a) \lambda'
$$
for all local sections $\lambda,\lambda'$ of $\cL$ and $a$ of $\cA$.

We usually assume that $L$ is a projective $A$-module  (resp. $\cL$ is a locally free $\cA$-module) of finite rank.
It is clear that for each point $p\in X$, by using the natural morphism $\Der_k(\cA)_p \to \Der_k(\cA_p)$, any Lie algebroid $\cL$ over $(k,\cA)$ gives rise to a Lie-Rinehart algebra $\cL_p$ over $(k,\cA_p)$.

For any Lie-Rinehart algebra $L$ over $(k,A)$ (resp. for any Lie algebroid $\cL$ over $(k,\cA)$), there is a universal $k$-algebra $\U ( L)$ (resp. sheaf of $k$-algebras $\U (\cL)$), endowed with canonical maps $A\to \U (L)$ and $L\to \U( L)$ (resp. $\cA\to \U (\cL)$ and $\cL\to \U( \cL)$), called the {\em enveloping algebra} of $L$ (resp. of $\cL$). One easily proves that for each $p\in X$, there is a canonical isomorphism
$\U(\cL_p) \cong \U(\cL)_p$. Enveloping algebras are naturally (positively) filtered.

Over a complex analytic manifold $X$, our basic example is $k=\dC$, $\cA=\cO_X$ and $\cL= \Theta_X = \Der_\dC(\cO_X)$, where the anchor is the identity. In this case, the corresponding enveloping algebra is the sheaf of differential operators $\cD_X$ (cf. \cite[\S~2]{Chemla-1999}).
The fact that $\cR_F \cD_X$ is the enveloping algebra of the Lie algebroid $z\Theta_X[z]$ over $(\dC[z], \cO_X[z])$, where the anchor map is given by the inclusion $z\Theta_X[z]\subset \Der_{\dC[z]}(\cO_X[z])$, is easy (it is basically \cite[Exercise 8.7]{MHM}). In general, the enveloping algebra of any Lie algebroid comes equipped with a natural filtration, but in the case of $\cR_F \cD_X$, we have an additional structure: the grading by $z$. All the objects $\dC[z],\cO_X[z],z\Theta_X[z]$ are graded by the powers of $z$, and the enveloping algebra of $z\Theta_X[z]$ inherits a graded structure which coincides with the original one in $\cR_F \cD_X$. Later, we will come back to the role of graded structures on our Lie algebroid (or our Lie-Rinehart algebra).

On the other hand, the natural filtration on $\wcD_X$ as enveloping algebra of $z\Theta_X[z]$ can be derived from its graded structure in the following way (notice that it is not the usual filtration associated to the grading):
$$
F_k \wcD_X = \left( \bigoplus_{i=0}^{k-1} F_i\cD_X\cdot z^i\right) \bigoplus \left( \bigoplus_{i=k}^\infty F_k\cD_X\cdot z^i\right).
$$

We have a canonical commutative diagram of graded $\cO_X[z]$-algebras
\begin{equation}\label{eq:DiagSym}
\begin{tikzcd}
\Sym^\bullet_{\cO_X[z]} (z\Theta_X[z])   \ar[r, "\cong"]  \ar[hook]{d} &
\gr_\bullet^F \wcD_X  \ar[hook]{d}\\
\Sym^\bullet_{\cO_X[z]} (\Theta_X[z]) \ar[r, "\cong"] \ar[d, "\cong"']  &  \gr_\bullet^F (\cD_X[z])
\ar[d, "\cong"]
\\
\Sym^\bullet_{\cO_X} (\Theta_X)[z]  \ar[r, "\cong"] & \gr_\bullet^F (\cD_X)[z],
\end{tikzcd}
\end{equation}
where $F_\bullet (\cD_X[z]):=(F_\bullet \cD_X)[z]$ and
where the horizontal arrows are isomorphisms by the Poincar\'e-Birkhoff-Witt theorem for Lie algebroids resp. Lie-Rinehart algebras (see, e.g., \cite[Theorem 3.1]{Rinehart-1963}). Let us notice that all objects in the above diagram are bigraded (where the additional grading is given by the $z$-degree), and all maps in the diagram are morphisms of bigraded algebras.

For a filtered left (resp. right) $\cD_X$-module $(\cM,F_\bullet)$, we write
\label{page:defReesFunctor}$\cR(\cM,F_\bullet):=\cR_F \cM:=\oplus_{k\in \dZ} F_k\cM\cdot z^k$
for its associated Rees module, which is naturally a graded (by degree in $z$) left (resp. right)  $\wcD_X$-module. A graded left (or right) $\wcD_X $-module $\widetilde{\cM}$ is called {\em strict} if it has no $z$-torsion.
Associating $\cR(\cM,F_\bullet)$ to $(\cM,F_\bullet)$ defines a faithful functor from the category of filtered left (resp. right) $\cD_X$-modules to the category of graded left (resp. right) $\wcD_X $-modules.
It is readily verified that a graded left (resp. right) $\wcD_X $-module
$\widetilde{\cM}$ is the Rees module of a filtered left (resp. right) $\cD_X$-module $(\cM,F_\bullet)$ iff it has no $z$-torsion. Notice that in this case we naturally have
\begin{equation}\label{eq:FunctorReesDModGr}
\widetilde{\cM}/(z-1)\widetilde{\cM}\cong \cM
\quad\quad\text{and}
\quad\quad
\widetilde{\cM}/z\widetilde{\cM}\cong \gr_\bullet^F\cM.
\end{equation}
Moreover, for a sequence of filtered $\cD_X$-modules
$$
(\cM',F'_\bullet) \xrightarrow{i} (\cM,F_\bullet) \xrightarrow{p} (\cM'',F''_\bullet),\ \text{with}\ p\circ i=0,
$$
the following properties are equivalent:
\begin{enumerate}
    \item The underlying sequence of $\cD_X$-modules
    $\cM' \xrightarrow{i} \cM \xrightarrow{p} \cM''$ is exact and strict with respect to $F_\bullet$
    (i.e. $F_k \cM\cap (\ker p)  = i(F'_k \cM')$ for all $k\in \dZ $).
                \item The associated sequence of $\wcD_X$-modules
$
\cR_{F'} \cM' \xrightarrow{\cR\, i} \cR_{F} \cM \xrightarrow{\cR\, p} \cR_{F''} \cM''
$
is exact.
    \item The associated sequence of $\gr_\bullet^F \cD_X$-modules
$
\gr^{F'}_\bullet \cM' \xrightarrow{\gr i} \gr^{F}_\bullet \cM \xrightarrow{\gr p} \gr^{F''}_\bullet \cM''
$
is exact.
\end{enumerate}

The left $\cD_X$-module $\cO_X$ will be always endowed with the trivial filtration
$$
F_k \cO_X = \left\{ \begin{array}{lcl}
0 & \text{if} & k < 0\\
\cO_X & \text{if} & k \geq 0,
\end{array}\right.
$$
such that $\cR_F \cO_X = \cO_X[z]$ is a graded left $\wcD_X $-module with the usual $z$-grading, that will be denoted by $\wcO_X$.

The canonical right $\wcD_X $-module (see \cite[Example 8.1.9]{MHM})
$$
\womega_X := \bigwedge^n \left( z \Theta_X[z] \right)^* = z^{-n} \omega_X[z],
$$
can also be seen as the Rees module $\cR_F \omega_X$ associated with the canonical right $\cD_X$-module $\omega_X$ endowed with the filtration
\begin{equation}\label{eq:FiltrOmega}
F_k \omega_X = \left\{ \begin{array}{lcl}
0 & \text{if} & k < -n\\
\omega_X & \text{if} & k \geq -n.
\end{array}\right.
\end{equation}
The main advantage to work with (graded) $\wcD_X$-modules rather than with filtered $\cD_X$-modules is that the former category
is abelian whereas the latter is not.
\medskip

Let $k$ be a commutative graded ring (e.g., $k=\dC[z]$) and $\wcB$ a sheaf of graded $k$-algebras over a topological space $X$ (e.g., $\wcB=\wcD_X$ or $\wcB=\wcO_X$ for $X$ a complex manifold). We recall the sheaf version of some well-known definitions for graded modules over graded rings (cf. \cite[\S 1]{Foss-Fox-1974}).

For any graded left resp. right $\wcB$-modules $\wcM$, $\wcN$, let us denote by $\astHom_{\wcB}(\wcM,\wcN)$ the sheaf of graded $k$-modules
$$\bigoplus_{i\in\dZ} \Homgr_{\wcB}(\wcM,\wcN(i)) \subset \Hom_{\wcB}(\wcM,\wcN),
$$
where the $\wcN(i)$ is the shifted graded module defined as $\wcN(i)_j = \wcN(i+j) $ for all $i,j\in \dZ$.
The above inclusion is an equality whenever $\wcM$ is locally of finite presentation.

For any graded right $\wcB$-module $\wcQ$ and any left $\wcB$-module $\wcM$, the graded tensor product $\wcQ \astotimes_{\wcB} \wcM$ is the sheaf of  $k$-modules $\wcQ \otimes_{\wcB} \wcM$ endowed with the grading
$$
\left[\wcQ_p \astotimes_{\wcB_p} \wcM_p \right]_d = \langle x \otimes y\ |\ x\in [\wcQ_p]_i, y\in [\wcM_p]_j, i+j = d \rangle,\quad \forall p\in X.
$$
Clearly, $\astHom_{\wcB}(\wcM(i),\wcN(j))= \astHom_{\wcB}(\wcM,\wcN)(j-i)$ and $\wcQ(i) \astotimes_{\wcB} \wcM(j) = \left(\wcQ \astotimes_{\wcB} \wcM\right)(i+j)$ for all integers $i,j$.
If $\wcM$ is a left resp. right $\wcB$-module, then $\astHom_{\wcB}(\wcM,\wcB)$ is a graded right resp. left $\wcB$-module.

The complex $\RastHom_{\wcD_X}(\wcO_X,\wcD_X)$ can be computed by means of the Spencer resolution of $\wcO_X$
$$
\Sp^\bullet_{\wcD_X}(\wcO_X) :=  \cdots\to\wcD_X \otimes_{\wcO_X} \bigwedge^n z\Theta_X[z]  \to \cdots \to \wcD_X \otimes_{\wcO_X}  z\Theta_X[z]   \to \wcD_X
\to 0,
$$
graded by
$$
 \left[ \wcD_X \otimes_{\wcO_X} \bigwedge^i z\Theta_X[z] \right]_k = \left[\wcD_X\right]_{k-i} \otimes_{\cO_X} \bigwedge^i \Theta_X,\ i=0,\dots,n.
$$
The complex $\RastHom_{\wcD_X}(\wcO_X,\wcD_X)$ is concentrated in homological degree $n$ and we have a canonical isomorphism of graded right $\wcD_X$-modules
$$
\womega_X \cong \cH^n \RastHom_{\wcD_X}(\wcO_X,\wcD_X) =: {{^*}\!\Ext}_{\wcD_X}^n(\wcO_X,\wcD_X).
$$

As for the non-graded case, we have natural equivalences of categories between graded left $\wcD_X$-modules and graded right $\wcD_X$-modules:
\begin{enumerate}
\item For any graded left $\wcD_X$-module $\widetilde{\cM}$, we define
$$
\widetilde{\cM}^{\text{right}} := \womega_X \astotimes_{\wcO_X} \widetilde{\cM} \cong
\omega_X \otimes_{\cO_X} \widetilde{\cM}
$$
with grading $\left[\widetilde{\cM}^{\text{right}}\right]_i = \omega_X \otimes_{\cO_X} \widetilde{\cM}_{i+n}. $

\item For any graded right $\wcD_X$-module $\widetilde{\cM}$, we define
$$
\widetilde{\cM}^{\text{left}} := {^*}\!\Hom_{\widetilde{\cO}_X}(\womega_X,\widetilde{\cM}) \cong
\Hom_{\cO_X}(\omega_X,\widetilde{\cM})
$$
with grading $\left[\widetilde{\cM}^{\text{left}}\right]_i = \Hom_{\cO_X}(\omega_X,\widetilde{\cM}_{i-n})$.
\end{enumerate}

For a complex of graded left (resp. right) $\wcD_X $-modules $\widetilde{\cM}$ (which may not come from a complex of filtered $\cD_X$-modules $(\cM,F_\bullet)$), we define its dual complex to be
\begin{equation}
\label{eq:DualRMod-left}
\astdD \widetilde{\cM} := \RastHom_{\wcD_X}(\widetilde{\cM}, \wcD_X )^{\text{left}}[\dim\,X]
\end{equation}
\begin{equation}
\label{eq:DualRMod}
(\text{resp.}\
\astdD \widetilde{\cM} :=
\RastHom_{\wcD_X}(\widetilde{\cM}, \wcD_X )^{\text{right}}[\dim\,X] =
\RastHom_{\wcD_X}(\widetilde{\cM}, \womega_X\astotimes_{\wcO_X} \wcD_X )[\dim\,X] ).
\end{equation}
We have canonical isomorphisms of graded left (resp. right) $\wcD_X$-modules (actually, complexes concentrated in degree $0$)
\begin{equation} \label{eq:grdual-O-omega}
    \astdD \wcO_X \cong \wcO_X\quad (\text{resp. } \astdD \womega_X \cong \womega_X).
\end{equation}
We also have canonical isomorphisms
\begin{equation}\label{eq:ShiftCMProp-1}
\astdD (\wcM(i)) \cong (\astdD \wcM)(-i)
\end{equation}
for all $i\in \dZ$.

Given a filtered holonomic module, one can consider the dual of its Rees module. However, this is not in general the Rees module of a filtered module. The next lemma gives a criterion to know when this is the case.
\begin{lemma}[{see, e.g., \cite[8.8.22]{MHM}}]  \label{lemma:8.8.22.MHM}
Let $(\cM,F_\bullet)$ be a filtered left or right holonomic $\cD_X$-module, and $\widetilde{\cM}:=\cR(\cM,F_\bullet)$ its associated Rees module.
Then
$\cH^i(\astdD \widetilde{\cM})=0$ for all $i\neq 0$ and $\cH^0(\astdD \widetilde{\cM})$ has no $z$-torsion iff $(\cM,F_\bullet)$ has the \emph{Cohen-Macaulay} property, that is, iff $\gr^F_\bullet \cM$ is a Cohen-Macaulay
$\gr^F_\bullet \cD_X$-module.

In this case, $\cH^0(\astdD \widetilde{\cM})$  is the Rees module of a filtered module, the underlying $\cD_X$-module of which is $\dD \cM$ and we denote the filtration thus defined by
$F^\dD_\bullet \dD \cM$. We  call $(\dD \cM, F^\dD_\bullet)$ the dual filtered module, and
$F^\dD_\bullet \dD \cM$ the dual filtration on $F_\bullet \cM$.
\end{lemma}

From \eqref{eq:grdual-O-omega}, the canonical left (resp. right) $\cD_X$-module $\cO_X$ (resp. $\omega_X$) is self-dual as filtered module with the trivial filtration (resp. with the filtration \eqref{eq:FiltrOmega}).

It is clear that if $(\cM,F_\bullet)$ is a filtered holonomic $\cD_X$-module having the Cohen-Macaulay property, then $(\cM,F(i)_\bullet)$ also has the Cohen-Macaulay property and (see \eqref{eq:ShiftCMProp-1} above)
\begin{equation}\label{eq:ShiftCMProp-2}
(F_{\bullet+i})^\dD = F^\dD_{\bullet-i}
\end{equation}
for any integer $i$.
\medskip

Let us come back to the situation of a free divisor $D\subset X$. The sheaf of logarithmic differential operators $\cV_X^D \subset \cD_X$ is the enveloping algebra of the Lie algebroid $\Theta_X(-\log D) \subset \Theta_X$  over the $\dC$-algebra $\cO_X$ (this is basically \cite[Corollary 2.2.6]{calde_ens}).
The Rees ring $\wcV_X^D$
of the filtered ring $(\cV_X^D, F_\bullet)$ turns out to be the enveloping algebra of the Lie algebroid $z\Theta_X(-\log D)[z] \subset \Der_{\dC[z]}(\cO_X[z])$
of rank $n$ over the $(\dC[z],\cO_X[z])$.

The Lie algebroids $z\Theta_X[z]$ and $z\Theta_X(-\log D)[z]$ with their $z$-grading are {\em graded Lie algebroids} in the sense that
they are graded as $\cO_X[z]$-modules and as sheaves of $\dC[z]$-Lie algebras, and that their anchor maps
$$
z\Theta_X(-\log D)[z],\ z\Theta_X[z] \xrightarrow{\text{nat.}} {^*}\!\Der_{\dC[z]}(\cO_X[z])
$$
are graded too. Moreover, the inclusion $ z\Theta_X(-\log D)[z]\subset z\Theta_X[z]$ is graded and so is a {\em map of graded Lie algebroids} in the sense we leave the reader to write down.

We know that the {\em relative dualizing module} (see Definition (A.28) of \cite{nar_symmetry_BS}) for the inclusion of Lie algebroids $\cL_0:=\Theta_X(-\log D) \subset \cL:=\Theta_X$ is $\omega_{\cL/\cL_0} = \Hom_{\cO_X}\left(\omega_\cL,\omega_{\cL_0}\right)$ with $\omega_\cL = \bigwedge^n \cL^* = \omega_X $, $\omega_{\cL_0} = \bigwedge^n \cL_0^* = \omega_X(\log D) $, and so $
\omega_{\cL/\cL_0} = \cO_X(D)$.

In a completely similar way, the relative dualizing module of the inclusion $\wcL_0:=z\Theta_X(-\log D)[z] \subset \wcL:=z\Theta_X[z]$ is $\omega_{\wcL/\wcL_0} = \Hom_{\wcO_X}\left(\omega_{\wcL},\omega_{\wcL_0}\right) $ with $\omega_{\wcL} = \bigwedge^n \wcL^*  = z^{-n} \omega_X[z] = \womega_X $, $\omega_{\wcL_0} = \bigwedge^n \wcL_0^* = z^{-n}\omega_X(\log D)[z]) =: \womega_X(\log D)$, and so we have a canonical isomorphism $
\omega_{\wcL/\wcL_0} = \cO_X(D)[z]$,
but in this case, due to the graded structures, this relative dualizing module is also naturally graded since it can be defined by using $\astHom_{\wcO_X}$ instead of $\Hom_{\wcO_X}$, and of course, this grading coincides with the $z$-grading of $\cO_X(D)[z]$.
\medskip

We have seen before (see Theorem \ref{theo:LCT}) that if $D \subset X$ is such that the LCT holds (e.g. if $D$ is strongly Koszul), then there is an isomorphism of left $\cD_X$-modules $\cD_X\otimes_{\cV_X^D} \cO_X(D) \cong \cO_X(*D)$. Clearly, $\cO_X(D)$ is a left $\cV_X^D$-module
that is locally free (actually, of rank one) over $\cO_X$.

This situation can be generalized by replacing $\cO_X(D)$ by an \emph{integrable logarithmic connection} (ILC) $\cE$ with respect to $D$, which by definition is a locally free $\cO_X$-module of finite rank endowed with a left $\cV_X^D$-module structure. This notion is useful
when one aims at studying the cohomology of $X\backslash D$
with respect to some local coefficient system (i.e., the local system of horizontal sections of $\cE$).
The next two results will be stated and proved in this greater generality. Although we will use them in this paper only for the case $\cE=\cO_X(D)$, we believe that they may be useful for future applications. Moreover, for the final proof of Proposition \ref{prop:StrictnessDualOrder} below (even if we are only interested in the case $\cE=\cO_X(D)$) some intermediate steps need to be carried out for an arbitrary ILC.

Given an ILC $\cE$ with respect to $D$, we define (in complete analogy to
Definition \ref{def:F-ord}) the {\em order filtration} on $\cD_X \otimes_{\cV_X^D} \cE$ to be:
$$
F_k^{\ord} \left(\cD_X \otimes_{\cV_X^D} \cE \right) := \im \left(  F_k\cD_X \otimes_{\cO_X} \cE\right) \subset \cD_X \otimes_{\cV_X^D}\cE.
$$
In other words, $(\cD_X \otimes_{\cV_X^D} \cE, F_\bullet^{\ord})$ is the filtered tensor product of $\cD_X$ with its order filtration and $\cE$ with its trivial filtration given by $F_k \cE = \cE$ for $k\geq 0$ and $F_k \cE = 0$ for $k< 0$.

We have a natural $\wcD_X$-linear graded map
\begin{equation} \label{eq:map-rees-ilc}
\wcD_X \astotimes_{\wcV_X^D} \wcE \longrightarrow \cR(\cD_X \otimes_{\cV_X^D} \cE, F_\bullet^{\ord})
\end{equation}
induced by the $\wcV_X^D$-linear graded map
$$
\begin{array}{rcl}
  \wcE = \cE[z] & \longrightarrow & \cR(\cD_X \otimes_{\cV_X^D} \cE, F_\bullet^{\ord}) \\ \\
 \sum_i e_i z^i & \longmapsto &\sum_i (1\otimes e_i) z^i.
\end{array}
$$
Notice that $\cR(\cD_X \otimes_{\cV_X^D} \cE, F_\bullet^{\ord})$ is a left graded $\wcV_X^D$-module through the inclusion $\wcV_X^D\subset \wcD_X$.

\begin{lemma} \label{lemma:rees-ilc}
Let $D$ be a Koszul free divisor, and let $\cE$ be an ILC with respect to $D$. Then:
\begin{enumerate}
    \item The map (\ref{eq:map-rees-ilc}) is an isomorphism of graded $\wcD_X$-modules
\item The complex $\wcD_X \Lastotimes_{\wcV_X^D} \wcE$ is concentrated in degree $0$ and so
we have an isomorphism in the derived category of complexes
 of left graded $\wcD_X$-modules
$$
\wcD_X \Lastotimes_{\wcV_X^D} \wcE \stackrel{\cong}{\longrightarrow}
\wcD_X \astotimes_{\wcV_X^D} \wcE.
$$
\end{enumerate}
\end{lemma}

\begin{proof} Both properties can be proved by forgetting the graded structures.
It is easy to see that under the Koszul hypothesis on $D$, the inclusion of Lie algebroids (over the $\dC[z]$-algebra $\cO_X[z]$)
$$
z\Theta_X(-\log\,D)[z]\subset z\Theta_X[z]
$$
is a  \emph{Koszul pair} in the sense of \cite[Definition 1.16]{calde_nar_lct_ilc}, i.e.,
some (or any) local $\cO_X[z]$-basis of $z\Theta_X(-\log\,D)[z]$ forms a regular sequence in $\Sym^\bullet_{\cO_X[z]} (z\Theta_X[z])$, namely, by
using the inclusion $\Sym^\bullet_{\cO_X[z]}(z\Theta_X[z]) \hookrightarrow \Sym^\bullet_{\cO_X}(\Theta_X)[z])$ (see diagram \eqref{eq:DiagSym} above) and the fact that the Koszul property means exactly that the inclusion of Lie algebroids (over $\cO_X$) $\Theta_X(-\log\,D)\subset \Theta_X$ is a Koszul pair.

Following
\cite[\S~1.1.2]{calde_nar_lct_ilc}, we can define for both inclusions of Lie algebroids
$\Theta_X(-\log\,D) \subset \Theta_X$ and
$z\Theta_X(-\log\,D)[z] \subset z\Theta_X[z]$, for an ILC $\cE$ (which is a left $\cV_X^D$-module) and
the corresponding Rees module $\wcE$ (which is a $\wcV_X^D$-module) the Spencer complexes
$\Sp^\bullet_{\cV_X^D}(\cE)$ and
$\Sp^\bullet_{\wcV_X^D}(\wcE)$, which are, respectively,
$$
\begin{array}{l}
\Sp^\bullet_{\cV_X^D}(\cE):\\ \\  \cdots\to\cV^D_X \otimes_{\cO_X} \bigwedge^n \Theta_X(-\log D) \otimes_{\cO_X} \cE \to \cdots \to \cV^D_X \otimes_{\cO_X}  \Theta_X(-\log D) \otimes_{\cO_X} \cE \to \cV^D_X \otimes_{\cO_X} \cE \to 0,
\end{array}
$$
and
$$
\begin{array}{l}
\Sp^\bullet_{\wcV_X^D}(\wcE):\\ \\  \cdots\to\wcV^D_X \otimes_{\wcO_X} \bigwedge^n z\Theta_X(-\log D)[z] \otimes_{\wcO_X} \wcE \to \cdots \to \wcV^D_X \otimes_{\wcO_X}  z\Theta_X(-\log D)[z] \otimes_{\wcO_X} \wcE \to \wcV^D_X \otimes_{\wcO_X} \wcE
\to 0.
\end{array}
$$
They have augmentations to $\cE$ resp. to $\wcE$ and are a locally free $\cV_X^D$- resp. $\wcV_X^D$-resolution of $\cE$ resp. $\wcE$.

We can also consider the complex $\Sp^\bullet_{\cV_X^D,\cD_X}(\cE):=\cD_X \otimes_{\cV_X^D} \Sp^\bullet_{\cV_X^D}(\cE)$ resp.
$\Sp^\bullet_{\wcV_X^D,\wcD_X}(\wcE):=\wcD_X \otimes_{\wcV_X^D} \Sp^\bullet_{\wcV_X^D}(\wcE)$, which are
$$
\begin{array}{l}
\Sp^\bullet_{\cV_X^D,\cD_X}(\cE):\\ \\  \cdots\to\cD_X \otimes_{\cO_X} \bigwedge^n \Theta_X(-\log D) \otimes_{\cO_X} \cE \to \cdots \to \cD_X \otimes_{\cO_X}  \Theta_X(-\log D) \otimes_{\cO_X} \cE \to \cD_X \otimes_{\cO_X} \cE \to 0,
\end{array}
$$
and
$$
\begin{array}{l}
\Sp^\bullet_{\wcV_X^D,\wcD_X}(\wcE):\\ \\  \cdots\to\wcD_X \otimes_{\wcO_X} \bigwedge^n z\Theta_X(-\log D)[z] \otimes_{\wcO_X} \wcE \to \cdots \to \wcD_X \otimes_{\wcO_X}  z\Theta_X(-\log D)[z] \otimes_{\wcO_X} \wcE \to \wcD_X \otimes_{\wcO_X} \wcE
\to 0,
\end{array}
$$
respectively. According to  \cite[Proposition 1.18]{calde_nar_lct_ilc}, since
both inclusions of Lie algebroids are Koszul pairs, the cohomology of both complexes is concentrated in degree $0$, and equal to $\cD\otimes_{\cV_X^D} \cE$ and
$\wcD\otimes_{\wcV_X^D} \wcE$, respectively. This proves in particular the second statement.

But actually the proof of this result in {\em loc.~cit.} gives us an additional strictness property for
$\Sp^\bullet_{\cV_X^D,\cD_X}(\cE)$. Namely,
if we filter this complex as
$$
F^i_k \left( \cD_X \otimes_{\cO_X} \bigwedge^i \Theta_X(-\log D) \otimes_{\cO_X} \cE\right) = F_{k-i} \cD_X \otimes_{\cO_X} \bigwedge^i \Theta_X(-\log D) \otimes_{\cO_X} \cE,\ i=0,\dots,n,
$$
we obtain that
$$\gr^F_\bullet \left( \cD_X \otimes_{\cV_X^D} \SP^\bullet_{\cV_X^D}(\cE) \right) \cong \left(\gr^F_\bullet \cD_X \otimes_{\cO_X} \bigwedge^\bullet \Theta_X(-\log D) \right) \otimes_{\cO_X} \cE,
$$
and this complex is concentrated in degree 0
by the Koszul hypothesis (and so is $\Sp^\bullet_{\cV_X^D,\cD_X}(\cE)$), but this result also implies that the differentials
$$
\cD_X \otimes_{\cO_X} \bigwedge^i \Theta_X(-\log D) \otimes_{\cO_X} \cE \xrightarrow{d^{-i}} \cD_X \otimes_{\cO_X} \bigwedge^{i-1} \Theta_X(-\log D) \otimes_{\cO_X} \cE,\ i=1,\dots,n
$$
are strict for the above filtrations. In particular, the right exact sequence
$$
\left(\cD_X \otimes_{\cO_X}  \Theta_X(-\log D) \otimes_{\cO_X} \cE, F^1\right) \xrightarrow{d^{-1}} \left(\cD_X \otimes_{\cO_X} \cE,F^0\right) \to \left(\cD_X \otimes_{\cV_X^D} \cE, F_\bullet^{\ord}\right) \to 0
$$
is strict, or equivalently, the sequence
\begin{equation}\label{eq:StrictSeq}
\cR_{F^1}(\cD_X \otimes_{\cO_X}  \Theta_X(-\log D) \otimes_{\cO_X} \cE) \to \cR_{F^0}(\cD_X \otimes_{\cO_X} \cE) \to \cR_{F_\bullet^{\ord}}(\cD_X \otimes_{\cV_X^D} \cE) \to 0
\end{equation}
is exact. Notice also that we clearly have
\begin{equation}\label{eq:CommuteRees}
\cR_{F^i}
\left(
\cD_X\otimes_{\cO_X}\bigwedge^i\Theta_X(-\log\,D)\otimes_{\cO_X}\cE
\right)
\cong
\wcD_X\otimes_{\wcO_X}\bigwedge^i z\Theta_X(-\log\,D)[z]\otimes_{\wcO_X}\wcE
\end{equation}
for all $i\in\{0,\ldots, n\}$, since $F_k\cD_X$, $\Theta_X(-\log D)$ as well as $\cE$ are $\cO_X$-locally free (hence, flat over $\cO_X$).

To finish, consider the commutative diagram of graded left $\wcD_X$-modules

$$
\xymatrix{\ar @{} [dr] |{} \wcD_X \otimes_{\wcO_X} (z\Theta_X(-\log D)[z]) \otimes_{\wcO_X}
\wcE \ar[r] \ar[d] & \wcD_X \otimes_{\wcO_X}  \wcE \ar[r] \ar[d] & \wcD_X \otimes_{\wcV_X^D}  \wcE \ar[r] \ar[d] & 0\\
\cR_{F^1}(\cD_X \otimes_{\cO_X}  \Theta_X(-\log D) \otimes_{\cO_X} \cE) \ar[r] & \cR_{F^0}(\cD_X \otimes_{\cO_X} \cE) \ar[r] & \cR_{F_\bullet^{\ord}}(\cD_X \otimes_{\cV^D_X} \cE) \ar[r] & 0.
}
$$

The first row is exact since $\Sp^\bullet_{\wcV_X^D,\wcD_X}(\wcE)$ is a resolution of $\wcD_X\otimes_{\wcV_X^D}\wcE$, the second row is so as we have just explained (sequence \eqref{eq:StrictSeq} above). By \eqref{eq:CommuteRees},
the first and second vertical arrows are isomorphisms and so is the third one.
\end{proof}

To prove our main result in this section, we will use a graded (and Lie-algebroid) version of Theorem (A.32) of \cite{nar_symmetry_BS}. However, instead of stating it in full generality (for general graded Lie algebroids or Lie-Rinehart algebras), we will only state it for the case we need, namely, the inclusion of graded Lie algebroids $\wcL_0=z\Theta_X(-\log\,D)[z]\subset \wcL=z\Theta_X[z]$ and the corresponding map of graded enveloping algebras $\wcV_X^D \subset \wcD_X$.

\begin{proposition} \label{prop:A32-grad}
Let $\cF$ be a graded locally free $\wcO_X$-module of finite rank endowed with a graded left module structure over $\wcV_X^D$. We have a canonical isomorphism in the derived category of graded $\wcD_X$-modules
$$
\wcD_X \Lastotimes_{\wcV_X^D}  \left(  \cO_X(D)[z] \astotimes_{\wcO_X} \cF^*\right) \cong
\astdD_{\wcD_X} \left(\wcD_X \Lastotimes_{\wcV_X^D} \cF   \right),
$$
where $\cF^*= \astHom_{\wcO_X}(\cF,\wcO_X)$ as graded left $\wcV_X^D$-module.
\end{proposition}

\begin{proof}[Proof (Outline)] This proof is a straightforward translation of the proof of  Theorem (A.32) of \cite{nar_symmetry_BS} to the graded case. It essentially consists of replacing functors $\Hom$ and $\otimes$ by $\astHom$ and $\astotimes$ in all the constructions in the Appendix of {\em loc.~cit.}, as well as the fact that the relative dualizing module $\omega_{\wcL/\wcL_0}$ is naturally graded. We sketch it for the convenience of the reader, and, actually, we propose a shorter and slightly different approach.

Let us recall that $\omega_{\wcL/\wcL_0} = \astHom_{\wcO_X}\left(\womega_X,\womega_X(\log D)\right) \cong \cO_X(D)[z] $.
For each left graded $\wcV_X^D$-module $\cN$, the natural graded $\wcO_X$-linear map
$$
\begin{array}{rcl}
\D
\omega_{\wcL/\wcL_0} \astotimes_{\wcO_X} \astHom_{\wcO_X}\left(\womega_X(\log D), \astHom_{\wcV_X^D}(\cN,\wcV_X^D)\right) & \longrightarrow &
\D \astHom_{\wcO_X}\left(\womega_X,\astHom_{\wcV_X^D}(\cN,\wcD_X)\right)\\ \\
\D f \otimes g & \longmapsto & \D i \circ g \circ f,
\end{array}
$$
where $i:\wcV_X^D \to \wcD_X$ is the inclusion, turns out to be left $\wcV_X^D$-linear. It induces a natural graded $\wcD_X$-linear map
$$
\wcD_X \astotimes_{\wcV_X^D} \left(\omega_{\wcL/\wcL_0} \astotimes_{\wcO_X} \left[\astHom_{\wcV_X^D}(\cN,\wcV_X^D)\right]^\text{left}
\right) \longrightarrow \left[\astHom_{\wcV_X^D}(\cN,\wcD_X)\right]^\text{left} =
\left[\astHom_{\wcD_X}(\wcD_X \astotimes_{\wcV_X^D}\cN,\wcD_X)\right]^\text{left},
$$
and so a natural map in the derived category of left graded $\wcD_X$-modules
$$
\wcD_X \astotimes_{\wcV_X^D} \left( \omega_{\wcL/\wcL_0} \astotimes_{\wcO_X} \astdD_{\wcV_X} (\cN) \right) \longrightarrow
 \astdD_{\wcD_X} \left( \wcD_X \Lastotimes_{\wcV_X^D}\cN \right),
$$
which, by standard reasons, is an isomorphism whenever $\cN$ is coherent. To finish, notice that if $\cN$ is a graded locally free $\wcO_X$-module, then we have a canonical isomorphism
$$
\astdD_{\wcV_X} (\cN) \cong \astHom_{\wcO_X}(\cN,\wcO_X) = \cN^*
$$
(cf.  Corollary (A.24) of \cite{nar_symmetry_BS}).
\end{proof}

We are now ready to prove the announced result on the dual of the filtered module $(\cD_X\otimes_{\cV_X^D}\cE,F^{\ord}_\bullet)$.

\begin{proposition} \label{prop:StrictnessDualOrder}
Assume that $D\subset X$ is a Koszul free divisor and that $\cE$ is an ILC with respect to $D$. Consider the filtered holonomic $\cD_X$-module $\cM:=(\cD_X \otimes_{\cV_X^D} \cE, F_\bullet^{\ord})$, and its corresponding Rees module
$\widetilde{\cM}:=\cR(\cD_X \otimes_{\cV_X^D} \cE, F_\bullet^{\ord})$. Then the dual module
$\astdD \widetilde{\cM}$ is strict, that is,
$\cH^i(\astdD \widetilde{\cM})= 0 $ for $i\neq 0$
and $\cH^0(\astdD \widetilde{\cM})$ has no $z$-torsion.
Moreover, we have an isomorphism of graded left $\cD_X$-modules
$$
\cH^0(\astdD \widetilde{\cM})
=\cR(\cD_X \otimes_{\cV_X^D} \cE^*(D), F_\bullet^{\ord}).
$$
In particular, the filtered module $(\cD_X \otimes_{\cV_X^D} \cE, F_\bullet^{\ord})$ satisfies the Cohen-Macaulay property and its dual filtered module is isomorphic to $(\cD_X \otimes_{\cV_X^D} \cE^*(D), F_\bullet^{\ord})$.
\end{proposition}
\begin{proof}

By Lemma \ref{lemma:rees-ilc}, we have
$$
\wcD_X \Lastotimes_{\wcV_X^D} \wcE \xrightarrow{\sim} \wcD_X \astotimes_{\wcV_X^D} \wcE \xrightarrow{\sim} \wcM,
$$
and so
we have
$$
\begin{array}{rclcl}
\astdD_{\wcD_X} \wcM &\cong &\astdD_{\wcD_X} \left(\wcD_X \Lastotimes_{\wcV_X^D} \wcE\right) &\quad \quad \quad \quad & \textup{Lemma \ref{lemma:rees-ilc}, 1.+2. applied to the ILC }\cE \\ \\
&\cong& \wcD_X \Lastotimes_{\wcV_X^D} \left( \cO_X(D)[z] \astotimes_{\wcO_X} \wcE^* \right) && \textup{Proposition \ref{prop:A32-grad}} \\ \\
&\cong& \wcD_X \Lastotimes_{\wcV_X^D} \widetilde{\cE^*(D)} && \\ \\
&\cong&
\wcD_X \astotimes_{\wcV_X^D} \widetilde{\cE^*(D)} && \textup{Lemma \ref{lemma:rees-ilc}, 2. applied to the ILC }\cE^*(D)\\ \\
&\cong &\cR(\cD_X \otimes_{\cV_X^D} \cE^*(D), F_\bullet^{\ord}) && \textup{Lemma \ref{lemma:rees-ilc}, 1. applied to the ILC }\cE^*(D).
\end{array}
$$
Notice that we write $\cE^*=\Hom_{\cO_X}(\cE,\cO_X)$ and similarly
$\wcE^*=\Hom_{\wcO_X}(\wcE,\wcO_X)$.

In conclusion, the dual $\wcD_X$-module of $\wcM$ is strict, and the dual filtration $F^{\dD}_\bullet \; \dD \cM$ is given as the order filtration $F^{\ord}_\bullet$
on $\dD \cM \cong \cD_X\otimes_{\cV^D_X} \cE^*(D)$.
\end{proof}
\begin{remark}
For the case $\cE=\cO_X(D)$, we can directly deduce from the local representation
$$
\cD_{X,p}\otimes_{\cV_{X,p}^D}\cO_{X,p}(D)\cong \cO_{X,p}(*D)\cong \cD_{X,p}/(\delta_1,\ldots,\delta_{n-1},\chi+1)
$$
(see formula \eqref{eq:CyclicPresent-2}) that $\Gr^{F^{\ord}}_\bullet
\left(\cD_X\otimes_{\cV^D}\cO(D)\right)$ is a Cohen-Macaulay $\Gr^F_\bullet \cD_X$-module, since
the symbols in $\Gr^F_\bullet \cD_X$
of the operators $\delta_1,\ldots,\delta_{n-1},\chi+1$
form a regular sequence. However, later (see the proof of Theorem \ref{theo:DualFiltration} below) we need to know  that the dual filtration $F_\bullet^{\dD} \; \dD (\cD_X\otimes_{\cV_X^D}\cO_X(D))$ is
$F_\bullet^{\ord} (\cD_X\otimes_{\cV_X^D}\cO_X)$, which is provided by the proof above.
\end{remark}

The following is an easy variant of Proposition \ref{prop:StrictnessDualOrder} which we will need later in section \ref{sec:HodgeFiltration}.
\begin{corollary}
\label{cor:DualOrderShifted}
Let, as above, $D\subset X$ be a Koszul free divisor and $\cE$ be an ILC. Then for any $k\in \dZ$, the filtered holonomic module $(\cD\otimes_{\cV_X^D}\cE, F^{\ord}_{\bullet+k})$ has the Cohen-Macaulay property, and its dual filtered module is given by $(\cD_X \otimes_{\cV_X^D} \cE^*(D), F^{\ord}_{\bullet-k})$.
\end{corollary}
\begin{proof}
The statement of the corollary follows simply by combining the proof of Proposition \ref{prop:StrictnessDualOrder} with formula \eqref{eq:ShiftCMProp-2} and the remark that surrounds it.
\end{proof}

\section{Canonical and induced $V$-filtration}
\label{sec:VFilt}

In this section we are discussing in detail the $V$-filtration
on the module $i_{h,+} \cO_X(*D)$, where $i_h$ is the graph embedding for some local reduced equation $h$ of $D\subset X$.
This will be used in the next section in order to obtain information on
the Hodge filtration on $i_{h,+} \cO_X(*D)$, and on
$\cO_X(*D)$ itself.

Recall that for any complex manifold $M$, and
for a divisor $H \subset M$ with $\cI=\cI(H)$, we have the filtration $V^\bullet \cD_M$ defined by
$$
V^k \cD_M :=\left\{P\in \cD_M\,|\, P\cI^i \subset \cI^{i+k}\right\}.
$$
$V^0\cD_M$ is a sheaf of rings, notice that it equals the sheaf of logarithmic differential operators (with respect to $H$), which was denoted by $\cV_X^H$ in the previous chapter. Moreover, all $V^k\cD_M$ are sheaves of $V^0\cD_M$-modules. We will usually suppose that $H$ is smooth, and moreover that it is given by a globally defined equation $t\in \Gamma(M,\cO_M)$.

For
any holonomic $\cD_M$-module $\cM$, and any section $m\in \cM$, we write
$b^{\cM}_m(s)$ for the unique monic polynomial of minimal degree satisfying $b^{\cM}_m(\partial_t t)m\in t V^0\cD_M \cdot m$. Moreover, if $U^\bullet \cM$ is a good $V$-filtration on $\cM$, then it has (locally) a Bernstein polynomial denoted
by $b^{\cM}_{U^\bullet}(s)$, which is the minimal monic polynomial satisfying
$$
b^{\cM}_{U^\bullet}(s)(\partial_t t -k)U^k \cM \subset U^{k+1}\cM.
$$

We start by recalling the canonical $V$-filtration on a holonomic $\cD_M$-module. In general, it is indexed by the complex numbers, and in order to define it,
 one needs to choose an ordering
on $\dC$ such that for all $\alpha,\beta\in \dC$, we have $\alpha < \alpha+1$, $\alpha < \beta \Longleftrightarrow  \alpha+1<\beta+1$ and $\alpha<\beta+m$ for some $m\in \dZ$.
Notice however that for Hodge modules, only rational indices can occur. Moreover, we will later only use the integer parts of this filtration.
\begin{deflemma}\label{def:VFilt}
Let $M$ and $H=\{t=0\}$ as above, and let $\cM$ be a holonomic $\cD_M$-module.
\begin{enumerate}
\item (See \cite[Definition 3.1.1]{Saito1}, \cite[Section 1.2]{Saito94} and \cite[\S~4]{MebkhMais}) Then there exists a filtration $(V^\alpha_{\can} \cM)_{\alpha\in \dC}$ uniquely defined by the following properties:
\begin{enumerate}
    \item $\bigcup_{\alpha\in\dC} V^\alpha \cM = \cM$,
    \item $(V^k\cD_M)\cdot(V^\alpha_{\can}\cM)\subset V^{\alpha+k}_{\can}\cM$,
    \item For all $\alpha\in \dC$, the module $V^\alpha_{\can}\cM$ is $V^0\cD_M$-coherent,
    \item $t\cdot V_{\can}^\alpha\cM=V_{\can}^{\alpha+1}\cM$ for $\alpha>0$,
    \item The action of the operator $\partial_t\cdot t -\alpha$ on $\gr^\alpha_{V_{\can}} \cM := V_{\can}^\alpha \cM / V_{\can}^{>\alpha} \cM$ is nilpotent, where $V_{\can}^{>\alpha} \cM:=\cup_{\beta>\alpha} V_{\can}^\beta\cM$.
\end{enumerate}
$V^\bullet_{\can}\cM$ is called the \emph{canonical} $V$-filtration or Kashiwara-Malgrange filtration on $\cM$ with respect to $H$ or to $t$. It can be characterized by
$$
V^\alpha_{\can} \cM=\left\{
m\in \cM\,|\,
\textup{roots of }b^\cM_m(s) \,\subset\,[\alpha,\infty\}\right\},
$$
where $[\alpha,\infty):=\left\{c\in \dC\,|\,\alpha\leq c\right\}$.
\item
If $\cM = \cD_M/\cI$ is a cyclic $\cD_M$-module, where
$\cI$ is a sheaf of left ideals of $\cD_M$, then
we put for any $k\in \dZ$
$$
V^k_{\ind} \cM := \frac{V^k \cD_M}{\cI \cap V^k \cD_M}
\cong \frac{V^k \cD_M+\cI}{\cI},
$$
and we call the filtration $V^\bullet_{\ind} \cM$ the
\emph{induced} $V$-filtration on $\cM$.
In particular, if $\cM$ is holonomic, then $V^\bullet_{\ind} \cM$ has a (minimal and monic) Bernstein polynomial $b_{V_{\ind}^\bullet}^{\cM}(s)\in\dC[s]$.
\end{enumerate}
\end{deflemma}

We will mainly use the above definitions for the case where $H$ is the divisor $\{t=0\}\subset M=\dC_t\times X$ and where $\cM = i_{h,+}\cO_X(*D)$, $i_h$ being the graph embedding of a defining equation for a divisor $D\subset X$. Using a construction that goes back to Malgrange (see \cite{MalgrBernst}), one can find a cyclic generator for this module, so that it has an induced $V$-filtration. It is essentially well-known that the Bernstein polynomial of this filtration is given by the Bernstein polynomial $b_h(s)$ of the equation $h$ defining $D$, up to a change of variables. However, we recall the proof for the convenience of the reader. Notice also that this result holds quite generally for any divisor, and does not depend on freeness or any Koszul assumption.

We therefore let $D\subset X$ be a divisor defined locally at a point $p\in X$ by a reduced equation $h=0$, and we denote by $i_h:X\rightarrow\dC_t\times X$ its graph embedding. We put
$$
N(h)
:=
\left(i_{h,+}\cO_{X,p}(*D)\right)_{(0,p)},
$$
then $N(h)$ is a holonomic $\cD_{\dC_t\times X,(0,p)}$-module. Since $i_h$ is a closed immersion, we have that
$$
N(h)\cong (i_{h,*}\cO_X(*D))_{(0,p)}[\dd_t]
$$ as
$\cO_{\dC_t\times X,(0,p)}$-modules. It becomes an isomorphism of
$\cD_{\dC_t\times X,(0,p)}$-modules when equipping the right hand side with the
$\cD_{\dC_t\times X,(0,p)}$-structure given by
\label{page:RecallStructure}
\begin{equation}\label{eq:LeftActionGraph}
\begin{array}{ll}
a \cdot (m\, \partial_t^k) = (a\, m)\, \partial_t^k
,&
\partial_{x_i} \cdot (m\, \partial_t^k) =  (\partial_{x_i}\, m)\, \partial_t^k - (h'_{x_i}\, m)\,   \partial_t^{k+1}\;\forall\;i=1,\ldots,n, \\[5pt]
t\cdot (m\, \partial_t^k) = (h\, m)\, \partial_t^k - (k\, m)\, \partial_t^{k-1}, &
\partial_t \cdot (m\, \partial_t^k) =  m\, \partial_t^{k+1},
\end{array}
\end{equation}
for any $m\in\cO_{X,p}(*D)$ and any $a\in \cO_{X,p}$.
\begin{lemma}\label{lem:RootsElement1GraphEmbed}
In the above situation,
write  $j\in \dZ_{>0}$ for the negative of the smallest integer root of $b_h(s)$.
Then we have
$$b_h(-s-j)=b_{V_{\ind}^\bullet}^{N(h)}(s).$$
\end{lemma}
\begin{proof}
Following the constructions in \cite[\S~4]{MalgrBernst}, one can show that
$N(h) \cong (i_{h,*}\cO_X(*D))_{(0,p)}[s]\cdot h^s$, using that $t=h$ is invertible in $N(h)$ and substituting $-\dd_t t$ by $s$, where $h^s$ is a symbol on which tangent fields $\xi$ act as $\xi(h^s)=sh^{-1}\xi(h)\cdot h^s$.

The fact that $-j$ is the smallest integer root of $b_h(s)$ is equivalent to say that $\cO_{X,p}(*D)$ is generated as a $\cD_{X,p}$-module by $h^{-j}$. Therefore, $N(h)$ is generated by $h^{-j}\cdot h^s$, that we will write $h^{s-j}$ from now on.
Hence we can consider the filtration $V^\bullet_{\ind} N(h)= V^\bullet \cD_{\dC_t\times X,(0,p)}\cdot h^{s-j}$. The $V$-filtration on $\cD_{\dC\times X,(0,p)}$ can be written as $V^0 \cD_{\dC\times X,(0,p)}=\cD_{X,p}[\dd_tt]$, $V^k\cD_{\dC\times X,(0,p)}=  V^0 \cD_{\dC\times X,(0,p)} \cdot t^k$ for $k\geq 0$ and $V^k=\sum_{i=0}^{-k} V^0 \cD_{\dC\times X,(0,p)} \cdot \dd_t^i$ for $k\leq0$. Therefore, we obtain that
$$
V^k_{\ind}N(h)=\left\{\begin{array}{ll}
(i_{h,*}\cD_X)_{(0,p)}[s]h^{s-j} &\,k=0,\\
(i_{h,*}\cD_X)_{(0,p)}[s]h^{s-j+k}&\,k>0,\\
\sum_{i=0}^{-k} (i_{h,*}\cD_X)_{(0,p)}[s]s^ih^{s-j-i}&\, k<0,
\end{array}\right.
$$
where we have used that $s=-\dd_tt$ and $t=h$ in our alternative representation of $N(h)$.

As a consequence, we have isomorphisms
\begin{equation}\label{eq:VindFiltIso}
\frac{(i_{h,*}\cD_X)_{(0,p)}h^{s-j+k}}{(i_{h,*}\cD_X)_{(0,p)}h^{s-j+k+1}}\cong \frac{V_{\ind}^k N(h)}{V_{\ind}^{k+1} N(h)}
\end{equation}
for every $k\geq0$. On the other hand, for $k\leq0$, we have that
$$
V_{\ind}^{k-1}N(h)=(i_{h,*}\cD_X)_{(0,p)}[s]s^{-k+1}h^{s-j+k-1}+V_{\ind}^k N(h),
$$
so we obtain  surjections
$$
\varphi_k:(i_{h,*}\cD_X)_{(0,p)}[s]s^{-k+1}h^{s-j+k-1}\twoheadrightarrow \frac{V_{\ind}^{k-1}N(h)}{V_{\ind}^k N(h)}.
$$
Since
$$
(i_{h,*}\cD_X)_{(0,p)}[s]s^{-k+1}h^{s-j+k} \subset (i_{h,*}\cD_X)_{(0,p)}[s]s^{-k}h^{s-j-k}\subset V_{\ind}^k N(h),
$$
we see that $\varphi_k\left((i_{h,*}\cD_X)_{(0,p)}[s]s^{-k+1}h^{s-j+k}\right)=0$.
As a consequence, we obtain a surjection
$$
\overline{\varphi}_k:
\frac{(i_{h,*}\cD_X)_{(0,p)}[s]s^{-k+1}h^{s-j+k-1}}{(i_{h,*}\cD_X)_{(0,p)}[s]s^{-k+1}h^{s-j+k}}
\twoheadrightarrow
\frac{V_{\ind}^{k-1}N(h)}{V_{\ind}^k N(h)}
$$
for any $k\leq 0$.

We finally obtain surjections
\begin{equation}\label{eq:VindFiltSurj}
\frac{(i_{h,*}\cD_X)_{(0,p)}[s]h^{s-j+k-1}}{(i_{h,*}\cD_X)_{(0,p)}[s]h^{s-j+k}}\cong \frac{(i_{h,*}\cD_X)_{(0,p)}[s]s^{-k+1}h^{s-j+k-1}}{(i_{h,*}\cD_X)_{(0,p)}[s]s^{-k+1}h^{s-j+k}} \stackrel{\overline{\varphi}_k}{\twoheadrightarrow}
\frac{V_{\ind}^{k-1}N(h)}{V_{\ind}^k N(h)}
\end{equation}
for any $k\leq0$, where the first isomorphism holds because
for any $l\geq 0$ and any $l'\in \dZ$,
$s$ is a non-zero divisor on the modules
$(i_{h,*}\cD_X)_{(0,p)}[s] s^l h^{s+l'}$.

By the Bernstein functional equation, we know that $b_h(s)$ sends $h^s$ to $\cD_{X,p}[s]h^{s+1}$, hence
$b_h(-\dd_tt-j+k)$ annihilates
the left-hand side of equation \eqref{eq:VindFiltIso}
and $b_h(-\dd_tt-j+k-1)$ annihilates the left-hand side of the surjection \eqref{eq:VindFiltSurj}, so they annihilate the respective right-hand sides as well.
In other words,
$b_h(-\dd_tt-j+k)$ kills the quotient
$V_{\ind}^k N(h)/V_{\ind}^{k+1} N(h)$ for any $k\in \dZ$.
This means precisely that $b_h(-s-j)=b_{V_{\ind}^\bullet}^{N(h)}(s)$, as desired.

\end{proof}

The next result gives a precise description of the canonical $V$-filtration for $N(h)$ for divisors satisfying the additional assumption that the roots of $b_h$ are included in the open interval $(-2,0)$.
\begin{proposition}\label{prop:PreciseDescrCanVFilt}
Let $X$, $D$, $h$ and $N(h)$ be as in the previous lemma. Assume moreover that the roots of $b_h(s)$ are contained in $(-2,0)$. Then for all $k\in \dZ$, we have
\begin{equation}\label{eq:CanVFilt}
V^k_{\can} N(h) =
V_{\ind}^{k+1}N(h) + \prod_{\alpha_i\in B'_h} (\partial_t t-k +\alpha_i)^{l_i} V_{\ind}^k N(h);
\end{equation}
recall from the introduction (see Formula \eqref{eq:RootsMinusOneIntro}) that $B'_h:=\left\{\alpha_i\in \dQ \cap (0,1)\,|\, b_h(\alpha_i-1)=0\right\}$.

In particular, we have
$$
V^k_{\can} N(h) \subset V^k_{\ind} N(h).
$$
for all $k\in \dZ$.
\end{proposition}
\begin{proof}
Since the roots of $b_h(s)$ are in $(-2,0)$, we know thanks to Lemma \ref{lem:RootsElement1GraphEmbed} that the roots of $b^{N(h)}_{V_{\ind}^\bullet}(s)=b_h(-s-1)$ are contained in $(-1,1)$. Now we argue as in the proof of \cite[Theorem I]{KashiwaraVanishingCycles} (cf. also \cite[Proposition 4.2-6]{MebkhMais}):
Let $\lambda_1< \ldots<\lambda_c$ be the set of roots of $b^{N(h)}_{V_{\ind}^\bullet}(s)$ with $-1<\lambda_i<0$, with respective multiplicities $l_1,\ldots,l_c$. We write
$$
b^{N(h)}_{V_{\ind}^\bullet}(s)=\prod_{i=1}^c(s-\lambda_i)^{l_i}\cdot b'(s)
$$
with $b'(\lambda_i)\neq 0$ for all $i\in\{1,\ldots,c\}$.
Then for each $k\in \dZ$ put
\begin{equation}\label{eq:DefOverlineV}
\overline{V}^k N(h) := V_{\ind}^{k+1}N(h) + \prod_{i=1}^c(\partial_t t - k -\lambda_i)^{l_i} V_{\ind}^k N(h) \subset V^k_{\ind} N(h).
\end{equation}
Then $\overline{V}^\bullet N(h)$ is a good $V$-filtration on $N(h)$ and moreover,
$\prod_{i=1}^c(s-(\lambda_i+1))^{l_i}\cdot b'(s)$ is a Bernstein polynomial for $\overline{V}^\bullet N(h)$,
namely we have that
$$
\begin{array}{rcl}
\D \prod_{i=1}^c(\partial_t t-(\lambda_i+1)-k))^{l_i}\cdot b'(\partial_t t-k) \left[ V^{k+1}_{\ind} N(h)\right]  & = &
\D b'(\partial_t t-k) \prod_{i=1}^c(\partial_t t-(\lambda_i+1)-k))^{l_i} V^{k+1}_{\ind} N(h)  \\ \\
\D \stackrel{*}{\subset} b'(\partial_t t-k) \overline{V}^{k+1} N(h) & \subset & \D \overline{V}^{k+1} N(h)
\end{array}
$$
and
$$
\begin{array}{c}
\D \prod_{i=1}^c(\partial_t t-(\lambda_i+1)-k)^{l_i}\cdot b'(\partial_t t-k) \left[\prod_{i=1}^c(\partial_t t - k -\lambda_i)^{l_i} V_{\ind}^k N(h) \right]  \\ \\
=\D \prod_{i=1}^c(\partial_t t-(\lambda_i+1)-k)^{l_i}\cdot b^{N(h)}_{V_{\ind}^\bullet} (\partial_t t -k) V_{\ind}^k N(h)
\\ \\
\D \subset \prod_{i=1}^c(\partial_t t-(\lambda_i+1)-k)^{l_i} V^{k+1}_{\ind} N(h)  \stackrel{*}{\subset}  \overline{V}^{k+1} N(h),
\end{array}
$$
where the two inclusions marked as $\stackrel{*}{\subset}$ are due to the definition of the filtration
$\overline{V}^\bullet N(h)$, i.e., due to formula \eqref{eq:DefOverlineV}. Hence we obtain that
$$
\left(\prod_{i=1}^c(\partial_t t-(\lambda_i+1)-k)^{l_i}\cdot b'(\partial_t t -k)\right) \left(\overline{V}^k N(h)\right) \subset \overline{V}^{k+1} N(h),
$$
that is, $\prod_{i=1}^c(s-(\lambda_i+1))^{l_i}\cdot b'(s)$ is a Bernstein polynomial for $\overline{V}^\bullet N(h)$.
We conclude that $\overline{V}^\bullet N(h)$ is
a good $V$-filtration on $N(h)$ such that its Bernstein polynomial has all its roots in the interval $[0,1)$.

Now by \cite[Proposition 4.3-5]{MebkhMais} we conclude that
$$
V_{\can}^k N(h) = \overline{V}^k N(h).
$$
holds for all $k\in \dZ$, and this shows Formula \eqref{eq:CanVFilt}.

\end{proof}

\begin{remark}
Without the assumption that the roots of $b_h(s)$ are in $(-2,0)$ (i.e.,
for general $h\in \cO_{X,p}$) the previous result is no longer true in general. Namely, since in general the roots of $b_h(s)$ are negative rational numbers, we know by Lemma \ref{lem:RootsElement1GraphEmbed} that the roots of $b_{V_{\ind}^\bullet}^{N(h)}$ are contained in $(-j,\infty)$, where $-j$ is the smallest integer root of $b_h(s)$. It can thus happen that $b_{V_{\ind}^\bullet}^{N(h)}$ has roots bigger than $1$, and this prevents the inclusion
$V^k_{\can} N(h) \subset V^k_{\ind} N(h)$ to hold in general.
The roots of $b_h(s)$ are contained in $(-2,0)$
for the class of SK free divisors (see point 3 of Proposition \ref{prop:SummarySKLCT} and \cite[Theorem 4.1]{nar_symmetry_BS}), but it also holds for instance
when $h$ defines a central hyperplane arrangement by \cite[Theorem 1]{SaitoBSHyperArrange} (which is not necessarily a free divisor).
\end{remark}

In the remaining part of this section, we restrict our attention to SK free divisors. We study more in detail the induced $V$-filtration
on the module $N(h)$, since as we have seen, it helps to understand (at least the integer steps of) the canonical $V$-filtration on that module. In particular,
we will prove (see Proposition \ref{prop:CompatibilityVandF} below) a compatibility property between the induced $V$-filtration
and the order filtration on that module (which is similar to the order filtration $F^{\ord}_\bullet\cO_X(*D)$ studied in section \ref{sec:FFilt}). Its statement is similar to
\cite[Proposition 4.9]{ReiSe3}), although the presentation of the proof is slightly different.

We first give a concrete cyclic presentation of the graph embedding module for equations defining an SK free divisor,
starting from the presentation  for $\cO_X(*D)$ discussed in the previous section (see Equation \eqref{eq:CyclicPresent-1} above).
\begin{lemma}\label{lem:RepresentGraphEmbed}
Let $D\subset X$ be a free divisor and let $p\in D$ be such that $D$ is strongly Koszul at $p$. Let $h\in \cO_{X,p}$ be a reduced equation for $D$ near $p$. Let $i_h:X\hookrightarrow \dC_t\times X$, $x\mapsto (h(x),x)$ be the graph embedding of $h$, then we have an isomorphism
of $\cD_{\dC_t\times X,(0,p)}$-modules
$$
\D (i_{h,+}\cO_X(*D))_{(0,p)} \cong \D \frac{\cD_{\dC_t\times X,(0,p)}}{(t-h,\delta_1,\ldots,\delta_{n-1},\chi+\partial_tt+1)}.
$$
We still denote this module (resp. this cyclic presentation) by $N(h)$.
\end{lemma}
\begin{proof}
This is a consequence of the logarithmic comparison theorem, as stated in point 2 of Proposition \ref{prop:SummarySKLCT}. Namely, we know that
$$\cO_{X,p}(*D) \cong \frac{\cD_{X,p}}{\left(\delta_1,\ldots,\delta_{n-1},\chi+1\right)}.$$
Now we consider the direct image of this object under the graph embedding.
Then we have
$$
(i_{h,+}\cO_X(*D))_{\dC_t\times X,(0,p)}
\cong
i_{h,+}\left(\frac{\cD_{X,p}}{\left(\delta_1,\ldots,\delta_{n-1},\chi+1\right)}\right)
=\frac{\cD_{\dC_t\times X,(0,p)}}{(t-h,\delta_1,\ldots,\delta_{n-1},\chi+h\partial_t+1)},
$$
by taking into account that $\delta_i(h)=0$ for every $i=1,\ldots,n-1$ and $\chi(h)=h$. The claim now follows from the fact that in $N(h)$ we have $h\dd_t=\dd_t h=\dd_t t$
(here and later on, we sometimes denote operators and their
classes in a cyclic quotient module by the same symbol).
\end{proof}

We will be later interested in calculating the Hodge filtration on a mixed Hodge module
which has $i_{h,+}\cO_{X.x}(*D)$ as underlying $\cD_{\dC_t\times X,(0,p)}$-module. For that purpose, we consider the filtration
$F_\bullet^{\ord} N(h)$ which is induced on $N(h)$ by the filtration $F_\bullet \cD_{\dC_t\times X,(0,p)}$
by the order of differential operators.

For the sake of brevity, let us write $\cO:=\cO_{\dC_t\times X,(0,p)}$ and $\cD:=\cD_{\dC_t\times X,(0,p)}$. Consider the ring
$$
\cG:=\cO_{\dC_t\times X,(0,p)}[T,X_1,\ldots,X_n]=\dC\{t,x_1,\ldots,x_n\}[T,X_1,\ldots,X_n].
$$
Then we have $\cG=\gr_\bullet^F \cD$ (where, as before, $F_\bullet \cD$ is the filtration on $\cD$ by the order). Obviously, $\cG$ is graded by the degree of the variables $T,X_1,\ldots,X_n$, and we write $\cG_l$ for the degree $l$ part.
\begin{lemma}\label{lem:Involutive_F-Filt}
Let $D\subset X$ be a free divisor and let $p\in D$ such that
$D$ is strongly Koszul at $p$ and locally defined by some $h\in\cO_{X,p}$.
Write
$$
I(h):=\left(t-h,\delta_1,\ldots,\delta_{n-1}, \chi+\partial_t t +1\right),
$$
so that $N(h)=\cD / I(h)$.
Then the set $\{t-h,\delta_1,\ldots,\delta_{n-1}, \chi+\partial_t t +1\}$ is an involutive basis of the ideal $I(h)$, that is, we have
the equality
$$
\sigma(I(h)) =
\left(
t-h,\sigma(\delta_1),\ldots,\sigma(\delta_{n-1}), \sigma(\chi+\partial_t t +1)=\sigma(\chi)+T\cdot t
\right)
$$
of ideals in $\cG$. Recall that we denote for an element $P\in \cD$ its symbol in $\cG$ by $\sigma(P)$ and by $\sigma(I(h))$ the ideal $\{\sigma(P)\,|\,P\in I(h)\}$ in $\cG$.
\end{lemma}
\begin{proof}
The proof basically follows the argument in \cite[Proposition 4.1.2]{calde_ens} once we know that symbols of the generators of $I(h)$, i.e., the elements
$$
t-h,\sigma(\delta_1),\ldots,\sigma(\delta_{n-1}),\sigma(\chi)+T\cdot t
$$
form a regular sequence in the ring $\cG$. Notice also that Lemma \ref{lem:reg-seq-Groebner} below is a similar statement (in a commutative ring though), and we will give some indications of the proof there.

In order to show this regularity statement, first recall that
since $D$ is strongly Koszul at $p$ by assumption, it is in particular Koszul (see Definition \ref{defi:SK} and the subsequent remarks). Hence we know that
$$
\sigma(\delta_1),\ldots,\sigma(\delta_{n-1}),\sigma(\chi)
$$
is a regular sequence in the ring $\gr_\bullet^F(\cD_{X,p})=\dC\{x_1,\ldots,x_n\}[X_1,\ldots,X_n]$.
Then clearly the sequence
$$
\sigma(\delta_1),\ldots,\sigma(\delta_{n-1}),\sigma(\chi),T
$$
is regular in the ring $\dC\{x_1,\ldots,x_n\}[T,X_1,\ldots,X_n]$.
On the other hand, we have the following equality of ideals of
$\dC\{x_1,\ldots,x_n\}[T,X_1,\ldots,X_n]$:
$$
\left(\sigma(\delta_1),\ldots,\sigma(\delta_{n-1}),\sigma(\chi),T\right)
=
\left(\sigma(\delta_1),\ldots,\sigma(\delta_{n-1}),\sigma(\chi)+T\cdot h,T\right)
$$
so that also the sequence
$$
\sigma(\delta_1),\ldots,\sigma(\delta_{n-1}),\sigma(\chi)+T\cdot h,T
$$
is regular in $\dC\{x_1,\ldots,x_n\}[T,X_1,\ldots,X_n]$. Since we have an isomorphism of rings
$$
\dC\{x_1,\ldots,x_n\}[T,X_1,\ldots,X_n] \longrightarrow
\dC\{t,x_1,\ldots,x_n\}[T,X_1,\ldots,X_n]/(t-h),
$$
we obtain that the sequence
$$
t-h,\sigma(\delta_1),\ldots,\sigma(\delta_{n-1}),\sigma(\chi)+T\cdot h,T
$$
is regular in $\cG=\dC\{t,x_1,\ldots,x_n\}[T,X_1,\ldots,X_n]$. Then again by equality of ideals
$$
\left(
t-h,\sigma(\delta_1),\ldots,\sigma(\delta_{n-1}),\sigma(\chi)+T\cdot h,T
\right)
=
\left(
t-h,\sigma(\delta_1),\ldots,\sigma(\delta_{n-1}),\sigma(\chi)+T\cdot t,T
\right)
$$
in $\cG$ we get that
$$
t-h,\sigma(\delta_1),\ldots,\sigma(\delta_{n-1}),\sigma(\chi)+T\cdot t,T
$$
is a regular sequence in $\cG$,
but then also the shorter sequence
$$
t-h,\sigma(\delta_1),\ldots,\sigma(\delta_{n-1}),\sigma(\chi)+T\cdot t
$$
must be regular in $\cG$.
\end{proof}

We denote by $\widetilde{V}^\bullet \cG$ the filtration
induced by $V^\bullet \cD$ (the $V$-filtration on $\cD$ with respect to the divisor
$\{t=0\}\subset \dC_t\times X$) on the ring $\cG$. It can be
described as
$$
\widetilde{V}^d \cG = \left\{ \sum_{\alpha,k,\beta,l} a_{\alpha,k,\beta,l}\, x^\alpha t^k X^\beta T^l \in \cG
\;
\left|
\;
a_{\underline{\alpha},k,\underline{\beta},l}\in \dC,\,k-l\geq d
\right.
\right\}.
$$
For any $f\in \cG$, we write $\ord^{\widetilde{V}}$ for the maximal $b\in \dZ$ such that
$f\in \widetilde{V}^b \cG$. The graded ring of $\cG$ with respect to the filtration $\widetilde{V}$ looks quite similar to $\cG$ itself, namely, we have
$$
\gr^\bullet_{\widetilde{V}} \cG \cong \dC\{x_1,\ldots,x_n\}[t,T,X_1,\ldots,X_n],
$$
since $\widetilde{V}^\bullet$ induces the $t$-adic filtration
on the ring $\dC\{t,x_1,\ldots,x_n\}$.

As usual, for $f\in \cG$ we denote by $\sigma^{\widetilde{V}}(f)\in \gr^\bullet_{\widetilde{V}} \cG$ the symbol with respect to $\widetilde{V}$, i.e.,
the class of $f$ in $\gr^\bullet_{\widetilde{V}} \cG=\oplus_{l\in \dZ} \gr^l_{\widetilde{V}} \cG$.
Now consider an ideal $I=(f_1,\ldots,f_k)\subset \cG$.
We write
$$
\sigma^{\widetilde{V}}(I):=
\left\{\sigma^{\widetilde{V}}(f)\;|\; f\in I\right\}.
$$
Then we have the following fact (which is analogous to the statement of Lemma \ref{lem:Involutive_F-Filt} above, and in fact holds in a more general setting for certain filtered rings, although we do not consider such a generality here).
\begin{lemma}\label{lem:reg-seq-Groebner}
Suppose that $\sigma^{\widetilde{V}}(f_1),\ldots,
\sigma^{\widetilde{V}}(f_k)$ is a regular sequence
in $\gr^\bullet_{\widetilde{V}} \cG$. Then we have the equality
$$
\sigma^{\widetilde{V}}(I) =
\left(\sigma^{\widetilde{V}}(f_1),\ldots,
\sigma^{\widetilde{V}}(f_k)\right)
$$
of ideals in $\gr^\bullet_{\widetilde{V}} \cG$.
\end{lemma}
\begin{proof}
We follow the argument given (in a more algebraic situation though) in \cite[Proposition 4.3.2]{SST}. We obviously have $\sigma^{\widetilde{V}}(I) \supset
\left(\sigma^{\widetilde{V}}(f_1),\ldots,
\sigma^{\widetilde{V}}(f_k)\right)$ and we need
to show the converse inclusion. Assume that it does not hold, that is, take
$f\in I$ (so that $\sigma^{\widetilde{V}}(f)\in\sigma^{\widetilde{V}}(I)$)
such that $\sigma^{\widetilde{V}}(f) \notin \left(\sigma^{\widetilde{V}}(f_1),\ldots,
\sigma^{\widetilde{V}}(f_k)\right)$. For a representation
\begin{equation}\label{eq:NonGroebRep}
f=\sum_{i=1}^k g_i\cdot f_i
\end{equation}
write
$$
o(\underline{g}):=\min_{i=1,\ldots,k}\left(\ord^{\widetilde{V}}(g_i)+\ord^{\widetilde{V}}(f_i)\right),
$$
then clearly $o(\underline{g})\leq \ord^{\widetilde{V}}(f)$.
Choose a representation $f=\sum_{i=1}^k g_i\cdot f_i$ such that $o(\underline{g})$ is maximal.
Since $\sigma^{\widetilde{V}}(f) \notin \left(\sigma^{\widetilde{V}}(f_1),\ldots,
\sigma^{\widetilde{V}}(f_k)\right)$, we have $\ord^{\widetilde{V}}(f) > o(\underline{g})$.
We conclude that
$$
\sum_{i\,:\,\ord^{\widetilde{V}}(g_i)+\ord^{\widetilde{V}}(f_i)=o(\underline{g})} \sigma^{\widetilde{V}}(g_i) \cdot  \sigma^{\widetilde{V}}(f_i) =0.
$$
Now by assumption the symbols $\sigma^{\widetilde{V}}(f_1),\ldots,\sigma^{\widetilde{V}}(f_k)$ form a regular sequence in $\gr^\bullet_{\widetilde{V}}\cG$, and this implies that the module of syzygies between these elements of $\cG$ is generated by the so-called
Koszul relations, i.e.,
$$
\left(-\sigma^{\widetilde{V}}(f_k)\right) \cdot \sigma^{\widetilde{V}}(f_j)
+
\left(\sigma^{\widetilde{V}}(f_j)\right) \cdot \sigma^{\widetilde{V}}(f_k)=0.
$$
In other words, we have an equality
$$
\sum_{i\,:\,\ord^{\widetilde{V}}(g_i)+\ord^{\widetilde{V}}(f_i)=o(\underline{g})} \sigma^{\widetilde{V}}(g_i) \cdot  e_i =
\sum_{1\leq j<l\leq k} h_{jl}\left(\sigma^{\widetilde{V}}(f_j)e_l-\sigma^{\widetilde{V}}(f_l)e_j\right)
$$
of elements of $\cG^k$ (where $e_i$ is the $i$-th canonical generator of $\cG^k$). Reordering the right hand side of the above equation yields
$$
\sum_{i\,:\,\ord^{\widetilde{V}}(g_i)+\ord^{\widetilde{V}}(f_i)=o(\underline{g})} \sigma^{\widetilde{V}}(g_i) \cdot  e_i =
\sum_{i=1}^k\left(
\sum_{l<i} h_{li}\sigma^{\widetilde{V}}(f_l)
-\sum_{i<j} h_{ij}\sigma^{\widetilde{V}}(f_j)\right)e_i,
$$
and hence
$$
\sigma^{\widetilde{V}}(g_i)-
\left(
\sum_{l<i} h_{li}\sigma^{\widetilde{V}}(f_l)
-\sum_{i<j} h_{ij}\sigma^{\widetilde{V}}(f_j)\right)=0
$$
for all $i\in\{1,\ldots,k\}$ such that $\ord^{\widetilde{V}}(g_i)+\ord^{\widetilde{V}}(f_i)=o(\underline{g})$. Hence if we replace those $g_i$ in equation \eqref{eq:NonGroebRep} by $g'_i:=g_i-\left(\sum_{l<i} h_{li} f_l
-\sum_{i<j} h_{ij}f_j\right)$ (and collect all the $f_i$), we obtain a new expression $f=\sum_{i=1}^k g'_i f_i$ such that
$$
\min_{i=1,\ldots,k}\left(\ord^{\widetilde{V}}(g'_i)+\ord^{\widetilde{V}}(f_i)\right) > o(\underline{g})
$$
which contradicts the above choice of a relation $\underline{g}$ with maximal $o(\underline{g})$. Hence
we must have that $\sigma^{\widetilde{V}}(f) \in \left(\sigma^{\widetilde{V}}(f_1),\ldots,
\sigma^{\widetilde{V}}(f_k)\right)$ and so
$\sigma^{\widetilde{V}}(I)=\left(\sigma^{\widetilde{V}}(f_1),\ldots,
\sigma^{\widetilde{V}}(f_k)\right)$, as required.
\end{proof}

We apply the previous lemma in the situation where $I$ is given as the ideal $\sigma(I(h))\subset \cG$ (recall that $\sigma$ denotes the usual symbol with respect to the order filtration $F_\bullet \cD$).
For notational convenience, put
$G_0:=t-h$, $G_1:=\delta_1,\ldots, G_{n-1}:=\delta_{n-1}$ and $G_n:=\chi+\partial_t t +1$, so that $I(h)=(G_0,\ldots,G_n) \subset \cD$. Then we obtain the following consequence.

\begin{corollary}\label{cor:Involutive_V-Filt}
Let $D\subset X$ be a free divisor and let $x\in D$ such that
$D$ is strongly Koszul at $x$ and locally defined by some $h\in\cO_{X,p}$. Then the set
$$
\{\sigma(G_0),\ldots,\sigma(G_n)\}=\{t-h,\sigma(\delta_1),\ldots,\sigma(\delta_{n-1}), \sigma(\chi)+T\cdot t\}
$$ is an involutive basis of the ideal
$\sigma(I(h))\subset \cG$ with respect to the filtration $\widetilde{V}^\bullet \cG$
induced from the filtration $V^\bullet \cD$, that is, we have
$$
\sigma^{\widetilde{V}}(\sigma(I(h)) =
\gr^\bullet_{\widetilde{V}}\cG\cdot(\sigma^{\widetilde{V}}(\sigma(G_0)),\ldots,\sigma^{\widetilde{V}}(\sigma(G_n)))=
\gr^\bullet_{\widetilde{V}}\cG\cdot\left(
h,\sigma^{\widetilde{V}}(\sigma(\delta_1)),\ldots,\sigma^{\widetilde{V}}(\sigma(\delta_{n-1})),\sigma^{\widetilde{V}}(\sigma(\chi))+T\cdot t\right).
$$

As a consequence, any $f\in \sigma(I(h))\subset \cG$ has a standard representation with respect to $\widetilde{V}^\bullet$, that is, there are elements $k_0,k_1,\ldots,k_n\in \cG$ with $f=\sum_{i=0}^n k_i\cdot \sigma(G_i)
=k_0\cdot(t-h)+\sum_{i=1}^{n-1} k_i \cdot \sigma(\delta_i)+k_n\cdot (\sigma(\chi)+T\cdot t)$
such that $\ord^{\widetilde{V}}(k_i)+\ord^{\widetilde{V}}(\sigma(G_i))
\geq \ord^{\widetilde{V}}(f)$ holds for all $i\in\{0,\ldots,n\}$.
\end{corollary}
For the last statement, we remind the reader that the filtration $\widetilde{V}^\bullet \cG$ is descending, and that consequently, for $g\in \cG$,  $\ord^{\widetilde{V}}(g)$ denotes the maximum of all $l$ such that $g\in \widetilde{V}^l \cG$.
\begin{proof}
Consider the ring extension
$$
\begin{array}{rcl}
\Gr^T_\bullet(\cD[s])\cong \dC\{x_1,\ldots,x_n\}[s,X_1,\ldots,X_n]
&
\hookrightarrow
&
\Gr^\bullet_{\widetilde{V}}(\cG) \cong
\dC\{x_1,\ldots,x_n\}[t,T,X_1,\ldots,X_n] \\ \\
s & \longmapsto & -T\cdot t
\end{array}
$$
then clearly
$\Gr^\bullet_{\widetilde{V}}(\cG)$ is a flat $\Gr^T_\bullet(\cD[s])$-module,
since for any ring $R$, we have that $R[t,T]$ is a flat (and even free) $R[-T\cdot t]$-module.

From the very definition of the strongly Koszul property (see again Definition \ref{defi:SK}), we know that
$$
h,\sigma(\delta_1),\ldots,\sigma(\delta_{n-1}),\sigma(\chi)-s
$$
is a regular sequence in $\Gr^T_\bullet(\cD[s])$. Now
notice that
the elements $\sigma(\delta_1),\ldots,\sigma(\delta_{n-1}), \sigma(\chi)$ lie in the ring $\dC\{x_1,\ldots,x_n\}[X_1,\ldots,X_n]$, which is a subring of both $\cG$ and of $\Gr^\bullet_{\widetilde{V}}(\cG)$.
Clearly, for $p\in\dC\{x_1,\ldots,x_n\}[X_1,\ldots,X_n]$ we have $\sigma^{\widetilde{V}}(p)=p$, therefore, we know that
$$
h,\sigma^{\widetilde{V}}(\sigma(\delta_1)),\ldots,\sigma^{\widetilde{V}}(\sigma(\delta_{n-1})),\sigma^{\widetilde{V}}(\sigma(\chi))-s
$$
is a regular sequence in $\Gr_\bullet^T(\cD[s])$. Since the ring extension $\Gr_\bullet^T(\cD[s])\subset \Gr^\bullet_{\widetilde{V}}(\cG)$ is flat, it follows (see, e.g., \cite[Proposition 9.6.7 and remark right after it]{BourbakiAlgHom}) that the elements
$$
h,\sigma^{\widetilde{V}}(\sigma(\delta_1)),\ldots,\sigma^{\widetilde{V}}(\sigma(\delta_{n-1})),\sigma^{\widetilde{V}}(\sigma(\chi))+T\cdot t
$$
form a regular sequence in $\Gr^\bullet_{\widetilde{V}}(\cG)$.
Now the assertion of the Corollary follows from Lemma \ref{lem:reg-seq-Groebner} above.
\end{proof}

The next statement is a rather direct consequence of the
previous results. We include it here since it will be used in section \ref{sec:Examples} where we discuss techniques for computation of Hodge ideals.
\begin{corollary}\label{cor:generators_I(h)_V-filt}
Let $D\subset X$ be a free divisor, let $p\in D$ such that $D$ is strongly Koszul at $p$ and locally defined by some $h\in\cO_{X,p}$. Then
$$
I(h)\cap V^0\cD=\left(t-h,\delta_1,\ldots,\delta_{n-1}, \chi+\partial_t t +1\right)\cap V^0\cD=V^0\cD( t-h,\delta_1,\ldots,\delta_{n-1},\chi+\dd_t t+1).
$$
\end{corollary}
\begin{proof}
Since the generators of $I(h)$ belong all to $V^0\cD$, which is a ring, the inclusion $I(h)\cap V^0\cD\supset V^0\cD( t-h,\delta_1,\ldots,\delta_{n-1},\chi+\dd_t t+1)$ is trivial. Let us show the reverse one.

Let $P$ be an operator in $I(h)\cap V^0\cD$ and let us prove that it lies within $V^0\cD( t-h,\delta_1,\ldots,\delta_{n-1},\chi+\dd_tt+1)$ by induction on the order of $P$ with respect to  $F_\bullet  \cD$, written $\ord^F(P)$.

If $\ord^F(P)=-1$, the claim is trivial, for then $P=0$. Let us now assume that the statement is true for all $Q\in F_{d-1}\cD$ for a given non-negative integer $d$, and let us prove it for any $P\in F_d\cD$. Since $P\in I(h)$, we have the expression
$$P=A(t-h)+\sum_{i=1}^{n-1} B_i\delta_i+C(\chi+\dd_tt+1).$$
Thanks to Lemma \ref{lem:Involutive_F-Filt}, we know that $A\in F_d\cD$ and the $B_i$ and $C$ are in $F_{d-1}\cD$ and also that $\sigma(P)$, which belongs to $\wt V^0\cG$, can be written as
$$\sigma(P)=\sigma(A)(t-h)+\sum_{i=1}^{n-1}\sigma(B_i)\sigma(\delta_i)+\sigma(C)(\sigma(\chi)+Tt).$$
On the other hand, we know after Corollary \ref{cor:Involutive_V-Filt} that we can choose $\tilde{A}\in\gr^F_d(\cD)\cap\wt V^0\cG$ and some $\tilde B_i,\tilde C\in\gr^F_{d-1}(\cD)\cap \wt V^0\cG$ such that
$$\sigma(P)=\tilde A(t-h)+\sum_{i=1}^{n-1}\tilde B_i\sigma(\delta_i)+\tilde C(\sigma(\chi)+Tt).$$
Then we have that
$$
\varsigma:=(\sigma(A)-\tilde A,\sigma(B_1)-\tilde B_1,\ldots,\sigma(B_{n-1})-\tilde B_{n-1},\sigma(C)-\tilde C) \in \cG^{n+1}
$$
is a syzygy of the tuple $(t-h,\sigma(\delta_1),\ldots,\sigma(\delta_{n-1}),\sigma(\tilde \chi)+Tt).$ By Lemma \ref{lem:Involutive_F-Filt} again, since this tuple is a regular sequence in $\cG$, $\varsigma$ will be a sum of Koszul syzygies (i.e., those of the form $a_ie_j-a_je_i\in \cG^{n+1}$ for some elements $a_i, a_j$ of the regular sequence, $e_i$ and $e_j$ being the corresponding unit vectors).

Now notice that for all $k,l$, the canonical map
$$
F_l \cD \cap V^k \cD \longrightarrow
\widetilde{V}^k \cG_l
$$
is surjective. Hence, we can choose $A'\in F_d\cD\cap V^0\cD$ and $B'_i,C'\in F_{d-1}\cD\cap V^0\cD$ such that $\sigma(A')=\tilde A$, $\sigma(B'_i)=\tilde B_i$ for every $i$ and $\sigma(C')=\tilde C$. Then the tuple $(A-A',B_1-B'_1,\ldots,B_{n-1}-B'_{n-1},C-C')$ can be written as a sum of Spencer syzygies  of $(t-h,\delta_1,\ldots,\delta_{n-1},\chi+\dd_tt+1)$ (that is, the respective lifts of the Koszul ones in $\cG$ that come from the expression of the $[G_i,G_j]$, for any $i,j$, as a linear combination of the form $\sum_k\alpha_kG_k$ with coefficients in $\cO$)  plus another tuple $(A'',B''_1,\ldots,B''_{n-1},C'')$, where now $A''\in F_{d-1}\cD$ and $B''_i,C''\in F_{d-2}\cD$.

Let us call now
$$P'':=A''(t-h)+\sum_{i=1}^{n-1} B''_i\delta_i+C''(\chi+\dd_tt+1).$$
Consequently,
$$P=A'(t-h)+\sum_{i=1}^{n-1} B'_i\delta_i+C'(\chi+\dd_tt+1)+P''.$$
Summing up, we obtain that $P''\in F_{d-1}\cD\cap V^0\cD\cap I(h)$, so by the induction hypothesis, it can be written as a linear combination of $t-h,\delta_1,\ldots,\delta_{n-1},\chi+\dd_tt+1$ with coefficients in $V^0\cD$ and thus $P$ too.
\end{proof}

Now we can show the following result expressing a certain compatibility between
the induced $V$-filtration and the order filtration on $N(h)$.

\begin{proposition}\label{prop:CompatibilityVandF}
For all  $k,l\in \dZ$, the canonical morphism
$$
F_l \cD  \cap V^k \cD
\longrightarrow
F^{\ord}_l N(h) \cap V^k_{\ind} N(h)
$$
is surjective.
\end{proposition}
\begin{proof}
We reformulate the surjectivity  we want to prove as a compatibility property between \emph{three} filtrations on $\cD$, the first two being $F_\bullet \cD$ and $V^\bullet \cD$. The third one is defined (somewhat artificially) as
$$
J_r \cD :=
\left\{
\begin{array}{rcl}
I(h) & \quad & \forall r< 0,\\ \\
\cD & \quad & \forall r\geq 0.
\end{array}
\right.
$$
It is clear that the statement of the proposition is equivalent to the surjectivity of the map
\begin{equation}\label{eq:SurjSaito1}
F_l \cD \cap V^k \cD \cap J_r \cD \longrightarrow {'\!}F_l \gr_r^J \cD \cap {'\!}V^k \gr_r^J \cD
\end{equation}
for $r=0$, where we denote by ${'\!}F_\bullet$ resp. by ${'\!}V^\bullet$ the filtrations induced by $F_\bullet \cD$ resp. $V^\bullet \cD$ on $\gr_r^J\cD$
(notice that for $r=0$, these are nothing but $F^{\ord}_\bullet N(h)$ resp. $V^\bullet_{\ind}N(h)$). Since both the $J$- and the $F$-filtration are exhaustive, we can now apply \cite[Corollaire 1.2.14]{Saito1} from which we conclude that surjectivity of the map \eqref{eq:SurjSaito1} holds (for all $l,k,r$) if and only if the map
\begin{equation}\label{eq:SurjSaito2}
F_l \cD \cap V^k \cD \cap J_r \cD \longrightarrow  \widetilde{V}^k \gr_l^F \cD \cap \widetilde{J}_r \gr_l^F \cD =
\widetilde{V}^k \cG_l \cap \widetilde{J}_r \cG_l.
\end{equation}
is surjective for all $l,k,r$. Again, $\widetilde{J}_\bullet \cG_l$ is the filtration induced
by $J_\bullet \cD$ on $\cG_l$, i.e., $\widetilde{J}_r \cG_l=\cG_l$ if $r\geq 0$ and $\widetilde{J}_r \cG_l=\sigma_l(I(h)\cap F_l\cD)$ if $r< 0$, where, as usual, $\sigma_l:F_l\cD \twoheadrightarrow \cG_l$ denotes the map sending an operator of degree $l$ to its symbol in $\cG_l$. However, the surjectivity of \eqref{eq:SurjSaito2} is a nontrivial condition only for $r<0$ (and it is the same condition for all negative $r$): For $r\geq 0$, it means  that the map
$F_l \cD \cap V^k \cD \longrightarrow
\widetilde{V}^k \cG_l$
is surjective, which is true, as we have already noticed in the proof of Corollary \ref{cor:generators_I(h)_V-filt}.
If $r<0$, the map \eqref{eq:SurjSaito2} is nothing but
$$
F_l \cD \cap V^k \cD \cap I(h) \longrightarrow \widetilde{V}^k \cG_l   \cap \sigma_l(I(h)\cap F_l\cD).
$$
We will see that the surjectivity of the latter map follows from the involutivity properties proved above.
Namely, let $\sigma(P)=\sigma_l(P) \in \widetilde{V}^k\cG_l$ be given, where $P\in I(h)\cap F_l\cD$. Then by Lemma \ref{lem:Involutive_F-Filt} we have an expression
$$
\sigma(P)=\sum_{i=0}^n k_i \cdot \sigma(G_i);
$$
recall that we had put
$G_0:=t-h$, $G_1:=\delta_1,\ldots, G_{n-1}:=\delta_{n-1}$ and $G_n:=\chi+\partial_t t +1$ for the involutive basis of $I(h)$. Since obviously the symbols $\sigma(G_i)$ are homogeneous elements in the ring $\cG=\gr_\bullet^F \cD$, we have that $\deg(k_i) = l-\deg(\sigma(G_i)) = l-\ord(G_i)$.
On the other hand, we know by Corollary \ref{cor:Involutive_V-Filt} that $\ord^{\widetilde{V}}(k_i)\geq \ord^{\widetilde{V}}(\sigma(P))-\ord^{\widetilde{V}}(\sigma(G_i))$, that is,
$$
k_i\in \widetilde{V}^{\ord^{\widetilde{V}}(\sigma(P))-\ord^{\widetilde{V}}(\sigma(G_i))}\cG.
$$
Moreover, it follows from the concrete form of the operators $G_0,\ldots, G_n$ that
$\ord^{\widetilde{V}}(\sigma(G_i))=\ord^{V}(G_i)=0$, hence $k_i\in \widetilde{V}^{\ord^{\widetilde{V}}(\sigma(P))}\cG$. Now choose any lift of $k_i$ to an operator $K_i\in F_{l-\ord(G_i)}\cD\cap V^{\ord^{\widetilde{V}}(\sigma(P))}\cD$. Such a lift exists since, as we have already noticed above,
the map
$$
F_l \cD \cap V^k \cD \longrightarrow
\widetilde{V}^k \cG_l
$$
is surjective for all $l, k$. Then the element $P':=\sum_{i=0}^n K_i \cdot G_i\in I(h)$ is the preimage we are looking for, i.e., $\sigma(P')=\sigma(P)$ and we have
$$
P'\in F_l\cD \cap V^{\ord^{\widetilde{V}}(\sigma(P))} \cD
\subset F_l\cD \cap V^k \cD;
$$
recall that we had chosen $\sigma(P)\in \widetilde{V}^k \cG_l \cap \sigma_l(I(h)\cap F_l\cD)$,  meaning that $k  \leq \ord^{\widetilde{V}}(\sigma(P))$.
Hence we have shown the surjectivity of
$$
F_l \cD \cap V^k \cD \cap I(h) \longrightarrow \widetilde{V}^k \gr_l^F \cD  \cap \sigma_l(I(h)\cap F_l\cD),
$$
and as we said at the beginning of the proof, using \cite[Corollaire 1.2.14]{Saito1}, the surjectivity of
$$
F_l \cD  \cap V^k \cD = F_l \cD \cap V^k \cD \cap J_0
\longrightarrow
F_l \gr_0^J \cD \cap V^k \gr_0^J \cD= F^{\ord}_l N(h) \cap V^k_{\ind} N(h),
$$
as required.

\end{proof}

\section{Description of the Hodge filtration}
\label{sec:HodgeFiltration}

This section is the central piece of the article. We apply the results on the canonical $V$-filtration from the last section to compute the Hodge filtration on the mixed Hodge module which has $\cO_X(*D)$ as underlying $\cD_X$-module. The main result is
Theorem \ref{thm:HodgeOnMeromorphic}, which gives a precise description of $F^H_\bullet \cO_X(*D)$. We also complement it with some statements about the Hodge filtration on the dual module $\cO_X(!D)$, see Theorem \ref{theo:DualFiltration}, which we expect to be useful for future applications. We conjecture (see Conjecture \ref{con:GenLevel}) some bound for the so-called generating level of the Hodge filtration
(see Definition \ref{def:LevelHodgeFiltration} for this notion), which is supported by computation of examples in section \ref{sec:Examples} below.

Notice that some of the results in this section (first part of Theorem \ref{thm:HodgeOnMeromorphic} or Theorem \ref{con:GenLevel}) hold globally on $X$, however, for most of the proofs we will need to require that $D\subset X$ is defined by an equation $h$. As already indicated, we will determine the Hodge filtration using the graph embedding $i_h:X\hookrightarrow\dC_t\times X$ and by considering extensions of Hodge modules from $\dC_t^* \times X$ to $\dC_t\times X$.

Hence, let again $X$ be an $n$-dimensional complex manifold and $D\subset X$ a reduced free divisor,
which is strongly Koszul at each point $p\in D$.
We assume for now (until and including Corollary \ref{cor:PreciseDescrHFilt}) that $D$ is given by a reduced equation $h\in \Gamma(X,\cO_X)$. Notice also that contrary to the last section,
we will consider all the objects as sheaves since this is more convenient when applying the functorial constructions
from the theory of mixed Hodge modules.

Put $U:=X\backslash D$, and consider the following basic diagram:
\begin{equation}\label{eq:BasicDiag}
\begin{tikzcd}
U \arrow{rr}{i'_h}\arrow{dd}{j}  && \dC^*_t \times X\arrow{dd}{j'} \\ \\
X\arrow{rr}{i_h} && \dC_t \times X.
\end{tikzcd}
\end{equation}
Here $j$ and $j'$ are open embeddings, whereas $i_h$ and $i'_h$ (both are given by $\underline{x} \mapsto (\underline{x},h(\underline{x}))$ and may be called
graph embeddings) are closed.

We are considering the pure Hodge module $\dQ^H_U[n]$ on $U$. By the functorial properties
of the category of mixed Hodge modules, we know that there is an object $j_*\dQ^H_U[n]\in \MHM(X)$,
whose underlying $\cD_X$-module is the module $\cO_X(*D)$ of meromorphic functions. We are interested in describing
the Hodge filtration $F^H_\bullet$ on $\cO_X(*D)$.

The $\cD_U$-module underlying $\dQ^H_U[n]$ is simply the structure sheaf $\cO_U$ and we have
$F^H_k\cO_U=0$ for all $k<0$ and $F^H_k\cO_U=\cO_U$ for all $k\geq 0$. We  obviously have
$$
\Theta_U \cong j^*\Der_X(-\log D) = \cO_U \delta_1\oplus\ldots\oplus\cO_U\delta_{n-1}\oplus \cO_U \chi,
$$
recall from Proposition \ref{prop:SummarySKLCT} that the basis $\delta_1,\ldots,\delta_{n-1},\chi$ of $\Der_X(-\log D)$ was chosen such that $\delta_i(h)=0$ for $i=1,\ldots,n-1$ and $\chi(h)=h$. Hence
we have the following presentation
$$
\cO_U\cong \frac{\cD_U}{\left(\delta_1,\ldots,\delta_{n-1},\chi+1\right)},
$$
identifying the class of $1\in \cD_U$ on the right-hand side with the function $h^{-1}\in \cO_U$.
Under this isomorphism, $F^H_\bullet \cO_U \cong F^{\ord}_\bullet \cD_U/(\delta_1,\ldots,\delta_{n-1},\chi+1)$,
here $F^{\ord}_\bullet$ denotes the filtration induced on a cyclic $\cD$-module by the
filtration on $\cD$ by the order of differential operators.

From this presentation, we deduce the following statement.
\begin{lemma}\label{lem:chitilde}
The pure Hodge module $i'_{h,*}\dQ^H_U[n]$ has the underlying $\cD_{\dC_t^* \times X}$-module
$$
i'_{h,+}\cO_U \cong
\frac{\cD_{\dC^*_t\times X}}{\cD_{\dC^*_t\times X}(t-h,\delta_1,\ldots,\delta_{n-1},\widetilde{\chi}+1)},
$$
where $\widetilde{\chi}:=\chi+\partial_t t$.
Under this isomorphism, we have the following description of the Hodge filtration on $i'_{h,*}\dQ^H_U[n]$:
\begin{equation}\label{eq:HodgeFiltOutside}
F^H_k(i'_{h,+}\cO_U) \cong F^{\ord}_{k-1}\left[\frac{\cD_{\dC^*_t\times X}}{\cD_{\dC^*_t\times X}(t-h,\delta_1,\ldots,\delta_{n-1},\widetilde{\chi}+1)}\right],
\end{equation}
where again $F^{\ord}_\bullet$ is the filtration induced from $F_\bullet \cD_{\dC_t^*\times X}$.
\end{lemma}
\begin{proof}
We rewrite the morphism $i'_h$ as the composition $i'_h={^{(2)}}i'_h\circ{^{(1)}}i'_h$, where ${^{(1)}}i'_h:U \rightarrow \dC^*_t\times U$ is the graph embedding
of the restricted morphism $h_{|U}:U\rightarrow \dC^*_t$ (in particular, it is closed)
and where
$$
\begin{array}{rcl}
{^{(2)}}i'_h:\dC^*_t\times U & \longrightarrow  & \dC^*_t\times X   \\
(t,x) & \longmapsto & (t,x)
\end{array}
$$
is the canonical open embedding. Notice however that the composed morphism $i'_h:U\rightarrow
\dC^*_t\times X$ is closed, since
the closure of its image is contained in $\dC^*_t\times U$.

We can extend the above diagram \eqref{eq:BasicDiag} as follows:
$$
\begin{tikzcd}
& \dC^*_t\times U
\arrow{ddr}{^{(2)}i'_h}&\\ \\
U \arrow{uur}{^{(1)}i'_h} \arrow{rr}{i'_h}\arrow{dd}{j}  && \dC^*_t \times X\arrow{dd}{j'} \\ \\
X\arrow{rr}{i_h} && \dC_t \times X.
\end{tikzcd}
$$
We have
$$
^{(1)}i'_{h,+} \cO_U = {^{(1)}}i'_{h,+} \frac{\cD_U}{\left(\delta_1,\ldots,\delta_{n-1},\chi+1\right)} = {^{(1)}}i'_{h,*}\left(\frac{\cD_U}{\left(\delta_1,\ldots,\delta_{n-1},\chi+1\right)}\right)[\partial_t]
=
\frac{\cD_{\dC^*_t\times U}}{\left(t-h,\delta_1,\ldots,\delta_{n-1},\widetilde{\chi}+1\right)}
$$
and (see \cite[Formula 1.8.6]{Saitobfunc})
\begin{align*}
F^H_{k+1}(^{(1)}i'_{h,+}\cO_U)
&=\D \sum_{k_1+k_2=k} {^{(1)}i'_{h,*}}F^H_{k_1}\left(\cD_U/(\delta_1,\ldots,\delta_{n-1},\chi+1)\right)\partial_t^{k_2} \\
&=\D \sum_{k_1+k_2=k} {^{(1)}i'_{h,*}}F^{\ord}_{k_1}\left(\cD_U/(\delta_1,\ldots,\delta_{n-1},\chi+1)\right)\partial_t^{k_2} \\
&=F^{\ord}_k\left(\cD_{\dC^*_t \times U}/(t-h,\delta_1,\ldots,\delta_{n-1}, \widetilde{\chi}+1)\right).
\end{align*}
Again, since
\begin{equation}\label{eq:SuppInU}
\overline{\Supp\left(\cD_{\dC^*_t \times U}/(t-h,\delta_1,\ldots,\delta_{n-1},\widetilde{\chi}+1)\right)}
=\overline{\im(^{(1)}i'_h)} \subset \dC^*_t\times U,
\end{equation}
we obtain
$$
^{(2)}i'_{h,+}
\left(\cD_{\dC^*_t\times U}/(t-h,\delta_1,\ldots,\delta_{n-1},\widetilde{\chi}+1)\right)
=
\cD_{\dC^*_t \times X}/(t-h,\delta_1,\ldots,\delta_{n-1},\widetilde{\chi}+1)
$$
Now we consider the Hodge filtration on the module
$$
i'_{h,+}\cO_U=
(^{(2)}i'_h)_+
\left(\cD_{\dC^*_t \times U}/(t-h,\delta_1,\ldots,\delta_{n-1},\widetilde{\chi}+1)\right)
\cong \cD_{\dC^*_t \times X}/(t-h,\delta_1,\ldots,\delta_{n-1},\widetilde{\chi}+1).
$$
We use once more the inclusion in formula \eqref{eq:SuppInU}, which gives, when applying
\cite[Formula 4.2.1]{Saitobfunc}
to the module
$\cM=\cD_{\dC^*_t\times X}/(t-h,\delta_1,\ldots,\delta_{n-1},\widetilde{\chi}+1)$ and the morphism ${^{(2)}i'_h}$, that
$$
F^H_{k+1}
i'_{h,+}\cO_U=
(^{(2)}i'_h)_*
F_k^{\ord}\left(\cD_{\dC^*_t\times U}/(t-h,\delta_1,\ldots,\delta_{n-1},\widetilde{\chi}+1)\right)
=
F_k^{\ord}\left(\cD_{\dC^*_t\times X}/(t-h,\delta_1,\ldots,\delta_{n-1},\widetilde{\chi}+1)\right),
$$
(the second equality follows since
the closure of $\Supp\left(F_k^{\ord}(\cD_{\dC^*_t\times U}/(t-h,\delta_1,\ldots,\delta_{n-1},\widetilde{\chi}+1))\right)$ in $\dC^*_t\times X$ is contained in $\dC^*_t\times U$, i.e., from formula \eqref{eq:SuppInU}).
\end{proof}

Next we will deduce from this result a description of the Hodge filtration on
the Hodge module $j'_* (i'_h)_* \dQ^H_U[n]$. This is the crucial step towards our first main result (Theorem \ref{thm:HodgeOnMeromorphic} below). We will use the results from the last section concerning the canonical $V$-filtration
of the  graph embedding module $i_{h,+}\cO_X(*D)$ in an essential way. The result we are after can be stated as follows.
\begin{proposition}\label{prop:HodgeInclusionNh}
The mixed Hodge module $j'_* i'_{h,_*} \dQ^H_U[n]$ has
$i_{h,+}\cO_X(*D)$ as underlying $\cD_{\dC_t \times X}$-module.
Under the isomorphism
$$
i_{h,+}\cO_X(*D) \cong
\frac{\cD_{\dC_t\times X}}{\cD_{\dC_t\times X}(t-h,\delta_1,\ldots,\delta_{n-1},\widetilde{\chi}+1)} =: \cN(h)
$$
from Lemma \ref{lem:RepresentGraphEmbed} (which, as indicated at the beginning of this section, we write here as an isomorphism of sheaves of $\cD_{\dC_t\times X}$-modules rather than of germs)
we have the following inclusion of coherent $\cO_{\dC_t\times X}$-modules for all $k\in \dZ$:
\begin{equation}\label{eq:HodgeInOrder}
F^H_k(i_{h,+}\cO_X(*D)) \subset
F^{\ord}_{k-1} \cN(h)
\end{equation}
\end{proposition}
\begin{proof}

First put for notational convenience
$$
\cN'(h):=\frac{\cD_{\dC^*_t\times X}}{\cD_{\dC^*_t\times X}(t-h,\delta_1,\ldots,\delta_{n-1},\widetilde{\chi}+1)}
\cong
i'_{h,+}\cO_U.
$$
Now it is known (see \cite[Proposition 4.2]{Saitobfunc}) that the Hodge filtration on the Hodge module $j'_* i'_{h,_*} \dQ^H_U[n]
\cong i_{h,*} j_* \dQ^H_U[n]$ has the following description (this is due to the construction of the open direct image for mixed Hodge modules, see \cite[Section 2.7-2.8]{SaitoMHM}):
\begin{equation}\label{eq:SaitoFormula}
F^H_k(i_{h,+}\cO_X(*D)) =
\sum_{i\geq 0}\partial_t^i\left(
V_{\can}^0 \cN(h) \cap j'_* (j')^* F^H_{k-i}
i_{h,+}\cO_X(*D)
\right),
\end{equation}
but since
$$
(j')^* F^H_{k-i}
(i_{h,+}\cO_X(*D))=
F^H_{k-i} i'_{h,+}\cO_U
=
F^{\ord}_{k-1-i} \cN'(h),
$$
we are left to show that we have
$$
\sum_{i\geq 0}\partial_t^i\left(
V_{\can}^0 \cN(h) \cap j'_* F^{\ord}_{k-1-i} \cN'(h)
\right) \subset F^{\ord}_{k-1} \cN(h).
$$
Since
$$
\partial_t^i F^{\ord}_r \cN(h) \subset
F^{\ord}_{r+i} \cN(h)
$$
holds for all $r,i\in\dN$, it is sufficient to show the inclusion
$$
V_{\can}^0 \cN(h) \cap j'_* F^{\ord}_r \cN'(h)
\subset F^{\ord}_r \cN(h)
$$
for all $r\in \dN$. We have seen in Proposition
\ref{prop:PreciseDescrCanVFilt} that
$V_{\can}^k \cN(h) \subset V_{\ind}^k \cN(h)$
holds for all $k\in\dZ$ (recall that since we assume in this section that $D$ is SK free, we know that the roots of $b_h(s)$ are included in $(-2,0)$ by \cite[Theorem 4.1]{nar_symmetry_BS}, so that the assumptions of Proposition \ref{prop:PreciseDescrCanVFilt} are satisfied). Thus it suffices to show that we have
$$
V_{\ind}^0 \cN(h) \cap j'_* F^{\ord}_r \cN'(h)
\subset F^{\ord}_r \cN(h)
$$
for all $r\in \dN$. Let an element
$m\in V_{\ind}^0 \cN(h) \cap j'_* F^{\ord}_r \cN'(h)
$ be given. Then there is some $p\in\dN$ such that $t^p \cdot m \in V^p_{\ind} \cN(h) \cap F_r^{\ord} \cN(h)$ (see also the proof of \cite[Proposition 3.2.2]{Saito1}, especially the implication
(3.2.1.2) $\Rightarrow$ (3.2.2.1)): Since the
filtration $F^{\ord}_\bullet$ is exhaustive on $V^0_{\ind} \cN(h)$ we know that there is some $l\in \dN$ such that $m\in F^{\ord}_l \cN(h) \cap V^0_{\ind}\cN(h)$, if $l\leq r$, we are done by putting $p=0$. Otherwise, since $m\in j'_*F^{\ord}_r \cN'(h)$ (which means by definition that
$m_{|\dC^*_t\times X}\in F^{\ord}_r\cN'(h)$) it follows
that the class of $m$ in the quotient
$F^{\ord}_l \cN(h)/ F^{\ord}_r \cN(h)$ is a $t$-torsion element in that module. Notice that
$F^{\ord}_l \cN(h)/ F^{\ord}_r \cN(h)$ is
$\cO_{\dC_t\times X}$-coherent, hence there
is some $p$ such that the class of $t^p\cdot m$ is zero in this quotient, that is $t^p\cdot m \in
 F^{\ord}_r \cN(h)$, and obviously we have $t^p\cdot m \in V_{\ind}^p \cN(h)$ since
 $m\in V^0_{\ind} \cN(h)$.

Now by Proposition \ref{prop:CompatibilityVandF} we know that there exists an operator $P'\in V^p\cD_{\dC_t\times X}\cap F_r \cD_{\dC_t\times X}$ projecting to
$t^p\cdot m \in V^p_{\ind} \cN(h) \cap F_r^{\ord} \cN(h)$. By definition,
$P'$ can be written as $P'=t^p\cdot P$, where $P\in V^0\cD_{\dC_t\times X}\cap F_r \cD_{\dC_t\times X}$. Then
the class $[P]$ of $P$ in $\cN(h)$ satisfies
$[P]\in F^{\ord}_r \cN(h)$ and obviously we have $[P]=m$. Hence
the inclusion $V_{\ind}^0 \cN(h) \cap j'_* F^{\ord}_r \cN'(h)
\subset F^{\ord}_r \cN(h)$ is proved.
\end{proof}

Using the description of the canonical $V$-filtration on the module $\cN(h)$ along the divisor $\{t=0\}$ from Corollary \ref{prop:PreciseDescrCanVFilt}, we can give a more precise description of the Hodge filtration on that module. We also recall
from such corollary that
$B'_h:=\left\{\alpha\in \dQ \cap (0,1
)\,|\, b_h(\alpha-1)=0\right\}$, and for $\alpha\in B'_h$, we write $l_\alpha$ for the multiplicity of
$\alpha-1$ in $b_h(s)$.

For the sake of brevity, we will sometimes write in the following $F^H\cN(h)$, translating the Hodge filtration on $i_{h,+}\cO_X(*D)$ using the isomorphism $\cN(h)\cong i_{h,+}\cO_X(*D)$.
\begin{corollary}\label{cor:PreciseDescrHFilt}
Under the isomorphism $i_{h,+}\cO_X(*D)\cong \cN(h)$, we have the following (equivalent) descriptions of the Hodge filtration steps $F^H_k \cO_X(*D)$ for all $k\in \dZ$:
\begin{eqnarray}
     \D F_k^H i_{h,+}\cO_X(*D) & = & \label{eq:FormulaHodgeNh_NonRec}\D \sum_{i\geq 0} \partial_t^i\left(V_{\can}^0\cN(h)\cap F^{\ord}_{k-1-i}\cN(h)\right)
     =
      \sum_{i= 0}^{k-1} \partial_t^i\left(V_{\can}^0\cN(h)\cap F^{\ord}_{k-1-i}\cN(h)\right),
     \\      \label{eq:FormulaHodgeNh_Simple}\D F_k^H i_{h,+}\cO_X(*D)& = &\D
\partial_t F^H_{k-1} i_{h,+}\cO_X(*D)  +
V^0_{\can}\cN(h) \cap F^{\ord}_{k-1} \cN(h) \\ & \cong & \label{eq:FormulaHodgeNh}
 \D \partial_t
F^H_{k-1} i_{h,+}\cO_X(*D)  +
\left(
V_{\ind}^1\cN(h) + \prod_{\alpha \in B'_h}(\partial_t t +\alpha)^{l_\alpha} V_{\ind}^0 \cN(h)
\right) \cap F^{\ord}_{k-1} \cN(h).
\end{eqnarray}

\end{corollary}
\begin{proof}

From the proof of the last lemma we know already that
\begin{equation}\label{eq:SaitoFormOrd}
F^H_k(i_{h,+}\cO_X(*D)) =
\sum_{i\geq 0}\partial_t^i\left(
V_{\can}^0 \cN(h) \cap j'_* F^{\ord}_{k-1-i} \cN'(h)\right)
\end{equation}
so that in particular $F^H_k i_{h,+}\cO_X(*D) =0$ for all $k<1$ since $F^{\ord}_l \cN'(h)=0$ for negative $l$.
The same holds for $F^{\ord}_l\cN(h)$, hence, the above formulas are proved when $k<1$.
We rewrite formula \eqref{eq:SaitoFormOrd} in a recursive way, namely
\begin{align}
\nonumber F^H_k(i_{h,+}\cO_X(*D)) & =
 \sum_{i >  0}\partial_t^i
 \left(
 V_{\can}^0 \cN(h) \cap j'_* F^{\ord}_{k-1-i} \cN'(h)
\right)+
V_{\can}^0 \cN(h) \cap j'_* F^{\ord}_{k-1} \cN'(h)
\\ \nonumber & =
 \partial_t \left(\sum_{i \geq  0}\partial_t^i
 \left(
 V_{\can}^0 \cN(h) \cap j'_* F^{\ord}_{k-2-i} \cN'(h)
\right)\right)+
V_{\can}^0 \cN(h) \cap j'_* F^{\ord}_{k-1} \cN'(h)
 \\ \label{eq:RecursiveFOnNh}
&=
\partial_t F_{k-1}^H(i_{h,+}\cO_X(*D)) +
V_{\can}^0 \cN(h) \cap j'_* F^{\ord}_{k-1} \cN'(h)
\end{align}
where the last equality uses again formula \eqref{eq:SaitoFormOrd}.

Now we conclude using Proposition \ref{prop:HodgeInclusionNh}: We know from equation \eqref{eq:HodgeInOrder} that $F^H_k(i_{h,+}\cO_X(*D)) \subset F^{\ord}_{k-1} \cN(h)$, hence we obtain
$$
V_{\can}^0 \cN(h) \cap F^{\ord}_{k-1} \cN(h) \subset
V_{\can}^0 \cN(h) \cap j'_* F^{\ord}_{k-1} \cN'(h) \stackrel{*}{\subset}
F^H_k(i_{h,+}\cO_X(*D)) \subset F^{\ord}_{k-1} \cN(h).
$$
since we clearly have
$F^{\ord}_{k-1} \cN(h) \subset j'_* F^{\ord}_{k-1} \cN'(h)$
by the very definition of the functor $(j')_*$
and where the inclusion $\stackrel{*}{\subset}$ follows from equation \eqref{eq:RecursiveFOnNh} above.
As a consequence, the first inclusion
$V_{\can}^0 \cN(h) \cap F^{\ord}_{k-1} \cN(h) \subset
V_{\can}^0 \cN(h) \cap j'_* F^{\ord}_{k-1} \cN'(h)$ is in fact an equality, and therefore equation \eqref{eq:RecursiveFOnNh} becomes
$$
F^H_k(i_{h,+}\cO_X(*D)) = \partial_t F_{k-1}^H(i_{h,+}\cO_X(*D)) +
V_{\can}^0 \cN(h) \cap F^{\ord}_{k-1} \cN(h),
$$
so that the recursive Formula \eqref{eq:FormulaHodgeNh_Simple} is shown.
Formula \eqref{eq:FormulaHodgeNh} follows by replacing
the term $V^0_{\can} \cN(h)$ with the expression from
Proposition \ref{prop:PreciseDescrCanVFilt}. Moreover, it is clear that the non-recursive Formula \eqref{eq:FormulaHodgeNh_NonRec} follows from the recursive one \eqref{eq:FormulaHodgeNh_Simple} by induction.

\end{proof}
Our main purpose is to describe the Hodge filtration on $\cO_X(*D)$.
This description will be obtained as a consequence of Proposition \ref{prop:HodgeInclusionNh} and Corollary \ref{cor:PreciseDescrHFilt}.
Recall that the graph embedding module $\cN(h)$
can be alternatively described as $\cM(h)[\partial_t]$,  where the left action by $\cD_{\dC_t\times X}$ on
$\cM(h)[\partial_t]$ is given by Formula \eqref{eq:LeftActionGraph}.

We then have the following result.
\begin{theorem}\label{thm:HodgeOnMeromorphic}
Consider a strongly Koszul free divisor $D\subset X$.
Then we have the following inclusion of coherent $\cO_X$-modules (which holds globally on $X$ and does not require that $D$ is given by a global equation)
$$
F^H_k \cO_X(*D) \subset F_k^{\ord} \cO_X(*D),
$$
where $F^H_\bullet \cO_X(*D)$ is the filtration such that
the filtered $\cD_X$-module $(\cO_X(*D),F^H_\bullet)$ underlies
the mixed Hodge module $j_* \dQ^H_U[n]$.

If $h=0$ is a local reduced equation for $D$ and using the local isomorphism $\cM(h)\cong  \cO_X(*D)$,  we have the recursive formula \begin{equation}\label{eq:MainFormula}
F_k^H \cO_X(*D) \cong F_k^H \cM(h) \cong
\left[
\partial_t F^H_k \cN(h)+
\left(V_{\ind}^1\cN(h) + \prod_{\alpha \in B'_h}(\partial_t t +\alpha)^{l_\alpha} V_{\ind}^0 \cN(h)
\right)
\right]
\cap
\left( F_k^{\ord}\cM(h)\otimes 1\right),
\end{equation}
where the intersection takes place in $\cM(h)[\partial_t] \cong \cN(h)$.

\end{theorem}

Formula \eqref{eq:MainFormula} is recursive despite the appearance of the same index $k$ on both sides:
in order to calculate (locally) the $k$-th filtration step of the Hodge filtration on $\cO_X(*D)$, we need to know the $k$-th filtration step of the Hodge filtration on $\cN(h)$, the knowledge of which is equivalent to the knowledge of  $F^H_{k-1} \cM(h)\cong F^H_{k-1} \cO_X(*D)$ (see the discussion in the proof below, in particular, Formula \eqref{eq:HodgeMhFromHogdeNh}).

\begin{proof}
It is sufficient to prove the second statement, since locally where $D$ is defined by an equation $h=0$, formula \eqref{eq:MainFormula} exhibits $F_k^H\cO_X(*D)$ as a submodule of $F_k^{\ord}\cM(h)$. However, the
statement $F^H_k \cO_X(*D) \subset F_k^{\ord} \cO_X(*D)$
(where $F^{\ord}_\bullet \cO_X(*D)$ is as in Definition \ref{def:F-ord}) holds globally on $X$, once it is shown at each $p\in D$.

First notice that we have the equality
$$
i_{h,*} j_* \dQ^H_U[n] =
j'_* i'_{h,*} \dQ^H_U[n]
$$
of objects in $\MHM(\dC_t\times X)$. It follows that
the filtered module
$(\cN(h),F_\bullet^H)$ (where $F^H_\bullet$ is the filtration considered in Proposition \ref{prop:HodgeInclusionNh} and Corollary \ref{cor:PreciseDescrHFilt}) underlies the mixed Hodge module
$i_{h,*} j_* \dQ^H_U[n]$.
Then \cite[Formula (1.8.6)]{Saitobfunc} yields the following relation between the filtration $F^H_\bullet \cN(h)$ and
the Hodge filtration on $j_* \dQ^H_U[n]$ (which is
$F_\bullet^H \cO_X(*D)$, and again we will freely denote it by $F^H_\bullet \cM(h)$ using the isomorphism $\cM(h)\cong \cO_X(*D)$):
\begin{equation}\label{eq:HodgeNhFromHogdeMh}
F_k^H \cN(h) =
\sum_{j\geq 0} F_{k-1-j}^H \cM(h)\otimes \partial_t^j,
\end{equation}
in other words, we obtain
\begin{equation}\label{eq:HodgeMhFromHogdeNh}
F_{k-1}^H\cO_X(*D) \cong F_{k-1}^H\cM(h) \cong F^H_k\cN(h) \cap \left(\cM(h)\otimes 1\right)
\end{equation}
for all $k>1$. As we have noticed in Corollary \ref{cor:PreciseDescrHFilt}, we have the inclusion
$$
F^H_k i_{h,+}\cO_X(*D) \cong F^H_k \cN(h) \subset F^{\ord}_{k-1} \cN(h),
$$
and since obviously
$F^{\ord}_{k-1} \cN(h)
\cap\left(\cM(h)\otimes 1\right) =
F^{\ord}_{k-1} \cN(h)
\cap\left(F^{\ord}_{k-1}\cM(h)\otimes 1\right)$, we obtain by plugging in formula
\eqref{eq:FormulaHodgeNh} that
$$
F_{k-1}^H\cO_X(*D) \cong
\left[
\partial_t F^H_{k-1}\cN(h)+
\left(
V^1_{\ind}\cN(h)+ \prod_{\alpha \in B'_h}(\partial_t t +\alpha)^{l_\alpha} V_{\ind}^0 \cN(h)
\right)
\cap F^{\ord}_{k-1}\cN(h)
\right]
\cap \left(F^{\ord}_{k-1}\cM(h)\otimes 1\right).
$$
Shifting the indices by one, the inclusion
$$
F_k^H\cO_X(*D) \subset
\left[
\partial_t F^H_k \cN(h)+
\left(V_{\ind}^1\cN(h) + \prod_{\alpha \in B'_h}(\partial_t t +\alpha)^{l_\alpha} V_{\ind}^0 \cN(h)
\right)
\right]
\cap
\left( F_k^{\ord}\cM(h)\otimes 1\right)
$$
is then clear. But we also have the inclusion in the other direction:
Let $a\in \partial_t F_k^H\cN(h)$ and $b\in V_{\ind}^1\cN(h) + \prod_{\alpha \in B'_h}(\partial_t t +\alpha)^{l_\alpha} V_{\ind}^0 \cN(h)$ and assume that $a+b\in F^{\ord}_k\cM(h)\otimes 1\subset F^{\ord}_k\cN(h)$.
Then since $\partial_t F_k^H\cN(h)\subset \partial_t F_{k-1}^{\ord}\cN(h)\subset F_k^{\ord}\cN(h)$, we must necessarily have
$b\in F_k^{\ord}\cN(h)$, so that
$$
a+b\in \left[
\partial_t F^H_k \cN(h)+
\left(V_{\ind}^1\cN(h) + \prod_{\alpha \in B'_h}(\partial_t t +\alpha)^{l_\alpha} V_{\ind}^0 \cN(h)
\right)\cap F^{\ord}_k\cN(h)
\right]
\cap
\left(F_k^{\ord}\cM(h)\otimes 1\right)
\cong F^H_k\cO_X(*D),
$$
as required.

\end{proof}

\begin{remark}\label{rem:HodgeOrderPoleOrder}
(see also subsection \ref{subsec:NCD} below)
As shown in \cite[Proposition 0.9]{Saitobfunc}
we have
$$
F^H_\bullet \cO_X(*D) \subset P_\bullet \cO_X(*D)
$$
for any reduced divisor, where $P_k \cO_X(*D) := \cO_X((k+1)D)$ is the pole order filtration.
In our situation of a strongly Koszul free divisor, we obviously have
$$
F^H_k \cO_X(*D) \subset F^{\ord}_k \cO_X(*D) \subset P_k \cO_X(*D),\ k\geq 0,
$$
with $F^H_0 \cO_X(*D)\subset F^{\ord}_0 \cO_X(*D) = P_0 \cO_X(*D)$. One can thus consider $F^{\ord}_k \cO_X(*D)$, in the current situation, as a better approximation to the Hodge filtration than the pole order filtration $P_\bullet \cO_X(*D)$ (though both are
equal at level $k=0$) .
A very special example is the case where $D$ has normal crossings, then for any local reduced equation $h=0$ for $D$ the only root of $b_h(s)$ is $-1$, and one easily obtains $F_k^H\cO_X(*D)= F_k^{\ord} \cO_X(*D)$. Correspondingly,
we have $F_0^H\cO_X(*D) = F_0^{\ord} \cO_X(*D)=P_0 \cO(*D) = \cO(D)$,
so that $\cI_0(D)=\cO_X$, but the higher Hodge ideals are
non-trivial even for the normal crossing case, precisely because $F_k^{\ord} \cO_X(*D)\subsetneq P_k \cO(*D)$ for $k>0$.
\end{remark}

We have seen so far that $F^H_\bullet \cO_X(*D) \subset F^{\ord}_\bullet \cO_X(*D)$, with equality iff
$-1$ is the only root of $b_h(s)$ for any local reduced equation $h$ of $D$. However, we can actually
give an inclusion in the reverse direction, but with a specific shift, which we conjecture to be the generating level of the Hodge filtration on $\cO_X(*D)$ (see Conjecture \ref{con:GenLevel} below). As a preparation, we need the following result.

\begin{lemma}\label{lem:1InFr}
 Let $(D,p) \subset (X,p)$ be a germ of strongly Koszul free divisor with reduced equation $h=0$, let $\beta(s)=\prod_{\alpha \in B'_h}(s +\alpha)^{l_\alpha}$, where $l_\alpha$ is the multiplicity of $\alpha-1$ in $b_h(s)$ and put $r := \deg\beta(s)$. Consider the generator $[1]\in \cN(h)$, then we have
$$
[1] \in F^H_{r+1} \cN(h).
$$
\end{lemma}

\begin{proof}
Write
$\beta(s)=\sum_{i=0}^r a_i s^i$ with $a_0,\ldots, a_r\in \dC$. We have
$[t\cdot(\partial_t t)^{i-1}]\in V^1_{\ind}\cN(h)\cap F^{\ord}_{i-1}\cN(h)$ for all $i\in\{1,\ldots,r\}$. Now we may rewrite Formula \eqref{eq:FormulaHodgeNh} as
{\small
$$
F^H_{k+1}\cN(h) \cong \partial_t
\Big(
\partial_t F^H_{k-1}\cN(h)+\left(\beta(\partial_t t)V^0_{\ind}\cN(h)+V^1_{\ind}\cN(h)\right)
\cap F^{\ord}_{k-1}\cN(h)\Big)
+
\left(\beta(\partial_t t)V^0_{\ind}\cN(h)+V^1_{\ind}\cN(h)\right)
\cap F^{\ord}_k\cN(h).
$$
}
Then we have for all $i\in\{1,\ldots,r\}$
$$
[(\partial_t t)^i] = \partial_t \cdot [t(\partial_t t)^{i-1}]
\in \partial_t
\left[
V^1_{\ind}\cN(h)\cap F^{\ord}_{i-1}\cN(h)
\right]\subset F^H_{i+1} \cN(h) \subset F^H_{r+1}\cN(h).
$$
On the other hand, we have $\beta(\partial_t t) \cdot [1] \in \beta(\partial_t t)V^0_{\ind}\cN(h)\cap F^{\ord}_r\cN(h) \subset F^H_{r+1}\cN(h)$. Since $\beta(\partial_t t)-
\sum_{i=1}^r a_i (\partial_t t)^i=a_0$
we find that $a_0[1]\in F^H_{r+1}\cN(h)$.
By definition of the set $B'_h$, we have $\beta(0)\neq 0$ and hence $a_0\neq 0$,
so that finally $[1]\in F^H_{r+1} \cN(h)$, as required.
\end{proof}
We obtain the following two consequences that we will use below to discuss the dual Hodge filtration.
\begin{corollary}\label{cor:Inclu}
With the aforementioned notations, we have
\begin{enumerate}
    \item $[1]\in F^H_r \cM(h)$,
        \item For all $k\in \dZ_{\geq 0}$, we have the inclusion of coherent $\cO_X$-modules $F^{\ord}_{k-r} \cM(h) \subset F^H_k \cM(h)$.
\end{enumerate}
\end{corollary}
\begin{proof}
$\mbox{ }$
\begin{enumerate}
    \item Recall Formula \eqref{eq:HodgeMhFromHogdeNh}, which says that
    $$
    F^H_{k-1} \cM(h) \cong F^H_k \cN(h)\cap \left(\cM(h)\otimes 1\right).
    $$
    Now obviously the element $[1]\in \cN(h)$ belongs to the submodule $\cM(h)\otimes 1\subset \cN(h)$,
    hence we obtain from Lemma \ref{lem:1InFr} that
    the element $[1]\in \cM(h)$ lies in $F^H_r\cM(h)$.
            \item The statement is trivial for $k<r$. On the other hand, we have by definition that for any $k\geq r$ any class $[P]\in F^{\ord}_{k-r} \cM(h)$ can be represented by an operator $P\in \cD_X$ of order $k-r$. Moreover,
    $F^H_\bullet \cM(h)$ is a good filtration
     (in particular, compatible with $F_\bullet \cD_X$), so $[P]=P\cdot[1]\in F_{k-r}\cD_X\cdot [1]\subset
    F_{k-r}\cD_X\cdot F^H_r\cM(h)\subset F^H_k\cM(h)$.  Hence $F^{\ord}_{k-r} \cM(h) \subset F^H_k \cM(h)$, as required.
\end{enumerate}
\end{proof}

Let us call $b_D(s)$ the least common multiple of the local $b$-functions of all reduced local equations of $D$ at $p$, for $p\in D$, whenever it exists (it always exists if $D\subset X$ is algebraic or if $X$ is compact),
$B'_D:=\left\{\alpha_i\in \dQ \cap (0,1)\,|\, b_D(\alpha_i-1)=0\right\}$
and $\beta_D(s)=\prod_{\alpha \in B'_D}(s +\alpha)^{l_\alpha}$, where $l_\alpha$ is the multiplicity of $\alpha-1$ in $b_D(s)$ and put $r := \deg\beta_D(s)$. The strong Koszul hypothesis on $D$ then implies the symmetry property $b_D(-s-2) = \pm b_D(s)$ and so  $r=\frac12\left(\deg(b_D(s))-\textup{mult}_{b_D(s)}(-1)\right)$.
The following global result about the divisor $D\subset X$ is a consequence of the first point of Theorem \ref{thm:HodgeOnMeromorphic} and of Corollary \ref{cor:Inclu}.

\begin{corollary} \label{cor:Inclu-global}
Under the above conditions, we have:
\begin{equation}\label{eq:SummaryMainInlcusions}
F^{\ord}_{\bullet-r} \cO_X(*D) \subset
F^H_\bullet \cO_X(*D) \subset
F^{\ord}_\bullet \cO_X(*D).
\end{equation}
\end{corollary}

We will finish this section by showing some results about the Hodge filtration on the dual Hodge module $\dD j_* \dQ_U^H[n]$. Its underlying $\cD_X$-module is $\dD\cO_X(*D)$, which we denote by $\cO_X(!D)$. From the logarithmic comparison theorem we know that $\cD_X \otimes_{\cV_X^D} \cO_X(D) \cong \cO_X(*D)$, and so we can apply Proposition \ref{prop:StrictnessDualOrder}
(for the case $\cE=\cO_X(D)$) to obtain that the holonomic filtered module $(\cO_X(*D),F_\bullet^{\ord})$ has the Cohen-Macaulay property and moreover
that the dual filtered module of $(\cO_X(*D),F_\bullet^{\ord})$ is isomorphic to $(\cO_X(!D)\cong\cD_X \otimes_{\cV_X^D} \cO_X, F_\bullet^{\ord})$.
We recall that a similar property holds for the Hodge filtration, more precisely, we have the following.
\begin{theorem}[{\cite[Section 2.4]{Saito1}, \cite[Lemme 5.1.13]{Saito1}, see also  \cite[Theorem 29.3]{Sch14}}]\label{theo:HodgeCM}
Let $(\cM,F_\bullet)$ be a filtered left or right $\cD_X$-module underlying a pure Hodge module $M$ of weight $w$, then $(\cM,F_\bullet)$ has the Cohen-Macaulay property. In particular, its dual module is strict, and underlies a pure Hodge module $\dD M$ of weight $-w$. Similarly, if $(M,W_\bullet M)$ is a mixed Hodge module, then so is $(\dD M, \dD W_{-\bullet} M)$.
In particular, if $(\cM,F_\bullet)$ underlies a mixed Hodge module $M$, then it also fulfills the Cohen-Macaulay property. Its dual filtration $F^{H,\dD}_\bullet(\dD \cM)$ is the Hodge filtration $F^{\bD,H}_\bullet (\dD \cM)$ of $\dD M$ up to a shift, i.e., we have
\begin{equation}\label{eq:DualHodgeFiltrationShift}
F^{H,\dD}_\bullet(\dD \cM) = F^{\bD,H}_{\bullet+\dim(X)} (\dD \cM).
\end{equation}
\end{theorem}
With these preparations, we obtain the following result concerning the Hodge filtration on the dual module $\cO_X(!D)$.
\begin{theorem}\label{theo:DualFiltration}
Let, as above, $D\subset X$ be a strongly Koszul free divisor,
then we have the following inclusions of coherent $\cO_X$-submodules of $\cO_X(!D)$:
$$
F^{\ord}_\bullet \cO_X(!D) \subset F^{H,\dD}_\bullet \cO_X(!D)
\subset F^{\ord}_{\bullet+r} \cO_X(!D).
$$
In particular, since $F^{\ord}_{-1}\cO(!D)=0$, we obtain the vanishing
$$
F^{H,\dD}_{-r-1} \cO_X(!D) = F_{n-r-1}^{\dD,H}\cO_X(!D)=0.
$$
\end{theorem}

\begin{proof}
Equation \eqref{eq:SummaryMainInlcusions} from Corollary \ref{cor:Inclu-global} says
$$
F^{\ord}_{\bullet-r} \cO_X(*D) \subset F^H_\bullet \cO_X(*D) \subset F^{\ord}_{\bullet} \cO_X(*D).
$$
Applying the Rees functor $\cR(-)$ defined on page \pageref{page:defReesFunctor} to the first inclusion yields a short exact sequence of graded $\wcD_X$-modules
\begin{equation}\label{eq:ShortSeqRMod}
0\longrightarrow (\cR_{F^{\ord}} \cO_X(*D))(-r) \longrightarrow \cR_{F^H} \cO_X(*D) \longrightarrow \cC \longrightarrow 0,
\end{equation}
recall that  for a graded $\wcD_X$-module $\wcM$, we write $\wcM(l)$ for the same module with grading $\wcM(l)_k=\wcM_{k+l}$. Notice that the cokernel $\cC$ is a $z$-torsion module, since, e.g., for any section $s\in F_k^H \cO_X(*D)$ we have $s \in F^{\ord}_k \cO_X(*D)$ (due to the second inclusion from Formula \eqref{cor:Inclu-global}). Hence $[s]\cdot z^k\in (\cR_{F^{\ord}}\cO_X(*D))_k= (\cR_{F^{\ord}}\cO_X(*D)(-r))_{k+r}$ and therefore $[s]\cdot z^k$ maps to 0 in $\cC$.

Now apply the duality functor for graded left $\wcD_X$-modules (defined by formula \eqref{eq:DualRMod-left}) to the short exact sequence \eqref{eq:ShortSeqRMod}. From Corollary \ref{cor:DualOrderShifted} we deduce that
$$
\astdD(\cR_{F^{\ord}}\cO_X(*D)(-r))\cong \cH^0\astdD(\cR_{F^{\ord}}\cO_X(*D))(r) = (\cR_{F^{\ord}}\cO_X(!D))(r).
$$
On the other hand, by Theorem \ref{theo:HodgeCM} we
know that the filtered module $(\cO_X(*D), F_\bullet^H)$ satisfies the Cohen-Macaulay property as well, and we write as before $(\cO(!D), F_\bullet^{H,\dD})$ for its dual filtered module.

Hence, the triangle resulting from applying $\astdD$ to the sequence
\eqref{eq:ShortSeqRMod} has the following long exact cohomology sequence:
\begin{equation}\label{eq:DualSequence}
0 \longrightarrow \cH^0 \astdD\cC \longrightarrow \cR_{F^{H,\dD}}\cO(!D) \longrightarrow (\cR_{F^{\ord}}\cO_X(!D))(r) \longrightarrow \cH^1 \astdD\cC \longrightarrow 0.
\end{equation}

The $\wcD_X$-module $\cC$ is a $z$-torsion module, which implies that $\cH^0\astdD\cC$ is so (since $z$ is a central element in $\wcD_X$).
The $\wcD_X$-module $\cR_{F^{H,\dD}}\cO(!D)$ (and similarly $(\cR_{F^{\ord}}\cO_X(!D))(r)$) has no $z$-torsion, since it is the Rees module of a filtered $\cD_X$-module,
hence $\cH^0\astdD\cC=0$. Notice however that $\cH^1\astdD\cC$ does not necessarily vanish, since the homological dimension of the rings $R_U$ defined in Formula \eqref{eq:DefRU} is bounded from above by $2n+1$ only (which is the global homological dimension of $\gr_\bullet^F R_U=\cO_U[z,\xi_1,\ldots,\xi_n]$, where $\xi_i$ is the symbol of $z\partial_{x_i}$), see \cite[Theorem D.2.6 and Theorem D.4.3 (ii)]{Hotta}.
Summing up, the exact sequence \eqref{eq:DualSequence} yields the injective
morphism $\cR_{F^{H,\dD}}\cO(!D) \hookrightarrow (\cR_{F^{\ord}}\cO_X(!D))(r)$ of graded $\wcD_X$-modules, which gives an inclusion
\begin{equation}
F^{H,\dD}_k\cO_X(!D)\subset F^{\ord}_{k+r}\cO_X(!D)
\end{equation}
of coherent $\cO_X$-modules for any $k\in \dZ$.

The very same argument, starting from the inclusion $F^H_\bullet \cO_X(*D)\subset F^{\ord}_\bullet \cO_X(*D)$ (and using the Cohen-Macaulay property of $(\cO_X(*D),F^{\ord}_\bullet)$ proved in Proposition \ref{prop:StrictnessDualOrder} instead of the one
of  $(\cO_X(*D),F^H_\bullet)$ as above) yields that
$$
F^{\ord}_k \cO_X(!D)\subset F^{H,\dD}_k \cO_X(!D)
$$
for all $k\in \dZ$, which completes the proof.
\end{proof}

Finally, we will state a conjecture about the so-called generating level of the Hodge filtration. Let us recall the following definition from \cite{SaitoOnTheHodgeFiltration}.

\begin{definition}\label{def:LevelHodgeFiltration}
Let $X$ be a complex manifold and
let a well filtered module $(\cM,F_\bullet)$ be given. Then we say that the  filtration $F_\bullet \cM$ is generated at level $k$ if
$$
F_l \cD_X\cdot F_k \cM = F_{k+l} \cM
$$
holds for all
$l\geq 0$,
or, equivalently, if for all $k'\geq k$, we have
$$
F_1\cD_X \cdot F_{k'} \cM = F_{k'+1} \cM.
$$
The smallest integer $k$ with this property is called the generating level of $F_\bullet \cM$.
\end{definition}

Based on our calculations in section \ref{sec:Examples}, we conjecture the following bound for the generating level of $F^H_\bullet$ on $\cO_X(*D)$.
\begin{conjecture}\label{con:GenLevel}
Under the assumptions of Theorem \ref{thm:HodgeOnMeromorphic} and Corollary \ref{cor:Inclu-global}, the Hodge filtration $F^H_\bullet \cO(*D)$ is generated at level $r:=\frac12\left(\deg(b_D(s))-\textup{mult}_{b_D(s)}(-1)\right)$.
\end{conjecture}

\begin{remark}\label{rem:ConseqGenLev}
\begin{enumerate}
\item
As we will see in section \ref{sec:Examples}, this bound for the generating level is not always better than the known general bound, which is $n-2$ (see \cite[Theorem B]{PopaMustata}). For the specific class of \emph{linear free divisors}, it is known however that $\deg(b_h(s))=n$, and in this case the bound $r$ would be a drastic improvement, see subsection \ref{subsec:LFD} below. Notice also that the general bound
$n-1-\lceil \widetilde{\alpha}_D \rceil $ from \cite[Theorem A]{MP-Inv} is not helpful in our situation, since the minimal exponent $\widetilde{\alpha}_D$, which is the negative of the biggest root of the \emph{reduced} Bernstein polynomial
$b_h(s)/(s+1)$, lies in $(0,1)$, so that this bound is again $n-2$.
\item
There is a criterion in \cite[Lemma 2.5]{Saito94} that guarantees generating level smaller or equal to $k$ for a filtered module $(\cM, F_\bullet)$ satisfying the Cohen-Macaulay property in terms of vanishing of the dual filtration. It cannot, however, directly be applied in our situation since for this we would need the vanishing $F^{\dD,H}_{2n-r-1} \cO_X(!D)$ instead of
$F^{\dD,H}_{n-r-1} \cO_X(!D)=0$, which is provided by Theorem \ref{theo:DualFiltration}. One may hope though that some refinement of this argument (see, e.g., \cite[Proposition 3.3]{MP-Inv} for a statement in a related situation) can give the desired result.
\item As an immediate consequence of the conjecture, we would obtain a local vanishing result for a log resolution of an SK-free divisor, as provided by \cite[Theorem 17.1]{PopaMustata}. Namely, for a log resolution
$\pi:\widetilde{X}\longrightarrow X$ which is an isomorphism over $X\backslash D$ and
where $E:=(\pi^{-1}(D))_{\textup{red}}$ (a reduced normal crossing divisor on $\widetilde{X}$) we would have (provided that Conjecture \ref{con:GenLevel} holds true)
$$
R^k\pi_* \Omega^{n-k}_{\widetilde{X}}(\log\, E) = 0
$$
for all $k>r$.
\end{enumerate}
\end{remark}

\section{Computations of Hodge ideals and examples}
\label{sec:Examples}

The purpose of this section is to develop some techniques to calculate Hodge ideals of a strongly Koszul free divisor using our main result (Theorem \ref{thm:HodgeOnMeromorphic}).
We will illustrate these methods by concrete computations of Hodge ideals and of the generating level of the Hodge filtration for some interesting examples. We start with the following preliminary result.
\begin{lemma}\label{lem:reduction_to_D_X[s]_trick}
Let $D\subset X$ be a free divisor, let $p\in D$ such that $D$ is strongly Koszul at $p$ and take
a reduced local defining equation $h\in\cO_{X,p}$ of $(D,p)\subset (X,p)$. Consider the vector fields
$\delta_1,\ldots,\delta_{n-1},\tilde\chi$ as in Lemma \ref{lem:chitilde}.
Let $\beta(s)$ be any polynomial in $\dC[s]$. Then
\begin{equation}
\label{eq:LemmaV0}
V^0\cD_{\dC_t\times X,(0,p)}( t,\beta(\dd_tt),t-h,\delta_1,\ldots,\delta_{n-1},\tilde\chi+1)\cap\cD_{X,p}[\dd_tt] =\cD_{X,p}[\dd_tt]( h,\beta(\dd_tt),\delta_1,\ldots,\delta_{n-1},\tilde\chi+1).
\end{equation}
\end{lemma}
\begin{proof}
For the sake of simplicity, let us denote $\cD_{\dC_t\times X,(0,p)}$ just by $\cD$.
Before starting the proof, let us explain something we will take for granted throughout it: we know that $V^0\cD =\dC\{x,t\}[\partial_x,\dd_t t]$ and that any $P\in V^0\cD $ can be expressed in a unique way as a series
$$P= \sum_{|\beta|+ m\leq d} \sum_{\alpha, \ell} a_{\alpha,\ell}^{\beta,m} x^\alpha t^\ell \partial_x^\beta (\dd_t t)^m$$
with constant coefficients, and for each $(\beta, m)$, the series $\sum_{\alpha, \ell} a_{\alpha,\ell}^{\beta,m} x^\alpha t^\ell$ is convergent. Now, by using the identity $t^\ell (\dd_t t)^m = (\dd_t t -\ell)^m t^\ell $ we find another unique formal representation
$$ P =  \sum_{|\beta|+ i\leq d} \sum_{\alpha,\ell} c_{\alpha,\ell}^{\beta,i} x^\alpha \partial_x^\beta (\dd_t t)^i t^\ell
$$
with
$$c_{\alpha,\ell}^{\beta,i} = \sum_{i\leq m\leq d} a_{\alpha,\ell}^{\beta,m} \binom{m}{i}(-\ell)^{m-i},
$$
and one easily sees that for fixed $\beta,i,\ell$ the series $\sum_\alpha c_{\alpha,\ell}^{\beta,i} x^\alpha $ is convergent. So we have a unique formal expression
$$ P = \sum_\ell P_\ell t^\ell \text{, with } P_\ell = \sum_{|\beta|+ i\leq d} \sum_{\alpha} c_{\alpha,\ell}^{\beta,i} x^\alpha \partial_x^\beta (\dd_t t)^i \in \cD_{X,p}[\dd_tt],
$$
and it makes sense to consider $P\in \cD_{X,p}[\dd_tt][[t]]$, where
the last ring is the completion of $\cD_{X,p}[\dd_tt][t]$ with respect to the $( t )$-adic topology. Here one has to use that the left ideal generated by $t$ coincides with the right ideal generated by $t$, and so it is a bilateral ideal, and also that the monomials $t^\ell, \ell\geq 0$, form a basis of $\cD_{X,p}[\dd_tt][t]$ as a left and as a right $\cD_{X,p}[\dd_tt]$-module.

Let us now begin with the actual proof and denote by $\wt\cI$ the ideal $( t,\beta(\dd_tt),t-h,\delta_1,\ldots,\delta_{n-1},\tilde\chi+1)$ of $V^0\cD$. The inclusion $\supset$ in equation \eqref{eq:LemmaV0}  is trivial, let us show the reverse one. We will use a more suitable set of generators of $\wt\cI$, namely $\{t,\tilde\beta,h,\delta_1,\ldots,\delta_{n-1},\chi+\dd_tt+1\}$, where $\tilde\beta=\beta(-\chi-1)$.

Let then $P\in\wt\cI\cap\cD_{X,p}[\partial_t t]$. Then there exist operators $Q,R,A,B^1,\ldots,B^{n-1},C\in V^0\cD$ such that
\begin{equation}\label{eq:reduction_to_D_X[s]_trick}
P=Qt+R\tilde\beta+Ah+\sum_{i=1}^{n-1}B^i\delta_i+C(\tilde\chi+1).
\end{equation}
Our goal is to show that $Q$ must vanish and that the other operators belong actually to $\cD_{X,p}[\dd_tt]$. Let us write $Q=\sum_kQ_kt^k$, with $Q_k\in\cD_{X,p}[\dd_tt]$, and analogously for $R$, $A$, the $B_i$ and $C$. We can express the right-hand side of equation \eqref{eq:reduction_to_D_X[s]_trick} as a new operator $S=\sum_kS_kt^k$, where again $S_k\in \cD_{X,p}[\partial_t t]$. Moreover, since $t$, $\tilde\beta$, $h$, $\delta_1,\ldots, \delta_{n-1}$ and $\wt\chi+1$, when seen as elements of $\cD_{X,p}[\partial_t t][[t]]$, are homogeneous  in $t$, we have for every $k\geq0$ that
$$S_k=Q_{k-1}+R_k\tilde\beta +A_k h+\sum_{i=1}^{n-1}B_k^i\delta_i+C_k(\tilde\chi-k+1).$$
Comparing the $S_k\in\cD_{X,p}[\dd_tt]$ with the terms of degree $k$ in $t$ at the left-hand side (recall that $P\in \cD_{X,p}[\partial_t t]$), we find that $S_k=0$ for every $k>0$.

Let us write now $\bar R=R-R_0$ and similarly with the other operators. Each of them can be written as a series in $t$ and we have checked that
$$Qt+\bar R\tilde\beta+\bar Ah+\sum_{i=0}^{n-1}\bar B^i\delta_i+\sum_{k>0}C_k(\wt\chi-k+1)t^k=0.$$
Therefore, since $\sum_{k>0}C_k(\wt\chi-k+1)t^k=\bar C(\wt\chi+1)$,
$$P=R_0\tilde\beta+A_0h+\sum_{i=0}^{n-1}B_0^i\delta_i+C_0(\wt\chi+1).$$
Summing up, we have been able to write $P$ as a linear combination of $h$, $\tilde\beta$, the $\delta_i$ and $\wt\chi+1$ with coefficients in $\cD_{X,p}[\dd_tt]$, as we wanted.

\end{proof}

\begin{theorem}\label{theo:reduction_to_D_X[s]}
Let $D\subset X$ be a free divisor, let $p\in D$ such that $D$ is strongly Koszul at $p$
and let $h\in \cO_{X,p}$ be a reduced local equation of $D$. Let $\beta(s)=\prod_{\alpha\in B'_h}(s+\alpha)^{l_\alpha}$, where $l_\alpha$ is the multiplicity of $\alpha-1$ in $b_h(s)$, $N(h)=\cN(h)_{(0,p)}$ and $\pi_p:\cD_{X,p}[s]\ra N(h)$ the morphism of $\cD_{X,p}[s]$-modules given by $P(s)\mapsto [P(-\dd_tt)]$
(where we give $N(h)$ the $\cD_{X,p}[s]$-module structure coming from the inclusion $\cD_{X,p}[s]\hookrightarrow \cD_{\dC_t\times X,(0,p)}$ sending $s$ to $-\partial_t t$). Consider also the ideal $\cJ_p$ of $\cD_{X,p}[s]$ defined as
$$\cJ_p:=\cD_{X,p}[s]( h,\beta(-s),\delta_1,\ldots,\delta_{n-1},\chi-s+1).$$
Then,
\begin{enumerate}
    \item $\cD_{X,p}[s]$ can be endowed with a structure of $V^0\cD_{\dC_t\times X,(0,p)}$-module such that $\pi_p$ becomes $V^0\cD_{\dC_t\times X,(0,p)}$-linear. As a consequence, the terms of the total order filtration $T_\bullet \cD_{X,p}[s]$ (see Definition \ref{defi:SK}) are $\cO_{\dC_t\times X,(0,p)}$-modules.
    \item For any $k\geq 0$, we have the following equality of $\cO_{\dC_t\times X,(0,p)}$-modules:
\begin{equation} \label{eq:main-2}
    \left(V_{\can}^0 \cN(h)\cap F^{\ord}_k \cN(h)\right)_{(0,p)}=\pi_p(\cJ_p\cap T_k\cD_{X,p}[s]),
\end{equation}
and so
\begin{align}
\begin{split}\label{eq:recur-FHNh}
F_k^H \cN(h)_{(0,p)} & =
\partial_t  F^H_{k-1}\cN(h)_{(0,p)}  +  \pi_p(\cJ_p\cap T_{k-1}\cD_{X,p}[s]) \\
&=\D \sum_{i=0}^{k-1} \partial_t^k\left(\pi_p(\cJ_p\cap T_{k-1-i}\cD_{X,p}[s])\right).
\end{split}
\end{align}
\end{enumerate}
\end{theorem}
\begin{proof}
Following the same convention as in Lemma \ref{lem:reduction_to_D_X[s]_trick} above, we will write $\cO$ and $\cD$ for $\cO_{\dC_t\times X,(0,p)}$ and $\cD_{\dC_t\times X,(0,p)}$, respectively.

We will define on $\cD_{X,p}[s]$ an action of $\cO$ by putting $t\cdot P(s):=P(s+1)h$ and extending it by linearity. Let us check that such an action is well-defined. Indeed, let $a(x,t)=\sum_{\alpha,k}a_{\alpha k} x^\alpha t^k$ be a series in $\cO$ and let us show that we can multiply by $a$ in $\cD_{X,p}[s]$. As a first step, we will restrict ourselves to consider an element $P\in\cD_{X,p}$. In order to show that $a\cdot P$ lies within $\cD_{X,p}[s]$, it is equivalent by linearity to show it for $P=\dd^\beta$, with $\beta\in\dN^n$. Then we will have
\begin{align}\begin{split}
a\cdot \dd^\beta&:=\sum_{\alpha,k}a_{\alpha k}x^\alpha \dd^\beta h^k
=\sum_{\alpha,k}a_{\alpha k}\sum_{\substack{\gamma\leq\beta\\\gamma\leq\alpha}}(-1)^{|\gamma|}\binom{\beta}{\gamma} \frac{\alpha!}{(\alpha-\gamma)!}\dd^{\beta-\gamma}x^{\alpha-\gamma}h^k\\
&=\sum_{\gamma\leq\beta}(-1)^{|\gamma|}\binom{\beta}{\gamma}\dd^{\beta-\gamma} \sum_{\alpha\geq\gamma,k}\frac{\alpha!}{(\alpha-\gamma)!}a_{\alpha k}x^{\alpha-\gamma}h^k=
\sum_{\gamma\leq\beta}(-1)^{|\gamma|}\binom{\beta}{\gamma}\dd^{\beta-\gamma}\cdot \dd^\gamma(a)(x,h),
\end{split}\end{align}
where we use the common component wise partial ordering and the standard multi-index notation for the factorial and the binomial numbers.
Therefore, $a\cdot \dd^\beta$ is a finite sum of monomials in the $\dd_i$ times certain convergent series, so it belongs to $\cD_{X,p}\subset\cD_{X,p}[s]$.

Consider now $Ps^j\in\cD_{X,p}[s]$. Then,
\begin{align}\begin{split}
a\cdot Ps^j&:=\sum_{\alpha,k}a_{\alpha k}x^\alpha P(s+k)^jh^k
=\sum_{\alpha,k}a_{\alpha k}x^\alpha P\sum_{r=0}^j\binom{j}{r}s^r k^{j-r}h^k\\ &=\sum_{r=0}^j\binom{j}{r}s^r\sum_{\alpha,k}a_{\alpha k}k^{j-r}x^\alpha Ph^k
=\sum_{r=0}^j\binom{j}{r}s^r\left((t\dd_t)^{j-r}(a)\cdot P\right),
\end{split}\end{align}
which is another finite sum of elements in $\cD_{X,p}[s]$. Thus $a\cdot Ps^j$ belongs clearly to $\cD_{X,p}[s]$, as we wanted to show.

Now the action of $\cO$ on $\cD_{X,p}[s]$ can be extended to an action of $V^0\cD$. Indeed, we have $[t,s]\cdot P(s)=P(s+1)h=t\cdot P(s)$ for any $P(s)\in\cD_{X,p}[s]$, so we can take the action of $\dd_tt$ as that of $-s$. Now using that $V^0\cD\cong\cD_{X,p}[s]\otimes_{\cO_{X,p}}\cO$ as $\cO_{X,p}$-modules, we can extend both actions by linearity to get the desired one of $V^0\cD$.

Let us check now that $\pi_p$ becomes $V^0\cD$-linear with the new structure on $\cD_{X,p}[s]$. Indeed, consider an operator $P(s)\in\cD_{X,p}[s]$ and a series $a=\sum_{\alpha,k}a_{\alpha k}x^\alpha t^k\in\cO$ as before. Then,
\begin{align}\begin{split}\label{eq:Action_t}
a\cdot\pi_p(P(s))&=\sum_{\alpha,k}a_{\alpha k}x^\alpha t^k\cdot[P(-\dd_tt)]=
\left[\sum_{\alpha,k}a_{\alpha k}x^\alpha P(-\dd_tt+k)t^k\right]=
\left[\sum_{\alpha,k}a_{\alpha k}x^\alpha P(-\dd_tt+k)h^k\right]\\
&=\pi_p(a\cdot P(s)).
\end{split}\end{align}
Note that we are replacing the powers of $t$ by those of $h$ in the third equality. To justify this step, let us rewrite $$\sum_{\alpha,k}a_{\alpha k}x^\alpha P(-\dd_tt+k)t^k=\sum_{\beta,m}\dd^\beta (\dd_tt)^m p_{\beta,m}(x,t)=: \overline P,$$
the $p_{\beta,m}$ being convergent functions in $\cO$. Now dividing all such functions by $t-h$ we can write them as $p_{\beta m}(x,t)=q_{\beta m}(x,t)(t-h)+r_{\beta m}(x)$. (This is a very particular instance of the Weierstra{\ss} division theorem for convergent power series whose proof is elementary in this case.) Consequently,
$$
\overline P=\left(\sum_{\beta,m}\dd^\beta (\dd_tt)^m q_{\beta m}(x,t)\right) (t-h)+\sum_{\beta,m}\dd^\beta (\dd_tt)^m r_{\beta m}(x)=:\widehat A\cdot (t-h)+ \overline{P}'.
$$
For every $\beta$ and $m$ we know that $r_{\beta m}(x)=p_{\beta m}(x,h)$, hence, we have an expression
$$
\overline{P}'=\sum_{\alpha,k}a_{\alpha k}x^\alpha P(-\dd_tt+k)h^k
$$
and thus $[\overline{P}]=[\overline{P}']$ in $N(h)$ which shows the third equality in
Formula \eqref{eq:Action_t}.

Let us show the second point of the theorem.
Take an element $\xi\in V_{\can}^0N(h)\cap F_k^{\ord}N(h)$. Since by Proposition \ref{prop:PreciseDescrCanVFilt} we have $V_{\can}^0N(h)=V_{\ind}^1N(h) + \beta(\dd_tt) V_{\ind}^0 N(h)$, we know that there is a representative $\widetilde{P}\in\cD$ of $\xi$ and that there are operators $\wt Q,\wt R\in V^0\cD$ and $\widetilde{A},\widetilde{B}_1,\ldots,\widetilde{B}_{n-1},\widetilde{C}\in\cD$ such that
$$
\widetilde{P}=t\wt Q+\beta(\dd_tt)\wt R+
\underbrace{\widetilde{A}(t-h)+\sum_{i=1}^{n-1}\widetilde{B}_i\delta_i+\widetilde{C}(\widetilde\chi+1)}_{\in I(h)}.
$$
It is clear that $t\wt Q$ can be rewritten as $Q't$ for a suitable $Q'\in V^0\cD$. Regarding $\beta(\dd_tt)\wt R$, let us expand $\wt R$ as $\sum_k\wt R_kt^k$ with $\wt R_k\in V^0\cD$ in an analogous fashion as in the proof of Lemma \ref{lem:reduction_to_D_X[s]_trick}. Then,
\begin{align*}
\beta(\dd_tt)\wt R=&\sum_{k\geq0}\wt R_kt^k\beta(\dd_tt+k)=\\
=&\wt R\beta(\dd_tt)+\sum_{k\geq1}\wt R_kt^k(\beta(\dd_tt+k)-\beta(\dd_tt))=\\
=&\wt R\beta(\dd_tt)+\sum_{k\geq1}\wt R_kt^{k-1}(\beta(\dd_tt+k-1)-\beta(\dd_tt-1))t=:\wt R\beta(\dd_tt)+R't,
\end{align*}
where $R'\in V^0\cD$. Therefore, renaming $Q'+R'$ as $\wt Q\in V^0\cD$, we have a new expression
\begin{equation}\label{eq:ExpressionP}
\widetilde{P}=\wt Qt+\wt R\beta(\dd_tt)+\widetilde{A}(t-h)+\sum_{i=1}^{n-1}\widetilde{B}_i\delta_i+\widetilde{C}(\tilde\chi+1).
\end{equation}

Notice that since
$$
\xi \in V^0_{\can} N(h) \cap F_k^{\ord} N(h) \subset
V^0_{\ind} N(h) \cap F_k^{\ord} N(h),
$$
it follows from Proposition \ref{prop:CompatibilityVandF} that we can pick a new representative $\widehat{P} \in F_k\cD\cap V^0\cD$ of $\xi$, that is, we have
$\widehat{P}-\widetilde{P}\in I(h)$.
Moreover, when writing $\widehat{P}$ as an element
in $\cD_{X,p}[\partial_t t][[t]]$, we can replace any
positive power of $t$ by $h$, and this will not change the order. Hence, there exists an operator $P \in F_k\cD\cap \cD_{X,p}[\partial_t t]$ with
$P-\widehat{P}=\widehat{A}(t-h)$.
It follows that
$P-\widetilde{P}\in I(h)$, or, said otherwise, there
are coefficients $A',B'_1,\ldots,B'_{n-1}, C' \in \cD$ such that
$$
P-\widetilde{P}=A'(t-h)+\sum_{i=1}^{n-1}B'_i\delta_i+C'(\widetilde\chi+1).
$$
Hence (by replacing $\widetilde{P}$ with the expression from equation \eqref{eq:ExpressionP}) we obtain that the class $\xi\in V^0_{\can}N(h)\cap F_k^{\ord}N(h)$ can be represented by an operator $P\in F_k\cD\cap \cD_{X,p}[\partial_t t]$ which has an expression
$$
P=Qt+R\beta(\partial_t t) + A(t-h)
+\sum_{i=1}^{n-1} B_i \delta_i + C(\widetilde{\chi}+1),
$$
where $A:=\widetilde{A}+A'$, $B_i:=\widetilde{B}_i+B'_i$ and $C:=\widetilde{C}+C'$.
By construction, the operators $P$, $Q$ and $R$ are elements in $V^0\cD$. Then we apply Corollary \ref{cor:generators_I(h)_V-filt} and conclude that $A$, $B_1,\ldots,B_{n-1}$ and $C$ belong to $V^0\cD$ as well.

Summing up, we can choose $P$ inside
$$V^0\cD( t,t-h,\beta(\dd_tt),\delta_1,\ldots,\delta_{n-1},\tilde\chi+1)\cap F_k\cD\cap\cD_{X,p}[\dd_tt].$$

Applying Lemma \ref{lem:reduction_to_D_X[s]_trick} above, we obtain that
$$
P\in\cD_{X,p}[\dd_tt]( h,\beta(\dd_tt),\delta_1,\ldots,\delta_{n-1},\tilde\chi+1)\cap T_k\cD_{X,p}[\dd_tt],
$$
and replacing $\dd_tt$ by $-s$ provides the desired claim. For the last part, Formula \eqref{eq:recur-FHNh}, we combine Formula
\eqref{eq:FormulaHodgeNh_Simple} (or its non-recursive version, Formula \eqref{eq:FormulaHodgeNh_NonRec}) with Formula \eqref{eq:main-2}.
\end{proof}
As a first application of Theorem \ref{theo:reduction_to_D_X[s]} and Theorem \ref{thm:HodgeOnMeromorphic}, we can give a formula to calculate the zeroth Hodge ideal $\cI_0(D)_p$.
\begin{corollary}\label{coro:formula_I_0}
Under the assumptions of the previous Proposition, let $(\cJ_0)_p$ be the ideal of $\cO_{X,p}$ defined as
$$(\cJ_0)_p:=\cJ_p\cap \cO_{X,p} = \cD_{X,p}[s]( h,\beta(-s),\delta_1,\ldots,\delta_{n-1},\chi-s+1)\cap\cO_{X,p}.$$
Then
\begin{enumerate}
    \item $F_1^H  N(h)=\{[f]\in  N(h)\,:\,f\in  (\cJ_0)_p\}$.
    \item                     $F_0^H M(h)=\{[f]\in M(h)\,:\,f\in  (\cJ_0)_p\}$.
    \item $ \cI_0(D)_p= (\cJ_0)_p$.
\end{enumerate}
\end{corollary}
\begin{proof}
Point 1 is a consequence of Corollary \ref{cor:PreciseDescrHFilt} and Theorem \ref{theo:reduction_to_D_X[s]} above. For the second point we use formula \eqref{eq:HodgeMhFromHogdeNh}. We know that $F_0^H  M(h)=F_1^H N(h)\cap( M(h)\otimes 1)$, where $ M(h)$ is seen as a $\cO_{X,p}$-submodule of $N(h)$. However, since $ (\cJ_0)_p\subset \cO_{X,p}$, there is actually no proper intersection and the claim follows.

In order to obtain the formula for the zeroth Hodge ideal we just need to recall its definition and the isomorphism between $\cO_{X,p}(*D)$ and $M(h)$ (see point 2 of Proposition \ref{prop:SummarySKLCT}). Since $F_0^H\cO_{X,p}(*D)=\cI_0(D)_p\cdot\cO_{X,p}(D)$, we just need to multiply by $h$ the elements of $(\cJ_0)_p\cdot h^{-1}\subset \cO_{X,p}(*D)$, but again, $(\cJ_0)_p\subset\cO_{X,p}$, so $I_0(D)_p=(\cJ_0)_p$ and we are done.
\end{proof}

\begin{remark}
Notice that the ideal $\cJ_p$ from Theorem \ref{theo:reduction_to_D_X[s]} equals, up to
changing $s$ by $s+1$ the ideal
$$\widehat{\cJ}_p=\cD_{X,p}[s]( h,\widehat\beta(s),\delta_1,\ldots,\delta_{n-1},\chi-s),$$
where $\widehat\beta(s)=\prod_{\alpha\in B_h}(s-\alpha)^{l_\alpha}$, with $B_h$ and $l_\alpha$ being, respectively, the set of roots of $b_h$ in the interval $(-1,0)$ and the multiplicity of each root $\alpha$.
Then we have
$$
\widehat{\cJ}'_p:=\cD_{X,p}[s]( h,\delta_1,\ldots,\delta_{n-1},\chi-s) \subsetneq \widehat{\cJ}_p
$$
and it is known (see, e.g.,
\cite[Theorem 1.24]{calde_nar_lct_ilc} and \cite[\S~4]{nar_symmetry_BS})) that $\widehat{\cJ}'_p$ is the annihilator of the
class $h^s$ in $\cD_{X,p}[s]\cdot h^s/\cD_{X,p}[s]\cdot h^{s+1}$. Hence, by the Bernstein functional equation, we know that
$b_h(s) \in \widehat{\cJ}'_p$.  The polynomial $\widehat{\beta}(s)$, which is a proper factor of $b_h(s)$, is not an element of $\widehat{\cJ}'_p$ in general.
\end{remark}

\medskip

Our next aim is to describe an algorithm to calculate effectively the filtration steps $F_k^H\cO_X(*D)$ resp. the Hodge ideals $I_k(D)$ for an SK-free divisor. The starting point is Formula \eqref{eq:recur-FHNh}.
Recall (see the explanation before Lemma \ref{lem:RootsElement1GraphEmbed}) that
there is a left $\cD_{\dC_t\times X,(0,p)}$-linear isomorphism
$$
\begin{array}{rcl}
\Psi:N(h) & \longrightarrow & M(h)[\partial_t]  \\
    {[Q]}
     & \longmapsto &  \sum_{k=0}^{\ord(Q)} [Q_k]\partial^k_t
\end{array}
$$
which we need to make explicit in order to describe our algorithmic procedure to obtain the modules $F_k^H \cO_X(*D)$.

Let us consider the $\dC$-algebra automorphism $\varphi: \cD_{\dC_t\times X,(0,p)} \to \cD_{\dC_t\times X,(0,p)}$ defined as
$$
\begin{array}{rll}
\varphi(a) = a &\;\;\forall\;a\in\cO_{X,p},&
\varphi(t) = t +h, \\[5pt]
\varphi(\partial_{x_i}) = \partial_{x_i} - h'_{x_i} \partial_t&\;\;\forall\;i=1,\ldots,n,&
\varphi(\partial_t) = \partial_t.
\end{array}
$$
It is clear that  $\varphi(\delta)=\delta-\delta(h)\partial_t$ and $\varphi^{-1}(\delta)=\delta+\delta(h)\partial_t$ for each $\delta\in \Theta_{X,p}$, and $\varphi^{-1}(t)=t-h$.
Now given a class $[Q]\in N(h)$, by division by $t$ we find unique operators $A\in \cD_{\dC_t\times X,(0,p)}$ and $Q'\in \cO_{X,p}\langle \partial_{x_1},\dots, \partial_{x_n}, \partial_t \rangle$ such that
$$
\varphi(Q) = A\, t + Q'\quad (\ord(Q')\leq \ord(Q)).
$$
Then we write
$$
Q' = \sum_{k=0}^{\ord(Q)} Q_k \partial_t^k,\quad Q_k \in \cD_{X,p},
$$
and taking classes in $M(h)$ we obtain the
element $\sum_{k=0}^{\ord(Q)} [Q_k]\partial_t^k=:\Psi([Q]) \in M(h)[\partial_t]$ which is readily verified to be well-defined. Notice that the inverse map $\Psi^{-1}$ sends an element $\sum_{k=0}^d [Q_k] \partial_t^k \in M(h)[\partial_t]$, with $Q_k\in\cD_{X,p}$,   to
$$
\left[\sum_{k=0}^d \varphi^{-1}(Q_k) \partial_t^k
\right] \in N(h).
$$
Let us define $F^{\ord}_\bullet \left(M(h)[\partial_t]\right) := \Psi\left(F^{\ord}_\bullet N(h) \right)$. Then we clearly have
$$
F_k^{\ord} \left(M(h)[\partial_t]\right) = \bigoplus_{i=0}^k (F_{k-i}^{\ord} M(h))\,  \partial_t^i,\quad k\in \dZ.
$$

Similarly, put $F^H_\bullet \left(M(h)[\partial_t]\right) := \Psi\left(F^H_\bullet N(h) \right)$. Formula \eqref{eq:recur-FHNh}
becomes:
\begin{eqnarray}
F_k^H  \left(M(h)[\partial_t]\right)  & =&
\partial_t  F_{k-1}^H  \left(M(h)[\partial_t]\right) +  \Psibar(\cJ_p\cap T_{k-1}\cD_{X,p}[s]) \\ \nonumber \\
\label{eq:FHMhdt}
& =&
 \sum_{i=0}^{k-1}
\partial_t^i\,    \Psibar(\cJ_p\cap T_{k-1-i}\cD_{X,p}[s]).
\end{eqnarray}
where we put $\Psibar = \Psi \circ \pi_p$.

We denote $\Phi:M(h) \xrightarrow{\sim} \cO_{X,p}(* D)$ the isomorphism of left $\cD_{X,p}$-modules given by $\Phi([Q]) = Q(h^{-1})$, for each $Q\in\cD_{X,p}$ (see formula \eqref{eq:CyclicPresent-1}), and by $\Phibar: M(h)[\partial_t] \xrightarrow{\sim} \cO_{X,p}(* D)[\partial_t]$ the induced isomorphism of $\cD_{\dC_t\times X,(0,p)}$-modules:
$$
\Phibar\left( \sum_{k=0}^d [Q_k] \partial_t^k \right) = \sum_{k=0}^d Q_k(h^{-1}) \partial_t^k.
$$
We recall from Formulas \eqref{eq:HodgeNhFromHogdeMh} and \eqref{eq:HodgeMhFromHogdeNh} that the relationship between the Hodge filtrations
 on $\cO_{X,p}(* D)[\partial_t]\cong N(h)\cong (i_{h,+} \cO(*D))_{(0,p)}$ and on $\cO_{X,p}(* D)$ is given by
\begin{eqnarray*}
& \displaystyle
F^H_k \left( \cO_{X,p}(* D)[\partial_t] \right) =
\bigoplus_{i=0}^{k-1} (F_{k-1-i}^H \cO_{X,p}(* D))\,  \partial_t^i,
&
\\
&\displaystyle
F^H_k \cO_{X,p}(* D) = F^H_{k+1} \left( \cO_{X,p}(* D)[\partial_t] \right) \cap \cO_{X,p}(* D)
\end{eqnarray*}
for all $k\in \dZ$. We obviously have $\Phi(F^H_\bullet M(h)) = F^H_\bullet \cO_{X,p}(*D)$ and $\Phibar(F^H_\bullet ( M(h)[\partial_t])) = F^H_\bullet \cO_{X,p}(*D)[\partial_t]$.

The pole order filtration $P_k \cO_{X,p}(*D) = \cO_{X,p}((k+1)D) $, $k\geq 0$, induces a filtration on $\cO_{X,p}(*D)[\partial_t]$:
$$
P_k \left(\cO_{X,p}(*D)[\partial_t]\right) := \bigoplus_{i=0}^k  \cO_{X,p}((k-i+1)D) \partial_t^i,\quad k\geq 0.
$$
With these preliminary remarks, we can now  explain an effective procedure to compute a system of generators of the $\cO_{X,p}$-modules $F_k^H \cO_{X,p}(*D) $, $k\geq 0$:

\begin{description}
    \item[Step 1:]
        We compute an involutive  basis $\cB = \{P_1,\dots,P_N \}$ of $\cJ_p$ with respect to the total order filtration $T$ on $\cD_{X,p}[s]$, i.e. $P_1,\ldots,P_N$ are such that
        $$
        \gr_\bullet^T (\cJ_p) = \gr_\bullet^T(\cD_{X,p}[s])(\sigma^T(P_1),\ldots,\sigma^T(P_N)),
        $$
        where for $P\in\cD_{X,p}[s]$, $\sigma^T(P)$ denotes
        the symbol of $P$         in $\gr_\bullet^T(\cD_{X,p}[s])$.
        Let us write $\ord_T(P_j) = d_j$.
        Consequently, for each $k\geq 0$, a system of generators of the $\cO_{X,p}$-module $\cJ_p\cap T_k\cD_{X,p}[s]$ is given by
        $$
        \left\{s^l \upartial_x^\alpha P_j\ |\ l + |\alpha| + d_j \leq k
        \right\}.
        $$
    \item[Step 2:]
        After equation \eqref{eq:FHMhdt}, a system of generators of the $\cO_{X,p}$-module $F_{k+1}^H \left(M(h)[\partial_t]\right)$ is
        $$
        \left\{ \partial_t^i\, \Psibar(s^l \upartial_x^\alpha P_j)\ |\ l + |\alpha| + d_j + i\leq k
        \right\}.
        $$
    \item[Step 3:]
        By means of the isomorphism $\Phibar$, we obtain a system of generators of the $\cO_{X,p}$-module $F_{k+1}^H \left(\cO_{X,p}(*D)[\partial_t]\right)$, that we call $\xi_e$, $e=1,\dots,R$. Since $F_{k+1}^H \left(\cO_{X,p}(*D)[\partial_t]\right)$ is a submodule of the rank $(k+1)$ free $\cO_{X,p}$-module $P_k \left(\cO_{X,p}(*D)[\partial_t]\right)$, we have
        $$
        \xi_e = \sum_{i=0}^k \xi_{e,i} \partial_t^i  \in \cO_{X,p}((k+1)D) \oplus \cO_{X,p}(kD) \partial_t \oplus \cdots \oplus \cO_{X,p}(D) \partial_t^k, \ e=1,\dots,R.
        $$
        Finally, a syzygy computation allows us to compute a system of generators of the $\cO_{X,p}$-module
        $$
        \cO_{X,p}\langle\xi_1,\dots,\xi_R\rangle \cap \cO_{X,p}((k+1)D) = F_{k+1}^H \left(\cO_{X,p}(*D)[\partial_t]\right) \cap \cO_{X,p}(*D) =
        F_k^H \cO_{X,p}(*D).
        $$
\end{description}

In the following subsections we apply this algorithm to calculate the Hodge ideals and the generating level for some interesting classes of examples. Since it is known (\cite[Proposition 10.1]{PopaMustata}) that $\cI_0(D)=\cJ((1-\varepsilon)D)$, and since an algorithm for the calculation of multiplier ideals exists already (see \cite{BerkeschLeykin}), our approach first of all provides an alternative way to compute these multiplier ideals. In some cases our methods yield results whereas the algorithm in loc. cit. does not terminate. Moreover, we can effectively calculate higher Hodge ideals, and we indicate them below in some of the cases for which the generating level of the Hodge filtration $F^H_\bullet \cO_X(*D)$ is positive.

We also determine the generating level of the Hodge filtration, always confirming Conjecture \ref{con:GenLevel}. In many cases, the Hodge filtration is even generated at level zero. This is not always true, however, and we show in particular that
for linear free divisors, the conjectured bound from \ref{con:GenLevel} is sharp. Examples with low generating level are interesting because of the local vanishing conjecture (see point 3 of Remark \ref{rem:ConseqGenLev} as well as \cite[Theorem 17.1]{PopaMustata}). Notice that
only  plane curves, normal crossing divisors (cf. the introduction of \cite{PopaMustata}) and surfaces with rational singularities (\cite[Corollary B]{MP-Inv}) were previously known to have  generating level equal to zero.

\subsection{Divisors with normal crossings}
\label{subsec:NCD}
If $X$ is any complex manifold and $D\subset X$ is a reduced divisor with only normal crossings, then it is elementary to check that $D$ is free and strongly Koszul at each of its points. As we have already pointed out in Remark \ref{rem:HodgeOrderPoleOrder},
it easily follows from the fact that $-1$ is the only root of $b_D(s)$ that
we have
$$
F_\bullet^H \cO_X(*D) = F_\bullet^{\ord} \cO_X(*D).
$$
Consequently, since obviously the order filtration is generated at level $0$, so is
the Hodge filtration (a fact that is of course well known from Saito's theory).
For the Hodge ideals themselves, we obtain that $\cI_k(D)\cdot F_k^{\ord} \cO_X(*D) =
P_k \cO_X(*D)$, and an easy local calculation shows that this coincides with
\cite[Proposition 8.2]{PopaMustata}.

Summarizing, if the divisor $D$ has only normal crossings, then our results easily imply the known facts
about the Hodge filtration resp. the Hodge ideals for such divisors.

\subsection{Surfaces}

We consider two kinds of free surfaces in $\dC^3$, namely (two-dimensional hyper-)plane arrangements and Sekiguchi's free divisors (\cite{Sekiguchi}).
Let $D$ be a central hyperplane arrangement $\bigcup_{i=1}^k H_i\subset X=\dC^n$. It is a longstanding question to detect when an arrangement is a free divisor; for a recent account, see \cite[Chapter 8]{DimcaHypArrangementsBook}. However, if $D$ is free, then it is strongly Koszul since it is locally quasi-homogeneous (see, e.g., \cite{calde_nar_compo} for this implication), so that the methods from this paper do apply.

Below are some results for low-dimensional free arrangements:

$$
\begin{array}{c|c|c}
\textbf{\textup{Name}} & \textbf{\textup{Equation}} & \mathbf{\cI_0(D)}  \\ \hline
A_2 & (x-y)(x-z)(y-z) & (y-z,x-z) \\ \hline
D_3 & (x^2-y^2)(x^2-z^2)(y^2-z^2) &
(y^2z-z^3,x^2z-z^3,y^3-yz^2,xy^2-xz^2,x^2y-yz^2,x^3-xz^2)
\end{array}
$$

In both cases the Hodge filtration on $\cO_X(*D)$ is generated at level zero. Of course, these results do not depend on
the fact that $\dim(D)=2$, similar calculations are possible for other arrangements,
but the ideal $\cI_0(D)$ becomes difficult to print.

Let us consider now the examples of divisors in $\dC^3$ that are discussed in the paper \cite{Sekiguchi}. We refer to loc. cit. for details on their construction and only produce the results of the calculation of $\cI_0(D)$ here. All these examples are given by weighted homogeneous equations in three variables, and we write $(d;d_x,d_y,d_z)$ for a weight vector, where $d_x$, $d_y$ and $d_z$ are the weights of $x$, $y$ and $z$ and $d$ is the degree of the defining equation for the divisor $D\subset \dC^3$.
{\small $$
\begin{array}{c|c|l}
\textbf{\textup{\normalsize Name}} & \textbf{\textup{\normalsize Weight vector}} & \textbf{\textup{\normalsize Equation}}/\text{\normalsize $\mathbf{\cI_0(D)}$}   \\ \hline
A_1 & (12;2,3,4) & \begin{array}{rcl} \\ h&=&16x^4z-4x^3y^2-128x^2z^2+144xy^2z\\&&-27y^4+256z^3\\ \\ \mathbf{\cI_0(D)}&:&(9y^2-32xz,x^2+12z)\\ { }\end{array}\\ \hline
A_2 & (12;2,3,4)& \begin{array}{rcl} \\ h &=&2x^6-3x^4z+18x^3y^2-18xy^2z+27y^4+z^3\\ \\ \mathbf{\cI_0(D)}&:& (y^2,x^2-z)\\ { }\end{array} \\ \hline
B_1 & (9;1,2,3) & \begin{array}{rcl} \\ h&=&z(x^2y^2-4y^3-4x^3z+18xyz-27z^2) \\ \\ \mathbf{\cI_0(D)}&:& (y^2-3xz,xy-9z,x^2z-3yz)\\ { }\end{array}
\\ \hline
B_3 & (9;1,2,3) & \begin{array}{rcl} \\ h&=&z(-2y^3+9xyz+45z^2) \\ \\ \mathbf{\cI_0(D)}&:& (z,y^2)\\ { }\end{array}
\\ \hline
H_2 & (15;1,3,5) & \begin{array}{rcl} \\ h&=&100x^3y^4+y^5+40x^4y^2z-10xy^3z+4x^5z^2\\&&-15x^2yz^2+z^3 \\ \\ \mathbf{\cI_0(D)}&:& (y^2-xz,12x^2yz-z^2,12x^3z-yz)\\ { }\end{array}
\\ \hline
H_5 & (15;1,3,5) & \begin{array}{rcl} \\ h&=&x^3y^4-y^5+3xy^3z+z^3 \\ \\ \mathbf{\cI_0(D)}&:& (y^2,yz,z^2)\\ { }\end{array}
\end{array}
$$}

We have managed to perform the analogous calculations for all of Sekiguchi's examples but $H_1$. The Hodge filtration on $\cO_{\dC^3}(*D)$ is generated at level $0$ for all of them.

\subsection{Linear free divisors}
\label{subsec:LFD}

The paper \cite{BM} introduced a large class of examples of free divisors which appear as discriminants in pre-homogenous vector spaces. Since the module $\Theta(-\log\, D)$ has a basis of linear (in the global coordinates) vector fields in these examples, the divisor $D$ is called \emph{linear free}. A rich source of linear free divisors comes from representation of quivers. It can be shown that the strong Koszul assumption is satisfied in these cases if the underlying graph of the quiver is of ADE-type. Quivers of type A yield free divisors with normal crossings, so the first non-trivial example is the discriminant in the representation space of the $D_4$-quiver. Here we have
$$
D\subset \textup{Mat}(2\times 3,\dC)=
\left\{
\begin{pmatrix}
a_{11} & a_{12} & a_{13} \\
a_{21} & a_{22} & a_{23}
\end{pmatrix}\,|\, a_{ij}\in \dC
\right\},
$$
where $D=V(h)$ and $h=\Delta_1\cdot\Delta_2\cdot\Delta_3$, with $\Delta_1$, $\Delta_2$ and $\Delta_3$ being the three maximal minors of
an element of $\textup{Mat}(2\times 3, \dC)$ (hence, $h$ is a homogenous equation of degree $6$). In this case,
we have $b_h(s)=(s+1)^4(s+2/3)(s+4/3)$ and hence we consider
$$
\cJ=( h,s-1/3,\delta_1,\ldots,\delta_5,\chi-s+1)\subset \cD_{\dC^6}[s].
$$
The intersection with $\cO_{\dC^6}$ (which can equivalently be calculated as $\cD_{\dC^6}
( h,\delta_1,\ldots,\delta_5,\chi+2/3)
\cap \cO_{\dC^6}$) is
$$
\cI_0(D)=(a_{13}a_{22}-a_{12}a_{23},a_{13}a_{21}-a_{11}a_{23},a_{12}a_{21}-a_{11}a_{22}).
$$

The Hodge filtration on $\cO_X(\*D)$ is generated at level 1, and the first Hodge ideal is minimally generated by 13 polynomials in the variables $a_{ij}$ of degree 7 at most (for typsetting reasons, we refrain to print it here).

This discriminant is a first example of the series where the underlying graph
of the quiver is of $D_n$-type. The next example is the $D_5$-quiver, where the discriminant in the $10$-dimensional representation space yields a hypersurface of
degree $10$. A defining equation $h\in \dC[a,b,c,k,e,f,g,h,i,j]$ for $D$ is given by

{\scriptsize \begin{align*}
h&=a^2ke^3ghi^2-a^2ce^2fghi^2+2abke^2fghi^2-2abcef^2ghi^2+b^2kef^2ghi^2-b^2cf^3ghi^2-a^2cke^2h^2i^2-abk^2e^2h^2i^2+a^2c^2efh^2i^2\\
&-b^2k^2efh^2i^2+abc^2f^2h^2i^2+b^2ckf^2h^2i^2-a^2ke^3g^2ij+a^2ce^2fg^2ij-2abke^2fg^2ij+2abcef^2g^2ij-b^2kef^2g^2ij+b^2cf^3g^2ij\\
&+a^2c^2keh^2ij+2abck^2eh^2ij+b^2k^3eh^2ij-a^2c^3fh^2ij-2abc^2kfh^2ij-b^2ck^2fh^2ij+a^2cke^2g^2j^2+abk^2e^2g^2j^2-a^2c^2efg^2j^2\\
&+b^2k^2efg^2j^2-abc^2f^2g^2j^2-b^2ckf^2g^2j^2-a^2c^2keghj^2-2abck^2eghj^2-b^2k^3eghj^2+a^2c^3fghj^2+2abc^2kfghj^2+b^2ck^2fghj^2
\end{align*}}

and we obtain
$$
\cI_0(D)=(aei+bfi-acj-bkj, aeg+bfg-acr-bkr, keri-cfri-kegj+cfgj).
$$
Notice that for this example, the computation of the multiplier ideal
$\cJ((1-\varepsilon)D)$ with the methods from \cite{BerkeschLeykin} does not terminate. The generating level of the Hodge filtration on $\cO_X(*D)$ is two here. The first Hodge ideal $\cI_1(D)$ has a minimal generating set consisting of $24$ polynomials of degree $13$ at most, whereas $\cI_2(D)$ is minimally generated by $124$ polynomials, their maximal degree being $22$.

In fact, Conjecture \ref{con:GenLevel} would yield the following estimate for the generating level of the full $D_n$-series.
\begin{conjecture}
Let $D\subset X=\dC^{4n-10}$ be the linear free divisor $D_n$. Then, the Hodge filtration on $\cO_X(*D)$ is generated at level $n-3$.
\end{conjecture}
\begin{proof}[Proof assuming Conjecture \ref{con:GenLevel}]
The dimension of the divisors $D_n$ can be found in \cite[p.~1347]{dGMS}. The result then is just a consequence of applying the statement of Conjecture \ref{con:GenLevel}, once we know a general formula for the Bernstein-Sato polynomials of the divisors $D_n$. It can be found at \cite[Table 4.1]{Sev1} that they are, respectively,
$$\left(s+\frac{4}{3}\right)^{n-3}(s+1)^{2n-4}\left(s+\frac{2}{3}\right)^{n-3}.$$
\end{proof}

Note that this conjecture, if true, would give a bound for the generating level which is exactly a quarter of the general one given in \cite[Theorem B]{PopaMustata}. We also believe that such bounds are attained for all of the $D_n$ (i.e. that the generating level is exactly $n-3$), as we have checked for $n=4,5$.

Another typical example of a linear free divisor in low dimensions is the  discriminant in the space of binary cubics (see, e.g., \cite[Examples 1.4. (2)]{GMNS}). Here we have
$D=\cV(h)$, where $h=-y^2z^2+4xz^3+4y^3w-18xyzw+27x^2w^2\in \dC[x,y,z,w]$, and we obtain that
$$
I_0(D)=(z^2-3yw,yz-9xw,y^2-3xz).$$
In this case the Hodge filtration on $\cO_{\dC^4}(*D)$ is generated in level one, with first Hodge ideal
\begin{align*}
\cI_1(D)=&(y^2z^2-4xz^3-4y^3w+18xyzw-27x^2w^2,2z^5-15yz^3w+27y^2zw^2+27xz^2w^2-81xyw^3,\\
&yz^4-27xz^3w-18y^3w^2+135xyzw^2-243x^2w^3, 2xz^4-y^3zw-9xyz^2w+27xy^2w^2-27x^2zw^2,\\ &xyz^3-2y^4w+9xy^2zw-27x^2z^2w+27x^2yw^2, y^4z-18x^2z^3-27xy^3w+135x^2yzw-243x^3w^2,\\ &2y^5-15xy^3z+27x^2yz^2+27x^2y^2w-81x^3zw).
\end{align*}

Notice that similar computations are possible up to some extent for other linear free divisors of SK-type, such as those given by the $E_6$- and $E_7$-quivers. Under the assumption of Conjecture \ref{con:GenLevel}, one could give generating level bounds, using the calculation of Bernstein polynomials from \cite{ThesisBarz}.

We finish this subsection with some examples of non-reductive linear free divisors in low dimension, as found in \cite[Example 5.1, Table 6.1]{GMNS}, that we summarize in the table below.

$$
\begin{array}{c|c|c}
\textbf{\textup{Name}} & \textbf{\textup{Equation}} & \mathbf{\cI_0(D)}  \\ \hline
 \dim 3 & (y^2+xz)z & (y,z) \\ \hline
 \dim 4\text{, case }1 & (y^2+xz)zw & (y,z) \\ \hline
 \dim 4\text{, case }2 & (yz+xw)zw & (z,w) \\ \hline
 \dim 4\text{, case }3 & x(y^3-3xyz+3x^2w) & (x,y^2) \\ \hline
 \operatorname{Sym}_3(\dC) & x\begin{vmatrix}x&y\\y&u\end{vmatrix}\begin{vmatrix}x&y&z\\y&u&v\\z&v&w\end{vmatrix} & (xv-yz,xu-y^2,xzu-xyv)
\end{array}
$$

In all those cases, the Hodge filtration on the corresponding ring of meromorphic functions is generated at level 0.

\subsection{The Whitney umbrella and the cross caps}
\label{subsec:Whitney}

We finish this section by mentioning two examples that slightly fall out of the general setup of this paper. Namely, consider the Whitney umbrella
$D=\cV(h)=\cV(x^2-y^2z)\subset \dC^3$. This is not a free divisor, indeed, we have $\Theta(-\log\,D)=\sum_{i=1}^4\cO_{\dC^3} \delta_i$ where
$$
\delta_1=y\partial_y-2z\partial_z,\delta_2=
-yz\partial_x-x\partial_y, \delta_3=
 -y^2\partial_x-2x\partial_z,\delta_4=\chi=
(1/2)x\partial_x+(1/3)y\partial_y+(1/3)z\partial_z,
$$
so that $\Theta(-\log\,D)$ is not $\cO_{\dC^3}$-locally free. However, the isomorphism
$$
\cD_{\dC^3}/(\delta_1,\delta_2,\delta_3,\chi+1)\cong \cO_{\dC^3}(*D)
$$
of left $\cD_{\dC^3}$-modules still holds true, and we find that $b_h(s)=(s+1)^2(s+3/2)$. According to Lemma \ref{lem:RootsElement1GraphEmbed} from above, we have
$b_{V_{\ind}^\bullet}^{N(h)}=s^2(s-1/2)$. One checks that the main results of this paper can also be applied to this example. In particular, it follows that the integer $r$ appearing in Lemma
\ref{lem:1InFr} is zero, and hence we deduce from Corollary \ref{cor:Inclu} that
$$
F^H_\bullet \cO_{\dC^3}(*D) \cong F^{\ord}_\bullet\cO_{\dC^3}(*D).
$$
In particular, in this example the generating level of the Hodge filtration is zero since $F^{\ord}_\bullet \cO_{\dC^3}(*D)$ is generated at level $0$.

A similar reasoning applies to the higher dimensional example
$$
D=\cV(h):=\cV(x_1x_2x_3^2x_4+x_3^3x_4^2-x_1^2x_3^2x_5+x_2^3x_4-x_1x_2^2x_5+3x_2x_3x_4x_5-2x_1x_3x_5^2-x_5^3)\subset \dC^5,
$$
called cross-cap, which again is not free (the module of logarithmic vector fields has $9$ generators), here we have
$$
b_h(s)=(s+1)^3(s+3/2)(s+4/3)(s+5/3).
$$
As a conclusion, we obtain that for both the Whitney umbrella and the cross-cap we have $\cI_0(D)=\cO_X$, but $\cI_k(D) \subsetneq \cO_X$ for all $k>1$.

Actually, these two examples are the first two of a whole series, which is discussed in \cite{houston2009vector}. However, it is unclear at this point whether we always have the equality
$F^H_\bullet \cO_{\dC^n}(*D) \cong F^{\ord}_\bullet\cO_{\dC^n}(*D)$ since
this needs the fact that the roots of $b_h(s)$ are contained
in $(-2,-1]$, and a general formula for the Bernstein polynomial for the elements in this series is not known (actually, it seems computational impossible to obtain $b_h(s)$ even for the next example, which is a divisor in $\dC^7$).

\bibliographystyle{amsalpha}
\newcommand{\etalchar}[1]{$^{#1}$}
\def\cprime{$'$}
\providecommand{\bysame}{\leavevmode\hbox to3em{\hrulefill}\thinspace}
\providecommand{\MR}{\relax\ifhmode\unskip\space\fi MR }
\providecommand{\MRhref}[2]{  \href{http://www.ams.org/mathscinet-getitem?mr=#1}{#2}
}
\providecommand{\href}[2]{#2}

\vspace{1cm}

Alberto Casta\~no Dom\'{\i}nguez  \\
Departamento de \'{A}lgebra\\
Facultad de Matem\'{a}ticas\\
Universidad de Sevilla\\
41012 Sevilla\\
Spain

albertocd@us.es

\vspace{1cm}

\nd
Luis Narv\'{a}ez Macarro \\
Departamento de \'{A}lgebra \& Instituto de Matem\'{a}ticas (IMUS)\\
Facultad de Matem\'{a}ticas\\
Universidad de Sevilla\\
41012 Sevilla\\
Spain

narvaez@us.es
\vspace{1cm}

\nd
Christian Sevenheck\\
Fakult\"at f\"ur Mathematik\\
Technische Universit\"at Chemnitz\\
09107 Chemnitz\\
Germany\\
christian.sevenheck@mathematik.tu-chemnitz.de

\end{document}